\documentclass{./Article-eng}
\DeclareMathOperator{\ff}{f} 
\DeclareMathOperator{\rat}{rat}
\DeclareMathOperator{\Sing}{Sing}
\title{On the Iitaka Conjecture $C_{n,m}$ for Kähler Fibre Spaces}
\author{Juanyong WANG}
\date{}

\begin{document}
\maketitle

\begin{abstract}
By applying the positivity theorem of direct images and a pluricanonical version of the structure theorem on the cohomology jumping loci à la Green-Lazarsfeld-Simpson, we show that the klt Kähler version of the Iitaka conjecture $C_{n,m}$ (Ueno, 1975) for $f:X\to Y$ (surjective morphism between compact Kähler manifolds with connected general fibre) holds true when the determinant of the direct image of some power of the relative canonical bundle is big on $Y$ or when $Y$ is a complex torus. These generalize the corresponding results of Viehweg (1983) and of Cao-P\u aun (2017) respectively. We further generalize the later case to the geometric orbifold setting, i.e. prove that $C_{n,m}^{\orbifold}$ (Campana, 2004) holds when $Y$ is a complex torus. 
\end{abstract}

\tableofcontents

\section*{\hypertarget{sec_intro}{Introduction}}
\label{sec_intro}
\addcontentsline{toc}{section}{Introduction}

The Iitaka conjecture $C_{n,m}$\,, in its original form, predicts the subadditivity of the Kodaira dimension for algebraic fibre spaces (c.f. \cite[\S 11.5, Conjecture $C_n$\,, pp.~132-133]{Uen75}); more precisely, for $f:X\to Y$ a surjective morphism between normal projective varieties whose general fibre $F$ is connected, $C_{n,m}$ predicts that
\[
\kappa(X)\geqslant\kappa(F)+\kappa(Y).
\]
This conjecture is intimately related to the study of birational classification of complex algebraic varieties (the Minimal Model Program, abbr. MMP). According to the philosophy of MMP, $C_{n,m}$ is naturally generalized to the log version, usually called $C_{n,m}^{\log}$\,; Moreover, Frédéric Campana further generalize $C_{n,m}$ to the setting of geometric orbifolds, called $C_{n,m}^{\orbifold}$\,, which is formulated in \cite[Conjecture 4.1]{Cam04} and in \cite[Conjecture 6.1]{Cam09}. In addition, by taking into consideration the variation of the fibre space, Eckart Viehweg also propose a stronger version of the $C_{n,m}$\,, called $C_{n,m}^+$.

As shown in \cite{KMM87} (resp. \cite{Kaw85}), $C_{n,m}$ (resp. $C_{n,m}^+$) can be regarded as the consequence of the famous Minimal Model Conjecture and the Abundance Conjecture; moreover, in virtue of the subadditivity of Nakayama's numerical dimensions (c.f. \cite[\S V.4.a, 4.1.Theorem(1), pp.~220-221]{Nak04}), $C_{n,m}^{\text{log}}$ follows from the so-called generalized Abundance Conjecture (for $\QQ$-divisors), c.f. \cite[Remark 1.8]{Fuj17}. 

Although initially stated for projective varieties, $C_{n,m}$\,, as well as the MMP and the Abundance, are considered as still hold for complex varieties in the Fujiki class $\mathcal{C}$ (c.f. \cite{Fuj78, Cam04, HP16, CHP16, Fuj17}); nevertheless they do not hold true in general for non-Kähler complex varieties, c.f. \cite[Remark 15.3, p.~187]{Uen75} 
for an counterexample. The objective of this article is to prove the klt Kähler version of $C_{n,m}^{\log}$ in two important special cases and further generalize the second one to the geometric orbifold setting. Let us remark that since the Kodaira dimension as well as the klt/lc property is invariant under taking log-resolutions, hence for simplicity we will state our main results always for Kähler manifolds, but one can easily see that it remains true for normal complex varieties in the Fujiki class $\mathcal{C}$. Now let us state our main theorem: 

\begin{mainthm}[Main Theorem]
\label{main-thm} 
Let $f: X\to Y$ be a surjective morphism between compact Kähler manifolds such that its general fibre $F$ is connected. And let $\Delta$ be an effective $\QQ$-divisor on $X$ such that $(X,\Delta)$ is Kawamata log terminal (abbr. klt). Suppose that one of following conditions is verified:
\begin{itemize}
\item[\rm(I)]\label{main-thm_I} there is an integer $m>0$ such that $m\Delta$ is an integral divisor and that the determinant line bundle $\det\!f_\ast(K_{X/Y}\ptensor[m]\otimes\scrO_X(m\Delta))$ is big on $Y$; 
\item[\rm(I\!I)]\label{main-thm_II} $Y$ is a complex torus. 
\end{itemize}
Then 
\[
\kappa(X,K_X+\Delta)\geqslant\kappa(F,K_F+\Delta_F)+\kappa(Y),
\]
where $\Delta_F:=\Delta|_F$.
\end{mainthm}
{\hyperref[main-thm_I]{Part (I)}} of the {\hyperref[main-thm]{Main Theorem}} generalizes {\cite[Theorem I\!I]{Vie83}}, which is intimately related to $C_{n,m}^+$ (c.f. \cite{Vie83} for more details; this article, however, will not pursue in this direction); while {\hyperref[main-thm_II]{Part (I\!I)}} generalizes {\cite[Theorem 1.1]{CP17}} and it will be further generalized to the setting of geometric orbifolds, in other word, we will prove $C_{n,m}^{\orbifold}$ for $f$ when $Y$ is a complex torus. 
Moreover, by following the same strategy of the proof of {\hyperref[main-thm_I]{Part (I)}}, we recover the result that klt Kähler version of $C_{n,m}^{\log}$ holds for $f:(X,\Delta)\to Y$ when $Y$ is of general type, which generalizes {\cite[Theorem 3]{Kaw81}}; we also further generalize this result to the geometric orbifold setting. Let us remark that the general (log canonical) version of $C_{n,m}^{\orbifold}$ for $Y$ general type (in the orbifold sense) has already been proved in \cite{Cam04}; the proof is based on a weak positivity result for direct images of twisted pluricanonical bundles, for which \cite{Cam04} only proves the projective case, and gives some hints for the Kähler case; it is established in this generality in \cite{Fuj17}.

Now let us explain the strategy of the proof of the {\hyperref[main-thm]{Main Theorem}}. Generally speaking, as in the mainstream of works on $C_{n,m}$ (among others, \cite{Fuj78,Kaw81,Kaw82,Vie83,CP17,Fuj17}), our proof is based on the positivity of relative pluricanonical bundles and of their direct images. Before we get into details let us first recall some definitions: a surjective proper morphism between complex varieties is called an {\em analytic fibre space}; an analytic fibre space $f:X\to Y$ is called a {\em Kähler fibre space} if locally over $Y$, $X$ is a Kähler variety (c.f. \cite[Definition 2.2]{HP16}).



The key ingredient of the proof of {\hyperref[main-thm_I]{Part (I)}} of the {\hyperref[main-thm]{Main Theorem}} is the positivity of the relative $m$-Bergman kernel metric for Kähler fibre spaces, which is proved by Junyan Cao in \cite{Cao17} by applying the Ohsawa-Takegoshi extension theorem with optimal estimation for Kähler fibre spaces (c.f. {\hyperref[thm_Cao_OT]{Theorem \ref*{thm_Cao_OT}}}) also obtained in \cite{Cao17}, and states as follows (c.f. {\hyperref[thm_Cao_Bergman]{Theorem \ref*{thm_Cao_Bergman}}}): 
\begin{itemize}
\item[]Let $f: X\to Y$ be a Kähler fibre space between complex manifolds and let $(L,h_L)$ be holomorphic line bundle on $X$ endowed with a singular Hermitian metric whose curvature current is positive. Suppose that on the general fibre of $f$ there exists a section of $K_{X/Y}\ptensor[m]\otimes L$  satisfying the $L^{2/m}$-integrability condition for some $m$, then the $m$-Bergman kernel metric $h_{X/Y\!,L}^{(m)}$ on $K_{X/Y}\ptensor[m]\otimes L$ has positive curvature current. 
\end{itemize}
With the help of this positivity result, {\hyperref[main-thm_I]{Part (I)}} of our {\hyperref[main-thm]{Main Theorem}}, as well as the klt Kähler version of $C_{n,m}^{\log}$ for general type bases can both be deduced from (a global version of) the Ohsawa-Takegoshi type extension {\hyperref[thm_Deng_OT]{Theorem \ref*{thm_Deng_OT}}} as follows: 
\begin{itemize}
\item First by the useful {\hyperref[lemme_kod-eff+pull-ample]{Lemma \ref*{lemme_kod-eff+pull-ample}}}, we can reduce the proof of the addition formula to that of the non-vanishing of the (twisted) relative pluricanonical bundle, up to adding an ample line bundle from the base.
\item If $Y$ is of general type in the orbifold sense, the non-vanishing result mentioned above follows easily from the Ohsawa-Takegoshi type extension {\hyperref[thm_Deng_OT]{Theorem \ref*{thm_Deng_OT}}} in contrast to the proof in \cite{Vie83,Cam04,Fuj17}, where such non-vanishing results are deduced from the weak positivity of the direct images. Let us remark that: by generalizing the weak positivity theorem for $f$ Kähler fibre space and for $\Delta$ log canonical, the general (log canonical) version is proved in \cite{Cam04,Fuj17}. 
\item In the situation of {\hyperref[main-thm_I]{Part (I)}} of our {\hyperref[main-thm]{Main Theorem}}, the proof of this non-vanishing result follows the same strategy, but requires an extra effort to establish a comparison theorem between the determinant of the direct image and the canonical bundle of $X$, see {\hyperref[thm_CP_canonique:det-im-dir]{Theorem \ref*{thm_CP_canonique:det-im-dir}}}, which is a Kähler version of \cite[Theorem 3.13]{CP17}. 
\end{itemize}
The analytic proof given above does not explicitly involve any positivity result of direct images 
while it has the drawback of not being able to tackle the log canonical case.  

Now we turn to the proof of {\hyperref[main-thm_II]{Part (I\!I)}} of our {\hyperref[main-thm]{Main Theorem}}, for which we follow step by step the same argument in \cite{CP17}. Our proof is based on the positivity of the canonical $L^2$ metric on direct images sheaves (c.f. {\hyperref[thm_pos_im-dir]{Theorem \ref*{thm_pos_im-dir}}}): 
\begin{itemize}
\item[] Let $f: X\to Y$ be a Kähler fibre space between complex manifolds and let $(L,h_L)$ be a holomorphic line bundle on $X$ endowed with a semi-positively curved singular Hermitian metric. Then the canonical $L^2$ metric $g_{X/Y\!,L}$ on the direct image sheaf $f_\ast\left(K_{X/Y}\otimes L\otimes\scrJ(h_L)\right)$ is a semi-positively curved singular Hermitian metric which satisfies the $L^2$ extension property. 
\end{itemize}
The main strategy for the proof of the above positivity result is already implicitly comprised in \cite{HPS18}, and the result is explicitly shown in \cite{DWZZ18} by proving a more general positivity theorem for singular Finsler metrics on direct images. In fact, this result is a consequence of the Ohsawa-Takegoshi extension theorem with optimal estimations obtained in \cite{Cao17}; the new feature is the $L^2$ extension property, which generalizes the well-known property of $\scrO$ that a $L^2$ holomorphic function extends across any analytic subset (compare this with the "minimal extension property" in \cite[Definition 20.1]{HPS18}). By combining the above positivity result of the canonical $L^2$ metric on direct images with the positivity of the relative $m$-Bergman kernel metric and by using the explicit construction of the $m$-Bergman kernel metric to get rid of the multiplier ideal (as in \cite[\S 4, p.367]{CP17}), We obtain the following positivity theorem for direct images of twisted pluricanonical bundles, which serves as a key ingredient of the proof of the {\hyperref[main-thm_II]{Main Theorem, Part (I\!I)}}:
\begin{mainthm}
\label{thm_pos-im-pluri-can}
Let $f:X\to Y$ a Kähler fibre space with $X$ and $Y$ complex manifolds. Let $\Delta$ be an effective $\QQ$-divisor on $X$ such that the pair $(X,\Delta)$ is klt. For any integer $m>0$ such that $m\Delta$ is an integral divisor,
the torsion free sheaf  
\[
\scrF_{m,\Delta}:=f_\ast\left(K_{X/Y}\ptensor[m]\otimes\scrO_X(m\Delta)\right)
\]
admits a canonical semi-positively curved singular Hermitian metric $g_{X/Y\!,\Delta}^{(m)}$ which satisfies the $L^2$ extension property.  
\end{mainthm}
Historically, the study of the positivity of direct images of (twisted) (pluri)canonical bundle(s) is initiated by the works of Griffiths on the variation of Hodge structures in the 60s, and is pursued by Fujita in \cite{Fuj78} and by Kawamata in \cite{Kaw81}; afterwards the study splits into two (related and complementary) main streams: the Hodge-theoretical aspect is further developed by Viehweg in the framework of weak positivity by algebro-geometric methods, while the curvature aspect is exploited by Berndtsson, P\u aun and Takayama (among others) by complex-analytic methods and by introducing the notion of (semi-positively curved) singular Hermitian metrics. The results mentioned above follow the philosophy of the later stream. 
Let us remark that for a torsion free sheaf on a (quasi-)projective variety, the existence of a semi-positively curved singular Hermitian metric implies the weak positivity, while the reciprocal implication is not known yet (it is in fact a singular version of Griffiths's conjecture). The advantage to have a such metric is that: in case that the determinant line bundle is trivial, one can further deduce, by using the $L^2$ extension property, that this torsion free sheaf is a Hermitian flat vector bundle (c.f. {\hyperref[thm_det-triv_herm-plat]{Theorem \ref*{thm_det-triv_herm-plat}}}), in which way we obtain a stronger regularity and our proof of the {\hyperref[main-thm]{Main Theorem}}, like \cite{CP17}, leans on this regularity.  

As a corollary of {\hyperref[thm_pos-im-pluri-can]{Theorem \ref*{thm_pos-im-pluri-can}}}, one finds that the induced metric $\det\!g_{X/Y\!,\Delta}^{(m)}$ on the determinant bundle $\det\!\scrF_{m,\Delta}$ has positive curvature current. Now let $Y=T$ be a complex torus; by an induction argument we can further assume that $T$ is a simple torus, that is, containing no non-trivial subtori. Then by a structure theorem for pseudo-effective line bundles on complex tori \cite[Theorem 3.3]{CP17} we have the following dichotomy according the sign of $\det\!\scrF_{m,\Delta}$:
\begin{itemize}
\item there is a integer $m>0$ sufficiently large and divisible such that $\det\!\scrF_{m,\Delta}$ is ample;
\item for every $m$ sufficiently large and divisible, $\det\!\scrF_{m,\Delta}$ is numerically trivial.
\end{itemize}
Apparently the first case fall into the situation of {\hyperref[main-thm_I]{Part (I)}} of our {\hyperref[main-thm]{Main Theorem}}. Hence we only need to tackle the second case, where one can use the $L^2$ extension property to further conclude that $(\scrF_{m,\Delta}\,,g_{X/Y\!,\Delta}^{(m)})$ is a Hermitian flat vector bundle.   
Furthermore, by a standard argument dated to Kawamata, we are reduced to the case $\kappa(X,K_X+\Delta)\leqslant 0$, i.e. it is enough to prove that $\kappa(F,K_F+\Delta_F)\geqslant 1$ implies $\kappa(X,K_X+\Delta)\geqslant 1$. This reduction relies on the following a log Kähler version of \cite[Theorem 1]{Kaw81}, which follows from \cite[Theorem 4.2]{Cam04} or \cite[Theorem1.7]{Fuj17} (or {\hyperref[thm_Viehweg_Iitaka-type-gen]{Theorem \ref*{thm_Viehweg_Iitaka-type-gen}}} for the klt case):

\begin{mainthm}
\label{thm_Kawamata_Ab-Var}
Let $X$ be a compact Kähler manifold. Suppose that there is an effective $\QQ$-divisor $\Delta$ on $X$ such that $(X,\Delta)$ is log canonical and that $\kappa(X,K_X+\Delta)=0$ 
(i.e. $X$ is bimeromorphically log Calabi-Yau). 
Then the Albanese map $\alb_X:\,X\to \Alb_X$ of $X$ is a surjective morphism with connected fibres. 
\end{mainthm}

The proof of this theorem will be given in {\hyperref[sec_Alb-kod=0]{\S \ref*{sec_Alb-kod=0}}}, it is similar to that of \cite{Kaw81}. In fact, when $\Delta=0$ and $X$ projective, the theorem is proved in \cite{Kaw81}; for $\Delta=0$ and $X$ Kähler a proof is also sketched in \cite[Theorem 24]{Kaw81}, but does not contain enough details. In virtue of \cite[Theorem 1.7]{Fuj17} (or  {\hyperref[thm_Viehweg_Iitaka-type-gen]{Theorem \ref*{thm_Viehweg_Iitaka-type-gen}}} for the klt case) one can easily obtain {\hyperref[thm_Kawamata_Ab-Var]{Theorem \ref*{thm_Kawamata_Ab-Var}}} by following the strategies of \cite{Kaw81}, and it is exactly in this way our proof in {\hyperref[sec_Alb-kod=0]{\S \ref*{sec_Alb-kod=0}}} proceeds. Let us remark that a similar result with $\Delta=0$ for special varieties in the sense of Campana is also stated in \cite{Cam04} where the proof is sketched based on \cite{Kaw81}.

Now we are reduced to show that $\kappa(F,K_F+\Delta_F)\geqslant 1$ implies that $\kappa(X,K_X+\Delta)\geqslant 1$. $\scrF_{m,\Delta}$ being Hermitian flat, it is given by a unitary representation $\rho_m$ of the fundamental group of $T$; $\pi_1(T))$ being Abelian, this representation is decomposed into $1$-dimensional sub-representations. If the image of $\rho_m$ is finite, then one can use the parallel transport to extend pluricanonical sections on $F$ to $X$; if the image of $\rho_m$ is infinite, then a fortiori $\kappa(X,K_X+\Delta)\geqslant 1$ by the following pluricanonical klt Kähler version of the structure theorem on cohomology jumping loci à la Green-Lazarsfeld-Simpson (c.f. \cite{GL91,Sim93}), which consist of another key ingredient of the proof of {\hyperref[main-thm_II]{Main Theorem, Part (I\!I)}}:
\begin{mainthm}
\label{thm_Simpson_pluri-klt}
Let $g:X\to Y$ be a morphism between compact Kähler manifolds. Let $\Delta$ be an effective $\QQ$-divisor on $X$ such that $(X,\Delta)$ is a klt pair. Then for every $m>0$ such that $m\Delta$ is an integral divisor and for every $k>0$, the cohomology jumping locus 
\[
V_k^0\left(g_\ast\left(K_X\ptensor[m]\otimes\scrO_X(m\Delta)\right)\right):=\left\{\;\rho\in\Pic^0(X)\;\big|\;\dimcoh^0(Y,g_\ast(K_X\ptensor[m]\otimes\scrO_X(m\Delta))\otimes\rho)\geqslant k\;\right\}
\]
is a finite union of torsion translates of subtori in $\Pic^0(Y)$.
\end{mainthm}
The study of cohomology jumping loci is initiated by the works of Green-Lazarsfeld \cite{GL87,GL91} which assure that the components of cohomology jumping loci are translates of subtori, and is further developed by Carlos Simpson in \cite{Sim93}, where he proves that these translates are torsion translates. Recently, the main result of \cite{Sim93} is generalized by Botong Wang to the Kähler case in \cite{Wan16}, where he treats the case $g=\id_X$\,, $m=1$ and $\Delta=0$ and this is the starting point of our proof of {\hyperref[thm_Simpson_pluri-klt]{Theorem \ref*{thm_Simpson_pluri-klt}}}. In fact, when $g=\id_X$ and $X$ projective, the proof of the theorem is already implicitly comprised in \cite{CKP12} although they only explicitly state and prove a result corresponding to our {\hyperref[cor_thm_Simpson_pluri-klt]{Corollary \ref*{cor_thm_Simpson_pluri-klt}}} with $X$ smooth projective and $(X,\Delta)$ log canonical by using \cite{Sim93}; we thus follow the strategy in \cite{CKP12} to deduce {\hyperref[thm_Simpson_pluri-klt]{Theorem \ref*{thm_Simpson_pluri-klt}}} from the basic case treated in \cite[Corollary 1.4]{Wan16}. Notice that \cite{Wan16} and hence our {\hyperref[thm_Simpson_pluri-klt]{Theorem \ref*{thm_Simpson_pluri-klt}}} require that $X$ is "globally" Kähler; 
by contrast, {\hyperref[thm_pos-im-pluri-can]{Theorem \ref*{thm_pos-im-pluri-can}}} holds for any Kähler fibre space ($X$ is only assumed to be locally Kähler over $Y$). Let us remark that in the hypothesis of $C_{n,m}^{\log}$ it is essential to suppose that $X$ is globally Kähler, in fact \cite[Remark 15.3, p.~187]{Uen75} provides an example of a Kähler fibre space for which $C_{n,m}$ does not hold. 

Let us explain how to finish the proof of {\hyperref[main-thm_II]{Part (I\!I)}} of the {\hyperref[main-thm]{Main Theorem}} from {\hyperref[thm_Simpson_pluri-klt]{Theorem \ref*{thm_Simpson_pluri-klt}}}. By following the argument in \cite{CP11} one easily deduces from {\hyperref[thm_Simpson_pluri-klt]{Theorem \ref*{thm_Simpson_pluri-klt}}} (c.f. {\hyperref[cor_thm_Simpson_pluri-klt]{Corollary \ref*{cor_thm_Simpson_pluri-klt}}}):
\begin{itemize}
\item $K_X+\Delta$ is the most effective $\QQ$-line bundle in its numerical class. 
\item If $\kappa(X,K_X+\Delta)=\kappa(X,K_X+\Delta+L)=0$ for some numerically trivial ($\QQ$-)line bundle $L$, then $L$ is a torsion point in $\Pic^0(X)$. 
\end{itemize}
Now the proof of {\hyperref[main-thm_II]{Main Theorem, Part (I\!I)}} can be finished as follows: if $\Image(\rho_m)$ is infinite, by the decomposition of $\scrF_{m,\Delta}$ one sees that $K_X+\Delta$ has non-negative Kodaira dimension up to twisting a non-torsion numerically trivial ($\QQ$-)line bundle, hence the first point above shows that $\kappa(X,K_X+\Delta)\geqslant 0$; moreover, if $\kappa(X,K_X+\Delta)=0$ then the second point will lead to a contradiction, hence a fortiori $\kappa(X,K_X+\Delta)\geqslant 1$, thus we finish the proof of {\hyperref[main-thm]{Main Theorem}}. As a by-product of the first point above, we can prove the Kähler version of the (generalized) log Abundance Conjecture in the case of numerical dimension zero (c.f. {\hyperref[thm_Abundance_kod=0]{Theorem \ref*{thm_Abundance_kod=0}}}) by using the divisorial Zariski decomposition obtained in \cite{Bou04} (c.f.\cite[Definition 3.7]{Bou04}) .

Let us remark that one can follow the same strategies in \cite[\S 5]{CP17} to prove more generally that the $C_{n,m}^{\log}$ is true if $\det\!\scrF_{m,\Delta}$ is numerically trivial for some $m\in\ZZ_{>0}$ (i.e. the Kähler version of \cite[Theorem 5.6]{CP17}) by using the remarkable result of Zuo in \cite[Corollary 1]{Zuo96}. In this article, however, we will not further pursue in this direction.

Finally by using an induction argument and by applying the results already obtained we generalize {\hyperref[main-thm_II]{Part (I\!I)}} of the {\hyperref[main-thm]{Main Theorem}} to the geometric orbifold setting:
\begin{mainthm}
\label{main-thm_orbifold}
Let $f:X\to T$ be a surjective morphism with $X$ compact Kähler manifold and $T$ complex torus, such that the general fibre $F$ of $f$ is connected. And let $\Delta$ be an effective $\QQ$-divisor on $X$ such that $(X,\Delta)$ is klt. 
Then 
\[
\kappa(X,K_X+\Delta)\geqslant \kappa(F,\Delta_F)+\kappa(T,B_{f\!,\Delta}).
\]
where $\Delta_F:=\Delta|_F$ and $B_{f\!,\Delta}$ denotes the branching divisor on $T$ w.r.t $f$ and $\Delta$. 
\end{mainthm}

In the theorem above, the branching divisor is defined as following: for any analytic fibre space $f: (X,\Delta)\to Y$ between compact complex manifolds with $\Delta$ an effective $\QQ$-divisor on $X$, the branching divisor $B_{f\!,\Delta}$ with respect to $f$ and $\Delta$ is defined as the most effective $\QQ$-divisor on $Y$ such that $f^\ast B_{f\!,\Delta}\leqslant R_{f\!,\Delta}$ modulo exceptional divisors, where the ramification divisor w.r.t. $f$ and $\Delta$ is defined as $R_{f\!,\Delta}:=\Sigma_f+\Delta$ and
\[
\Sigma_f:=\sum_{f(W)\text{ is a divisor on }Y}(\Ramification_W(f)-1)W
\]
with $\Ramification_W(f)$ denoting the  ramification (in codimension $1$) index of $f$ along $W$. Precisely, assume the singular locus of $f$ is contained in a (reduced) divisor $\Sigma_Y\subseteq Y$ and write
\[
f^\ast\Sigma_Y=\sum_{i\in I}b_iW_i\,,
\]
where $W_i$ are prime divisors on $X$, then for $i\in I^{\divisor}$ where
\[
I^{\divisor}:=\text{ set of indices }i\in I\text{ such that }f(W_i)\text{ is a divisor on }Y,
\]
we have $b_i=\Ramification_{W_i}(f)$ and thus
\[
\Sigma_f=\sum_{i\in I^{\divisor}}(b_i-1)W_i\,.
\]
Let us remark that the above definition of $B_{f\!,\Delta}$ coincides with \cite[Definition 1.29]{Cam04} (orbifold base) when $\Delta$ is lc on $X$, c.f. {\hyperref[sec_orbifold]{\S \ref*{sec_orbifold}}}.

The organization of the article is as follows. In {\hyperref[sec_preliminary]{\S \ref*{sec_preliminary}}} we recall some preliminary results which may be of independent interest; especially, the definition of semi-positively curved singular Hermitian metrics and that of the $L^2$ extension property are formulated in {\hyperref[ss_preliminary_metric]{\S \ref*{ss_preliminary_metric}}}. {\hyperref[sec_pos-im-dir]{\S \ref*{sec_pos-im-dir}}} is dedicated to the construction of the $m$-Bergman kernel metric on the adjoint line bundle  and of the canonical $L^2$ metric on direct images as well as the proof of {\hyperref[thm_pos-im-pluri-can]{Theorem \ref*{thm_pos-im-pluri-can}}}. {\hyperref[main-thm_I]{Main Theorem, Part (I)}} and $C_{n,m}^{\log}$ for general type bases are established in {\hyperref[ss_Kah-Viehweg-Kawamata_type-gen]{\S \ref*{ss_Kah-Viehweg-Kawamata_type-gen}}} and {\hyperref[ss_Kah-Viehweg-Kawamata_det-gros]{\S \ref*{ss_Kah-Viehweg-Kawamata_det-gros}}} respectively. And the proof of {\hyperref[thm_Kawamata_Ab-Var]{Theorem \ref*{thm_Kawamata_Ab-Var}}} is done in {\hyperref[sec_Alb-kod=0]{\S \ref*{sec_Alb-kod=0}}}. The general definition of cohomology jumping loci for coherent sheaves, as well as the proof of {\hyperref[thm_Simpson_pluri-klt]{Theorem \ref*{thm_Simpson_pluri-klt}}} will be given in {\hyperref[sec_Simpson]{\S \ref*{sec_Simpson}}}, where {\hyperref[cor_thm_Simpson_pluri-klt]{Corollary \ref*{cor_thm_Simpson_pluri-klt}}} and {\hyperref[thm_Abundance_kod=0]{Theorem \ref*{thm_Abundance_kod=0}}} are proved in {\hyperref[ss_Simpson_CKP]{\S \ref*{ss_Simpson_CKP}}}. In {\hyperref[sec_demo-main-thm]{\S \ref*{sec_demo-main-thm}}} we complete the proof of the {\hyperref[main-thm_II]{Part (I\!I)}} of our {\hyperref[main-thm]{Main Theorem}}. And finally the geometric orbifold version {\hyperref[main-thm_orbifold]{Theorem \ref*{main-thm_orbifold}}} is established in {\hyperref[sec_orbifold]{\S \ref*{sec_orbifold}}}.

\paragraph{Acknowledgement}
The author owes a lot to his thesis advisors Sébastien Boucksom and Junyan Cao, who have given him enormous help to accomplish this work as well as to abolish the present article. He is also grateful to Daniel Barlet, Nero Budur, Frédéric Campana, Benoit Claudon, Andreas Höring, S\'andor Kov\'acs, Mihnea Popa and Chenyang Xu for helpful discussions about this work. Also the author would like to take this opportunity to acknowledge the support he has benefited from the ANR project "GRACK" during the preparation of the present article.
\section{Preliminary Results}
\label{sec_preliminary}
In this section, we collect miscellaneous results which not only serve our main purpose but also are of independent interest. 

\subsection{An Analytic Geometry Toolkit}
\label{ss_preliminary_toolkit}
In this subsection we state some auxiliary results which are well-known in algebraic geometry, but whose analytic versions, as far as we know, have not yet been well formulated in literatures; we will not give the detailed proofs but instead indicate how to get rid of the algebraicity hypothesis.

\paragraph{\quad (A) A Covering Lemma \\ }
First we state a covering lemma which allow us to reduce problems on pluricanonical bundles to the case of the canonical bundle.
\begin{lemme}
\label{lemme_cov-trick}
Let $X$ be compact complex manifold. 
and let $L$ be a line bundle on $X$ such that $\kappa(X,L)\geqslant 0$. Suppose that there exists an integer $m>0$ such that there exists an effective divisor $D\in\left|L\ptensor[m]\right|$ whose support is SNC. Then there is a complex manifold $V$ admitting a surjective generically finite projective morphism $f:V\to X$ such that the direct image of $K_V$ admits a direct decomposition: 
\[
f_\ast K_V\simeq\bigoplus_{i=0}^{m-1}K_X\otimes L\ptensor[i]\otimes\scrO_X(-\bigl\lfloor\frac{i}{m}D\bigr\rfloor).
\]
\end{lemme}

The construction of $f$ is done by taking a cyclic cover along $D$ followed by a desingularization. This construction is standard. However, there are three main ingredients in this construction that need to be clarified:
\begin{itemize}
\item[\rm(a)]\label{rmq_lemme_cov-trick_a} The construction of cyclic covers:  c.f. \cite[\S 4.1.B, pp.~242-243, vol.I]{Laz04} and \cite[\S 2.9, p.~9]{Kol97}, which can be easily generalized to the analytic case.
\item[\rm(b)]\label{rmq_lemme_cov-trick_b} Viehweg's results on rational singularities in \cite{Vie77}:
\begin{itemize}
\item[\rm(b1)]\label{rmq_lemme_cov-trick_b1} A finite ramified cover over a smooth projective variety with the cover space being normal and the branching locus being a SNC divisor, has quotient singularities (\cite[Lemma 2]{Vie77}); in this case, the singularity is toroidal, and the result is standard from \cite{KKMS73}.
\item[\rm(b2)]\label{rmq_lemme_cov-trick_b2} A quotient singularity is a rational singularity (\cite[Proposition 1]{Vie77}). This follows from Kempf's criterion on rationality of singularities (c.f. \cite[\S I.3, condition (d)(e) pp.~50-51]{KKMS73}), which is essentially a analytic result. 
  
\end{itemize}
\item[\rm(c)]\label{rmq_lemme_cov-trick_c} A duality theorem for canonical sheaves (the canonical sheaf of a complex variety is defined as the $(-d)$-th cohomology of the dualizing complex, where $d$ denotes the dimension of the complex variety) on singular complex varieties, which can be proved by applying \cite{RR70} or \cite{BS76} combined with a spectral sequence argument.
\end{itemize}

\begin{rmq}
\label{rmq_lemme_cov-trick}
For later use, we remark that the point {\hyperref[rmq_lemme_cov-trick_b2]{(b2)}} above can be further generalized to higher relative dimension 
by a local computation as in \cite[Lemma 3.6]{Vie83} and by \cite{KKMS73}: for $f:X\to Y$ be a proper flat morphism between complex manifolds such that the singular locus $\Sigma_Y\subseteq Y$ is a smooth divisor and the preimage $f^\ast\Sigma_Y$ is a {\em reduced} SNC divisor, then for any surjective morphism $\phi: Y'\to Y$ with $Y'$ smooth, the fibre product $X\underset{Y}{\times}Y'$ has (at most) rational singularities. C.f. also \cite[3.13.Lemma]{Hor10}.      
\end{rmq}

\paragraph{\quad (B) The Negativity Lemma \\ }
The negativity lemma is an important tool in the study of the classification theory of complex varieties. It is already well known in the algebraic case, c.f. \cite[Lemma 3.39, p.~102-103]{KM98}. By following the strategy of \cite[Proposition 2.12]{BdFF12} one can prove the following analytic version (for the convenience of the readers, we provide a proof in {\hyperref[appendix_sec_neg-lemma]{Appendix \ref*{appendix_sec_neg-lemma}}}):

\begin{lemme}[Negativity Lemma]
\label{lemme_neg}
Let $h: Z\to Y$ be a proper bimeromorphic morphism between normal complex varieties. Let $B$ be a Cartier divisor on $Z$ such that $-B$ is $h$-nef. Then $B$ is effective if and only if $h_\ast B$ is effective.   
\end{lemme}


\paragraph{\quad (C) A Flattening Lemma \\ }
In order to prove {\hyperref[main-thm]{Part (I)}} of the {\hyperref[main-thm]{Main Theorem}} we need the following auxiliary result, which is an analytic version of \cite[Lemma 7.3]{Vie83}\,:
\begin{lemme}
\label{lemme_Viehweg-aplatissement}
Let $p:V\to W$ a morphism of complex manifolds, then there exists a commutative diagram
\begin{center}
\begin{tikzpicture}[scale=2.0]
\node (A) at (0,0) {$W$};
\node (B) at (0,1) {$V$};
\node (A') at (-1,0) {$W'$};
\node (B') at (-1,1) {$V'$};
\path[->,font=\scriptsize,>=angle 90]
(B) edge node[right]{$p$} (A)
(B') edge node[left]{$p'$} (A')
(A') edge node[below]{$\pi_W$} (A)
(B') edge node[above]{$\pi_V$} (B);
\end{tikzpicture}
\end{center}
with $V'$ and $W'$ complex manifolds, the morphisms $\pi_W$ and $\pi_V$ projective and bimeromorphic such that the morphism $p'$ verifies the following propriety: every $p'$-exceptional (i.e. $\codim_{W'}p'(D')\geqslant 2$) divisor $D'$ de $V'$   is $\pi_V$-exceptional (i.e. $\codim_V(\pi_V(D'))\geqslant 2$). In addition, we can further assume that
\begin{itemize}
\item[\rm(a)] $\pi_W$ is an isomorphism over $W_0$, the (analytic) Zariski open subset of$W$ over which $p$ is smooth;
\item[\rm(b)] $\pi_V$ is an isomorphism over $p\inv W_0$\,; 
\item[\rm(c)] $\Sigma_{W'}:=\pi_W\inv(W\backslash W_0)$ and $p'^\ast\Sigma_{W'}$ are divisors of SNC support. 
\end{itemize}
\end{lemme}
\begin{proof}
This is simply a consequence of \cite[Flattening Theorem]{Hir75}. 
\end{proof} 

In the sense of \cite{Cam04}, the lemma above shows that any fibre space admits a (higher) bimeromorphic model which is {\em neat} and {\em prepared} (c.f. \cite[\S 1.1.3]{Cam04}). Moreover, {\hyperref[lemme_Viehweg-aplatissement]{Lemma \ref*{lemme_Viehweg-aplatissement}}} is well behaved with respect to klt/lc pairs, as implies the following fact:
 
\begin{lemme}
\label{lemme_preservation-klt}
Let $X$ be a complex manifold
and $\Delta$ an effective $\QQ$-divisor on $X$ such that the pair $(X,\Delta)$ is klt (resp. lc). For any log resolution $\mu:X'\to X$ of $(X,\Delta)$, there is an effective $\QQ$-divisor $\Delta'$ over $X'$ with SNC support such that the pair $(X',\Delta')$ is also klt (resp. lc) and that $\mu_\ast\Delta'=\Delta$.
\end{lemme}
\begin{proof}
The pair $(X,\Delta)$ being klt, we can write (an isomorphism of $\QQ$-line bundles):
\[
K_{X'}+\mu\inv_\ast\!\!\Delta-\sum_{a_i<0}a_i E_i\simeq \mu^\ast(K_X+\Delta)+\sum_{a_i>0}a_i E_i,
\]
where the $E_i$'s are $\mu$-exceptional prime divisors and 
\[
a_i:=a(E_i,X,\Delta)
\]
denotes the discrepancy of $E_i$ with respect to the pair $(X,\Delta)$. Put
\[
\Delta':=\mu\inv_\ast\!\!\Delta-\sum_{a_i<0}a_iE_i\,,
\]
then $\Delta'$ is an effective $\QQ$-divisor with SNC support and $\mu_\ast\Delta'=\Delta$. The hypothesis that $(X,\Delta)$ is klt (resp. lc) implies that $a_i>-1$ (resp. $a_i\geqslant-1$) for every $i$ and that the coefficients of prime components in $\Delta$ are $<1$ (resp. $\leqslant 1$), hence the coefficients of the prime components in $\Delta'$ are all $<1$ (resp. $\leqslant 1$). By \cite[Corollary 2.31(3), p.~53]{KM98} the pair $(X',\Delta')$ is klt (resp. lc).
\end{proof}

\subsection{Griffiths Semi-positive Singular Hermitian Metrics on Vector Bundles / Torsion Free Sheaves}
\label{ss_preliminary_metric}
In this subsection we will recall the notion of Griffiths semi-positively curved singular Hermitian metrics on vector bundles / torsion free sheaves. C.f. \cite[2.1 et 2.2]{CH17b} for a generalization of this semi-positivity notion. Let us fix $X$ a complex manifold.



\begin{defn}
\label{defn_fv-pos}
Let $E$ be holomorphic vector bundle on $X$. A (Griffiths) semi-positively curved singular Hermitian metric $h$ on $E$ is given by a measurable family of Hermitian functions on each fibre of $E$, such that for every (holomorphic) local section $s\in\Coh^0(U,E^\ast)$ of the dual bundle $E^\ast$, the function $\log|\sigma|^2_{h^\ast}$ is psh on $U$. The vector bundle $E$ is said semi-positively curved if it admits a semi-positively curved singular metric.
\end{defn}

\begin{rmq}
\label{rmq_defn_tf-pos}
This definition implies that $h$ is bounded almost everywhere, moreover, fix any smooth Hermitian metric $h_0$ on $E$, then
as a consequence of
{\cite[2.10.Remark, 2.18.Remark]{Pau16}}
the singular metric $h$ is locally uniformly bounded from below by $C\cdot h_0$ for some constant $C>0$.
\end{rmq}

The semi-positivity of singular Hermitian metrics is preserved by tensor products , pull-back by proper surjective morphisms, and by generically surjective morphisms of vector bundles (thus by symmetric products and wedge products), c.f. \cite[II.B.4]{GG18} and \cite[2.14.Lemma, 2.15.Lemma]{Pau16}. Moreover one has the following extension theorem for semi-positively curved singular Hermitian metrics:
\begin{prop}[c.f. {\cite[2.4.Proposition]{CH17b}}] 
\label{prop_fv-pos}
Let $E$ be a holomorphic vector bundle on $X$. Suppose that there is a (analytic) Zariski open subset $X_0\neq\varnothing$ of $X$ and a semi-positively curved singular Hermitian metric $h$ on $E|_{X_0}$. Then $h$ extends to a semi-positive singular Hermitian metric on $E$ if one the following two conditions is verified:
\begin{itemize}
\item[\rm(1)] $\codim(X\backslash X_0)\geqslant 2$;
\item[\rm(2)] $h$ is locally uniformly bounded below by a constant $C>0$ on $X_0$ with respect to some smooth Hermitian metric on $E$. \end{itemize}
\end{prop}
 
In virtue of {\hyperref[prop_fv-pos]{Proposition \ref*{prop_fv-pos}}} and \cite[Corollary 5.5.15, p.~147]{Kob87} one can extend 
{\hyperref[defn_fv-pos]{Definition \ref*{defn_fv-pos}}} to torsion free sheaves:
\begin{defn}
\label{defn_tf-pos}
Let $X$ be a complex manifold and let $\scrF$ be a torsion free sheaf on $X$. 
By \cite[Corollary 5.5.15, p.~147]{Kob87}, $\scrF$ is locally free in codimension $1$
A semi-positively curved singular Hermitian metric $h$ on $\scrF$ is a semi-positively curved singular Hermitian metric on $\scrF|_U$ for some (analytic) Zariski open subset $U$ such that $\codim_X U\geqslant 2$ and $\scrF|_U$ locally free. The torsion free sheaf $\scrF$ said semi-positively curved if it admits a semi-positively curved singular Hermitian metric. 
\end{defn}

Let $\scrF$ and $h$ as above, then $h$ induces a semi-positively curved singular Hermitian metric $\det\!h$ on the line bundle $\det\!\scrF$ where the determinant bundle $\det\!\scrF$ is defined as
\[
\det\!\scrF:=\left(\bigwedge^{r}\scrF\right)^\wedge
\]
with $r=\rank\scrF$ and $()^\wedge=()^{\ast\ast}$ denotes the reflexive hull (c.f. \cite[\S 5.6, pp.~149-154]{Kob87}). 

We end this subsection by two regularity theorems:
\begin{thm}
\label{thm_pos-reg}
Let $(E,h)$ be a holomorphic vector bundle on $X$ equipped with a semi-positively curved singular Hermitian metric $h$.
Suppose that the metric $\det\!h$ is locally bounded from above, then the coefficients of the Chern connection form $\theta_E$ (defined by the equation $h\theta_E=\partial h$) are $L^2_{\text{loc}}$ on $U$, and in consequence the total curvature current $\Theta_{h}(E)$ of $E$ is well defined and semi-positive in the sense of Griffiths, which can be locally written as $\Theta_{h}(E)=\bar\partial\theta_E$\,. In particular, if the curvature current $\Theta_{\det\!h}$ vanishes, then $(E,h)$ is Hermitian flat.
\end{thm}
\begin{proof}
The theorem is proved in \cite[Theorem 1.6]{Rau15} by an approximation argument (c.f. also \cite[2.25.Theorem, 2.26.Corollary]{Pau16}). Heuristically, this is a higher rank version of the well known fact (the line bundle case) that if a psh function $\phi$ is $L^\infty_{\text{loc}}$, then $\nabla\phi$ is $L^2_{\text{loc}}$. As for the last statement (c.f. \cite[2.26.Corollary]{Pau16} and \cite[2.7.Theorem]{CP17}): by our first statement the total curvature current $\Theta_h(E)$ is well defined and Griffith semi-positive, then the vanishing of $\Theta_{\det\!h}$ implies the vanishing of $\Theta_h(E)$; the regularity of $h$ results from the ellipticity of the Laplacian $\partial\bar\partial$.
\end{proof}
 
In the sequel we introduce the notion of "$L^2$-extension property", which is simply an analogue of the property of $\scrO$ that every $L^2$ holomorphic function extends. It helps to exclude certain unexpected pathology, e.g. the ideal sheaf $\scrI_Z$ of a analytic subset $Z$ of codimension $\geqslant 2$ admits a natural semi-positively curved singular Hermitian metric but does not satisfy the $L^2$ extension property. 
 
\begin{defn}
\label{defn_reflexivite-L2}
Let $\scrF$ be a torsion free sheaf on $X$ equipped with a singular Hermitian metric $h$. $h$ is said  to satisfy the "$L^2$-extension property" if for any open subset $U\subseteq X$, for any $Z\subsetneqq U$ analytic subset of $U$ and for any $\sigma\in\Coh^0(U\backslash Z,\scrF)$ such that 
\[
\int_U |\sigma|_h^2d\mu<+\infty\,,
\]
the section $\sigma$ extends (uniquely) to a section $\bar\sigma\in\Coh^0(U,\scrF)$.
\end{defn}
 
This propriety is particularly useful when we consider a torsion free sheaf whose determinant bundle is numerically trivial. In fact, we have the following:

\begin{thm}
\label{thm_det-triv_herm-plat}
Let $X$ be a connected complex manifold and let $\scrF$ be a torsion free sheaf of rank $r$ on $X$ equipped with a semi-positively curved singular Hermitian metric $h$. Suppose that
\begin{itemize}
\item[\rm(1)]  $\det\!\scrF$ is numerically trivial, i.e. $c_1(\det\!\scrF)=c_1(\scrF)=0$;
\item[\rm(2)] $h$ satisfies the $L^2$-extension property as in {\hyperref[defn_reflexivite-L2]{Definition \ref*{defn_reflexivite-L2}}}.
\end{itemize}
Then $(\scrF\!,h)$ is a Hermitian flat vector bundle.
\end{thm}
\begin{proof}
The proof is essentially analogous to that  of \cite[Theorem 5.2]{CP17}. Since $h$ is semi-positively curved, the metric $\det\!h$ on $\det\!\scrF$ is semipositively curved, thus the curvature current $\Theta_{\det\!h}(\det\!\scrF)$ is positive; but $\det\!\scrF$ is numerically trivial, hence a fortiori $\Theta_{\det\!h}(\det\!\scrF)=0$. Then by {\hyperref[thm_pos-reg]{Theorem \ref*{thm_pos-reg}}}, $(\scrF|_{X_{\scrF}},h|_{X_{\scrF}})$ is a Hermitian flat vector bundle (i.e. $h|_{X_{\scrF}}$ is a smooth Hermitian metric $\scrF|_{X_{\scrF}}$ whose curvature vanishes). 
By \cite[Proposition 1.4.21, p.~13]{Kob87} the Hermitian flat vector bundle $(\scrF|_{X_\scrF},h|_{X_\scrF})$ is defined by a representation 
\[
\pi_1(X_\scrF)\to \Unitary(r),
\]
$\pi_1(X_{\scrF})$ being isomorphic to $\pi_1(X)$, this extends to a representation
\[
\pi_1(X)\to \Unitary(r),
\]
which gives rise to a Hermitian vector bundle $(E,h_E)$ of rank $r$ on $X$. Then by construction we have an isometry
 \[
 \phi:\scrF|_{X_{\scrF}}\to E|_{X_{\scrF}}.
 \]
By reflexivity of $\SheafHom[\scrO_X](\scrF,E)$ this extends to an injection of sheaves $\scrF\hookrightarrow E$ which we still denote by $\phi$. It remains to show that $\phi$ is surjective. The problem being local, we can assume that $X$ is a small open ball, so that $E$ is trivial. Now take $u\in\Coh^0(X,E)$ a holomorphic section of $E$, since $h_E$ is a flat metric (hence smooth), $|u|_{h_E,z}$ is finite for every $z\in X$. $\phi|_{X_{\scrF}}$ being an isometry, there exists a section $v_0\in\Coh^0(X_{\scrF},\scrF)$ such that $i(v_0)=u|_{X_{\scrF}}$ and $|v_0|_{h,z}=|u|_{h_E,z}<+\infty$ for all $z\in X_{\scrF}$. But $(\scrF\!,h)$ satisfies the $L^2$ extension property, $v_0$ extends to a section $v\in\Coh^0(X,\scrF)$, thus $\phi(v)=u$, implying the surjectivity of $\phi$. 
\end{proof}
Let us remark that the condition on the $L^2$ extension property is indispensable in the theorem above. For example, as mentioned above, the ideal sheaf $\scrI_Z$ of an analytic subset $Z$ of codimension $\geqslant 2$ admits a natural semi-positively curved singular Hermitian metric $h_{\!\scrI_Z}$, which equals to the flat metric of $\scrO$ on $X\backslash Z$. The determinant of $\scrI_Z$ is trivial, but definitely $\scrI_Z$ is not a (Hermitian flat) vector bundle. Notice that $(\scrI_Z, h_{\!\scrI_Z})$ does not satisfy the $L^2$ extension property: let $B$ be a small ball in $X$ meeting $Z$, then non-zero constant functions on $B\backslash Z$ (which are $L^2$) cannot extend across $Z$.

\subsection{Reflexive Hull of the Direct Image of Line Bundles}
\label{ss_preliminary_ref-hull}
In this subsection we will prove the following theorem,  which is nothing but an analytic version of \cite[III.5.10.Lemma, pp.~107-108]{Nak04}. The proof of the theorem is not essentially different from that in \cite{Nak04}; except that, for the analytic case, one has to modify the arguments, especially in the {\hyperref[env-ref_etape2]{Step 2}} below, so that on can avoid the usage of the relative Zariski decomposition (which is not known in analytic case; even in the algebraic case, it is only established in some special cases in \cite{Nak04} and it does not hold in general due to a counterexample in \cite{Les16}). 

\begin{thm}
\label{thm_env-ref}
Let $\pi:X\to S$ be 
a proper surjective morphism between normal complex varieties, and let $L$ be a $\pi$-effective (i.e. $\pi_\ast L\neq 0$) line bundle on $X$. Then there is an effective $\pi$-exceptional (i.e. $\codim_S\pi(E)\geqslant 2$) Weil divisor $E$ such that for any $k\in\ZZ_{>0}$ one has
\begin{equation}
\label{eq_env-ref}
\left[\pi_\ast\! \left(L\ptensor[k]\right)\right]^{\wedge} \simeq \pi_\ast \left[L\ptensor[k]\otimes \scrO_X(kE)\right]\,.
\end{equation}
\end{thm}
Intuitively the theorem means that the vertical poles of the sections of $L\ptensor[k]$ are linearly bounded.
The proof of {\hyperref[thm_env-ref]{Theorem \ref*{thm_env-ref}}} proceeds in five steps:

\paragraph{Step 0}
First let us remark that we can always assume that $X$ is smooth by taking a desingularization by the following observation 
\begin{lemme}
\label{lemme_incl-div-birat}
Let $h:Z\to Y$ a bimeromorphic morphisme between normal complex varieties. Then for every Weil divisor $D$ on $Z$, we have an inclusion
\[
h_\ast\scrO_Z(D)\subseteq \scrO_Y(h_\ast D).
\]
\end{lemme}
\begin{proof}
Since $h$ is an isomorphism over a(n) (analytic) Zariski open subset of codimension $\geqslant 2$ in $Y$, $h_\ast\scrO_Z(D)$ and $\scrO_Y(h_\ast D)$ are isomorphic in codimension $1$; $h_\ast\scrO_Z(D)$ being torsion free and $\scrO_Y(h_\ast D)$ reflexive, we have (noting that on a normal complex variety reflexive sheaves are determined in codimension $1$): 
\[
h_\ast\scrO_Z(D)\hookrightarrow\left(h_\ast\scrO_Z(D)\right)^\wedge\simeq\scrO_Y(h_\ast D)\,.
\]
\end{proof}
In fact, assume that {\hyperref[thm_env-ref]{Theorem \ref*{thm_env-ref}}} holds for $X$ smooth, let us prove that it holds in general. To this end, let $\mu:X'\to X$ be a desingularization of $X$, then by our assumption, there is an effective divisor $E'$ on $X'$ such that 
\[
\left[\pi'_\ast\left(\mu^\ast L\ptensor[k]\right)\right]^\wedge=(\pi')_\ast\left[\mu^\ast L\ptensor[k]\otimes\scrO_{X'}(kE')\right],
\]
hence by {\hyperref[lemme_incl-div-birat]{Lemma \ref*{lemme_incl-div-birat}}} and the projection formula we have
\[
\left[\pi_\ast\left( L\ptensor[k]\right)\right]^\wedge=\pi_\ast\left[L\ptensor[k]\otimes\mu_\ast\scrO_{X'}(kE')\right]\subseteq
\pi_\ast\left[L\ptensor[k]\otimes\scrO_{X}(kE)\right]
\]
where $E:=\mu_\ast E'$; since the two sheaves above are reflexive and isomorphic in codimension $1$, the inclusion is in fact an equality. Consequently, we always assume that $X$ is smooth in the sequel.

\paragraph{Step 1}
\label{env-ref_etape1}
By the coherence of the reflexive hull $(\pi_\ast L)^\wedge$ there is an $\pi$-exceptional divisor $E$ making the equation \eqref{eq_env-ref} holds for $k=1$ (and thus one can choose $E$ such that \eqref{eq_env-ref} holds for a finite number of $k$).


\paragraph{Step 2}
\label{env-ref_etape2}
In virtue of {\hyperref[env-ref_etape1]{\bf Step 1}} we are able to prove the reflexivity criterion below:
\begin{prop}[Reflexivity Criterion]
\label{prop_critere-ref}
Let $\pi:X\to S$ and $L$ as in {\hyperref[thm_env-ref]{Theorem \ref*{thm_env-ref}}}. Suppose that for every effective $\pi$-exceptional divisor $G$, there is a component $\Gamma$ of $G$ such that 
$\left[L\otimes\scrO_X(G)\right]|_{\Gamma}$ is not $\pi|_\Gamma$-pseudoeffective, then $\pi_\ast L$ is reflexive on $S$. 
\end{prop}

Let us recall the notion of relative pseudoeffectivity for ($\QQ$-)line bundles / Cartier divisors in the analytic setting:
Let $p:V\to W$ a proper surjective morphism of analytic varieties and let $L$ be a $\QQ$-line bundle on $V$, then $L$ is said to be $p$-pseudoeffective if its pull-back $L|_{\tilde F}$ is pseudoeffective (c.f. \cite[\S 6.A, (6.2) Definition, p.~47]{Dem10}) where $\tilde F$ denotes a desingularization of the general fibre $F$ of $p$. A $\QQ$-Cartier divisor $D$ on $V$ is said to be $p$-pseudoeffective if its associated $\QQ$-line bundle $\scrO_X(D)$ is so. 

Now return to the proof of the {\hyperref[prop_critere-ref]{Reflexivity Criterion}}\,:
\begin{proof}[Proof of {\hyperref[prop_critere-ref]{Proposition \ref*{prop_critere-ref}}}]
We will show in the sequel that it suffices to prove {\hyperref[lemme_critere-ref]{Lemma \ref*{lemme_critere-ref}}} below. In fact, by {\hyperref[env-ref_etape1]{\bf Step 1}} there is an effective $\pi$-exceptional $E$, such that 
\[
\left(\pi_\ast L\right)^\wedge\simeq\pi_\ast\left[L\otimes\scrO_X(E)\right]\,;
\]
Apply {\hyperref[lemme_critere-ref]{Lemma \ref*{lemme_critere-ref}}} to $E$ and we obtain: 
\[
\pi_\ast L\simeq \pi_\ast\left[L\otimes\scrO_X(E)\right]\simeq\left(\pi_\ast L\right)^\wedge\,,
\]
hence $\pi_\ast L$ is reflexive.
\end{proof}

\begin{lemme}
\label{lemme_critere-ref}
Let $\pi:X\to S$ and $L$ as in {\hyperref[prop_critere-ref]{Proposition \ref*{prop_critere-ref}}}, then for any effective $\pi$-exceptional divisor $B$ on $X$, one has:
\begin{equation}
\label{eq_critere-ref}
\pi_\ast L\simeq \pi_\ast\left[L\otimes\scrO_X(B)\right]
\end{equation}
\end{lemme}
\begin{proof}
$B$ is effective, one can write
\[
B=\sum_{i=1}^r b_i B_i\,,
\]
with $b_i\in\ZZ_{>0}$ and $r\in\NN$ ($r=0$ simply means that $B=0$). Note
\[
b:=\sum_{i=1}^r b_i\,.
\]
Now let us prove \eqref{eq_critere-ref} by induction on $b$\,:
By our hypothesis on $L$ (the condition in {\hyperref[prop_critere-ref]{Proposition \ref{prop_critere-ref}}}), $\exists\, i\in\{\,1,\cdots, r\,\}$ such that $\left[L\otimes\scrO_X(B)\right]|_{B_i}$ is non-$\pi|_{B_i}$-pseudoeffective, thus
\[
(\pi|_{B_i})_\ast\left[L\otimes\scrO_X(B)\right]|_{B_i}=0.
\]
Consider the short exact sequence
\[
0\longrightarrow\scrO_X(-B_i)\longrightarrow\scrO_X\longrightarrow\scrO_{B_i}\longrightarrow 0.
\]
By tensoring with $L\otimes\scrO_X(B)$ and applying the functor $\pi_\ast$ on gets
\[
0\to\pi_\ast\left[L\otimes\scrO_X(B-B_i)\right]\longrightarrow\pi_\ast\left[L\otimes\scrO_X(B)\right]\longrightarrow(\pi|_{B_i})_\ast\left[L\otimes\scrO_X(B)\right]|_{B_i}=0
\]
hence $\pi_\ast\left[L\otimes\scrO_X(B-B_i)\right]\simeq \pi_\ast\left[L\otimes\scrO_X(B)\right]$\,. Apply the induction hypothesis we obtain that $\pi_\ast\left[L\otimes\scrO_X(B-B_i)\right]\simeq \pi_\ast L$, 
which proves the isomorphism (\ref{eq_critere-ref}). 
\end{proof}


\paragraph{Step 3}
\label{env-ref_etape3}
Let $\pi:X\to S$ and $L$ a $\pi$-effective line bundle on $X$ as in {\hyperref[thm_env-ref]{Theorem \ref*{thm_env-ref}}}. The problem begin local, one can replace $X$ (resp. $S$) by a neighbourhood of a compact in $X$ (resp. in $S$); in particular the set of $\pi$-exceptional prime divisors, denoted par $\scrE\!\!\text{\calligra xc}\,(\pi)$, is a finite set, and thus we can write:
\[
\scrE\!\!\text{\calligra xc}\,(\pi)=\left\{\,\Gamma_1,\cdots,\Gamma_t\,\right\}
\]
In virtue of the {\hyperref[prop_critere-ref]{Reflexivity Criterion}}, we prove that in order to prove {\hyperref[thm_env-ref]{Theorem \ref*{thm_env-ref}}} it suffices to search an effective $\pi$-exceptional divisor $E$  such that $E|_{\Gamma_i}$ is non-$\pi|_{\Gamma_i}$-pseudoeffective (The existence of such $E$ will be proven in the next step). 

\subparagraph{1.}
\label{env-ref_etape3-1}
$E$ being $\pi$-exceptional effective, we can write
\[
E=\sum_{i=1}^t a_i\Gamma_i\,,\;a_i\in\ZZ_{\geqslant 0}\,.
\]
We claim that the $a_i$'s are all strictly positive. Otherwise, there exists a $j$ such that $a_j=0$, implying that $\Gamma_j\not\subseteq\Supp(E)$, 
then $E|_{\Gamma_j}$ is an effective divisor, in particular it is $\pi|_{\Gamma_j}$-pseudoeffective, contradicting the hypothesis on $E$. 

\subparagraph{2.}
\label{env-ref_etape3-2}
Moreover claim that there is a $b\in\ZZ_{>0}$ such that $\forall\beta\geqslant b\,,\beta\in\QQ_{>0}$\,, $\left.\big(L+\beta E\big)\right|_{\Gamma_i}$ is a $\QQ$-line bundle which is non-$\pi|_{\Gamma_i}$-pseudoeffective for all $i=1,2,\cdots,t$. Otherwise there is a sequence of positive rational numbers $\beta_n\rightarrow +\infty$ such that for every $n$, $\left.\big(L+\beta_n E\big)\right|_{\Gamma_{i_n}}$ is a $\pi|_{\Gamma_{i_n}}$-pseudoeffective $\QQ$-line bundle for some $i_n$\,. In general $i_n$ should depend on $n$, but $\scrE\!\!\text{\calligra xc}\,(\pi)$ being a finite set, there must exist an $i$ appearing an infinity of times in the sequence $(i_n)_{n>0}$, thus up to taking a subsequence one can suppose that there exists an $i$ such that $\left.\big(L+\beta_n E\big)\right|_{\Gamma_i}$ is a $\pi|_{\Gamma_i}$-pseudoeffective $\QQ$-line bundle for every $n$. Hence
\[
\left. \left(E+\frac{1}{\beta_n}L\right) \right|_{\Gamma_i}
\]
is an $\pi|_{\Gamma_i}$-pseudo-effective $\QQ$-line bundle for every $n$. This implies (by letting $\beta_n\rightarrow +\infty$) that $E|_{\Gamma_i}$ is $\pi|_{\Gamma_i}$-pseudoeffective, contradicting to the point {\hyperref[env-ref_etape3-1]{\bf 1}} above.

\subparagraph{3.}
\label{env-ref_etape3-3}
Let us set
\[
L_k=L\ptensor[k]\otimes\scrO_X(kbE)\,,
\]
then in order to prove {\hyperref[thm_env-ref]{Theorem \ref*{thm_env-ref}}} we only need to show that $\pi_\ast L_k$ is reflexive. In fact, since $S$ is normal, and since $\pi_\ast\!\left(L^{\otimes k}\right)$ and $\pi_\ast L_k$ are isomorphic outside an analytic subset of codimension $\geqslant 2$, therefore as soon as $\pi_\ast L_k$ is reflexive, we get immediately  
\[
\pi_\ast L_k\simeq \left[\pi_\ast\left(L\ptensor[k]\right)\right]^\wedge\,.
\]
In the sequel we will prove that $\pi_\ast L_k$ is reflexive in virtue of {\hyperref[prop_critere-ref]{Proposition \ref*{prop_critere-ref}}}. It suffices to check that $L_k$ satisfies the conditions in {\hyperref[prop_critere-ref]{Proposition \ref*{prop_critere-ref}}}: 
let $G$ be an $\pi$-exceptional effective divisor, then there is a minimal $c\in\QQ_{>0}$ such that $cE\geqslant G$. In fact, if we write 
\[
G=\sum_{i=1}^t g_i\Gamma_i\,,
\]
then we can take 
\[
c=\max_{i=1,\cdots,t}\left\{\frac{g_i}{a_i}\right\}\,.
\]
In particular, by the minimality of $c$ there exists an $i$ such that $\Gamma_i\not\subseteq\Supp(cE-G)$, implying that the $\QQ$-divisor $\left.\big(cE-G\big)\right|_{\Gamma_i}$ is $\pi|_{\Gamma_i}$-pseudoeffective. However by the point {\hyperref[env-ref_etape3-1]{\bf 2}} above, the $\QQ$-line bundle
\[
\left.\big(L_k+G\big)\right|_{\Gamma_i}+\left.\big(cE-G\big)\right|_{\Gamma_i}=\left.k\left[L+\left(b+\frac{c}{k}\right)E\right]\right|_{\Gamma_i}
\]
is  non-$\pi|_{\Gamma_i}$-pseudoeffective, hence a fortiori the line bundle $\left.\big(L_k+G\big)\right|_{\Gamma_i}$ is not $\pi|_{\Gamma_i}$-pseudoeffective. Therefore $L_k$ satisfies the conditions in {\hyperref[prop_critere-ref]{Proposition \ref*{prop_critere-ref}}},  thus $L_k$ is reflexive. 

\paragraph{Step 4}
\label{env-ref_etape4}
In this last step, we finish the proof of  {\hyperref[thm_env-ref]{Theorem \ref*{thm_env-ref}}} by constructing the divisor $E$ in the {\hyperref[env-ref_etape3]{\bf Step 3}}. In fact one can show:
\begin{prop}
\label{prop_div-exc-non-rel-psef}
For any $\pi:X\to S$ as in {\hyperref[thm_env-ref]{Theorem \ref*{thm_env-ref}}}, there is an effective $\pi$-exceptional divisor $E$ such that for any $\pi$-exceptional prime divisor $\Gamma$, $E|_\Gamma$ is not $\pi|_\Gamma$-exceptional. 
\end{prop}
The proof is the same as that in \cite[III.5.10.Lemma, pp.~107-108]{Nak04}. For the convenience of the readers, we provide the details in the {\hyperref[appendix_sec_ref-hull]{Appendix \ref*{appendix_sec_ref-hull}}}.



\section{Positivity of Relative Pluricanonical Bundles and of their Direct Images}
\label{sec_pos-im-dir}

Let $f:X\to Y$ be a Kähler fibre space between complex manifolds, that is, a proper surjective morphism with connected fibres (an analytic fibre space) such that locally over $Y$ the complex manifold $X$ is Kähler (c.f. the definitions in {\hyperref[sec_intro]{Introduction}}). Let $(L,h_L)$ be a line bundle on $X$ equipped with a singular Hermitian metric $h_L$ whose curvature current $\Theta_{h_L}(L)$ is positive. The main purpose of this section is to establish the positivity result for the $L$-twisted relative pluricanonical bundles and their direct images mentioned in {\hyperref[sec_intro]{Introduction}} (c.f. {\hyperref[thm_Cao_Bergman]{Theorem \ref*{thm_Cao_Bergman}}} and {\hyperref[thm_pos_im-dir]{Theorem \ref*{thm_pos_im-dir}}}). To this end, we will explain the construction of the relative $m$-Bergman kernel metric $h_{X/Y\!,L}^{(m)}$ on $K_{X/Y}\ptensor[m]\otimes L$ and of the canonical $L^2$ metric $g_{X/Y\!,L}$ on the direct image sheaf $f_{\ast}\left(K_{X/Y}\otimes L\otimes\scrJ(h_L)\right)$.

Let us recall briefly the history of the study of these canonical metrics. Initially, the case with $h_L$ a smooth metric and $f$ smooth is considered in \cite{Ber09}, where the positivity of $f_{\ast}(K_{X/Y}\otimes L)$ is proved by an explicit calculation of the curvature; as a simple consequence, one deduces the positivity of the relative Bergman kernel metric (with $m=1$), c.f. \cite[\S 1, p.~348]{BP08}. In the more general case where $f$ is projective but not necessarily smooth and $f_\ast(K_{X/Y}\otimes L)$ is locally free, 
the positivity of $f_{\ast}(K_{X/Y}\otimes L)$ is proved in \cite[Theorem 3.5]{BP08} based on the work of Berndtsson; this result, is in turn used in \cite[Corollary 4.2]{BP08} to prove the positivity of the relative $m$-Bergman kernel metric under the assumption that the direct image sheaf $f_\ast\left(K_{X/Y}\ptensor[m]\otimes L\right)$ is locally free. 
In \cite{PT18}, these positivity results are established for $f$ projective with the locally freeness conditions for direct images removed: it is made clear that the positivity of the relative $m$-Bergman kernel metric can be regarded as a result of the Ohsawa-Takegoshi extension Theorem with the optimal estimate, and thus can be obtained independent of the positivity of direct images; while the proof of the positivity of $f_{\ast}(K_{X/Y}\otimes L)$ is based on \cite{BP08} and is done by a semistable reduction plus an explicit calculation. A little later, it is realized that the positivity of the canonical metric is also a consequence of the Ohsawa-Takegoshi extension theorem with the optimal estimate, as is explained in \cite{HPS18}. Therefore in order to obtain a Kähler version of this theorem, all one needs is to generalize the Ohsawa-Takegoshi extension theorem to the Kähler case. Thanks to \cite{Cao17}, this result is established and the positivity of the relative $m$-Bergman kernel metric is also proved in \cite{Cao17} as a corollary; in consequence, by virtue of the main result in \cite{Cao17} one can follow the same arguments in \cite{HPS18} to demonstrate the positivity of the canonical $L^2$ metric $g_{X/Y\!,L}$ for $f$ Kähler fibre space. Recently we are informed that this result is established in \cite{DWZZ18} by following the strategy of \cite{HPS18} and by a more general positivity theorem for singular Finsler metrics on direct images. 
For the convenience of the readers, we will nevertheless provide some details of the proof in {\hyperref[ss_pos-im-dir_NS]{\S \ref*{ss_pos-im-dir_NS}}}.

\subsection{Ohsawa-Takegoshi Extension Theorems}
\label{ss_pos-im-dir_OT}
As is explained above, the key point of the proof of {\hyperref[thm_pos_im-dir]{Theorem \ref*{thm_pos_im-dir}}}, like many other results in complex geometry, is the Ohsawa-Takegoshi extension theorem. In this subsection we will state theorems of Ohsawa-Takegoshi type for Kähler fibre spaces in the following two forms:

\paragraph{\quad Local Version} 
For a Kähler fibre space whose base is an open ball in some $\CC^d$, we have the following extension theorem of Ohsawa-Takegoshi type with optimal estimation: 

\begin{thm}[higher dimensional version of {\cite[Theorem 1.1 (Corollary 3.1)]{Cao17}}]
\label{thm_Cao_OT}
Let $p:X\to B$ be a analytic (Kähler) fibre space with $X$ a Kähler manifold and $B\subseteq \CC^d$ the open ball of centre $0$ and of radius $R$. Let $(L,h_L)$ be a holomorphic line bundle on $X$ equipped with $h_L$ a singular Hermitian metric such that the curvature current of $h_L$ is positive. Suppose that $X_0:=p\inv(0)$ is a smooth fibre of $p$, and that $h_L|_{X_0}$ is not identically $+\infty$. Then for any holomorphic section $f\in\Coh^0(X_0,K_{X_0}\otimes L|_{X_0}\otimes\scrJ(h_L|_{X_0}))$, there exists a section $F\in\Coh^0(X,K_X\otimes L)$ such that $F|_{X_0}=f$ and
\[
\frac{1}{\mu(B)}\int_X |F|^2e^{-\phi_L}\leqslant \int_{X_0}|f|^2 e^{-\phi_L},
\]
where $\mu(B)$ denotes the Lebesgue measure of $B$.
\end{thm}
\begin{proof}
We obtain the theorem by applying \cite[Theorem 1.3]{Cao17} to the fibre space $X\xrightarrow{p}B$ with $E=p^\ast\scrO_B^{\oplus d}$\,, $v=p^\ast t$ where $t=(t_1,\cdots, t_d)$ and $t_i$'s are standard coordinates of $\CC^d$, $A=2d\log R$, $c_A(t)\equiv1$, and by letting $\delta\to+\infty$ (c.f. also \cite[\S 4.2, Lemma 4.14]{GZ15}). In particular, when $d=1$ one recovers \cite[Theorem 1.1 (Corollary 3.1)]{Cao17}.
\end{proof}

\paragraph{\quad Global Version}
In many cases, one needs a global version of Ohsawa-Takegoshi extension theorem for Kähler fibre spaces over projective bases; in this case, one cannot obtain an optimal estimation, but one still has an surjection of section spaces up to a twisting by a ample line bundle from the base, along with a weaker estimation on the $L^2$ norm. In fact we have the following: 
\begin{thm}[Kähler version of {\cite[Corollary 2.10]{Den17}}]
\label{thm_Deng_OT}
Let $Y$ be a smooth projective variety of dimension $d$ and let $f:X\to Y$ be a surjective morphism between compact Kähler manifolds with connected fibres. Let $(A_Y,y)$ be any pair with $A_Y$ ample line bundle on $Y$ and $y\in Y_0$ (where $Y_0$ denotes the smooth locus of $f$), such that the Seshadri constant 
\[
\epsilon(A_Y,y)>\dim Y=d. 
\]
Let $(L,h_L)$ be any holomorphic line bundle on $X$ equipped with a singular Hermitian metric $h_L$ whose curvature current is positive, such that $h_L|_{X_y}\not\equiv+\infty$. Then for any section $u\in\Coh^0(X_y,K_{X_y}\otimes L|_{X_y}\otimes\scrJ(h_L|_{X_y}))$, there is a section $\sigma\in\Coh^0(X,K_X\otimes L\otimes f^\ast A_Y)$ such that $\sigma|_{X_y}=u$ with an $L^2$ estimate independent of $L$.   
\end{thm}
For the proof, refer to \cite[Corollary 2.10]{Den17}. Just remark that: in \cite{Den17} this theorem is only stated for $f$ a projective morphism. The above Kähler version holds because the proof of \cite[Corollary 2.10]{Den17} depends only on \cite[(2.8)Theorem]{Dem15} (c.f. also \cite[Theorem 2.9]{Den17}), which is valid for any pseudo-convex Kähler manifold.

\subsection{Positivity of the Relative \texorpdfstring{$m$}{text}-Bergman Kernel Metric}
\label{ss_pos-im-dir_Bergman}
Let $f:X\to Y$ be an analytic fibre space between complex manifolds and let $(L,h_L)$ be a holomorphic line bundle on $X$ equipped with a singular Hermitian metric $h_L$ with curvature current  $\Theta_{h_L}(L)\geqslant 0$. Set $n=\dim X$, $d=\dim Y$ and $e=\dim X-\dim Y=n-d$. Let us recall the construction of the relative $m$-Bergman kernel metric on $K_{X/Y}\ptensor[m]\otimes L$. We will follow \cite[\S 2.1]{CP17} and \cite[\S 3.2]{Cao17}; for more details, c.f. \cite[\S 1]{BP08}. 
 
As in the statement of {\hyperref[thm_pos_im-dir]{Theorem \ref*{thm_pos_im-dir}}}, let $Y_0$ be the (analytic) Zariski open subset of $Y$ over which $f$ is smooth. Let $x\in f\inv Y_0$ and let $z_1,\cdots,z_{d+e}$ the local coordinates near $x$; write $y=f(x)\in Y_0$ and let $t_1,\cdots,t_d$ be local coordinates near $y$ such that $z_{j+e}=f^\ast t_j$. Suppose in addition that over the coordinate neighbourhood of $x$ (resp. of $y$) chosen as above the line bundle $L$ as well as the canonical bundles of $X$ are trivial (resp. the canonical bundle of $Y$ is trivial). 

 
Suppose that $f_\ast\left(K_{X/Y}\ptensor[m]\otimes L\right)\neq 0$. We define a $L^{2/m}$-Finsler norm on $\Coh^0(X_y,K_{X_y}\ptensor[m]\otimes L|_{X_y})$ by taking the integral over the fibre 
\begin{equation}
\label{eq_def-norm-fibre}
\|u\|_{m,y,L}^{\frac{2}{m}}:=\int_{X_y}|u|^{\frac{2}{m}}e^{-\frac{1}{m}\phi_L}
\end{equation}
(we authorize this to be $+\infty$, which is the case when $h_L|_{X_y}\equiv+\infty$). In addition, we denote by $F_u$ the coefficient of $(dz_1\wedge\cdots\wedge dz_{d+e})\ptensor[m]$ in the local expression of $u\wedge f^\ast(dt_1\wedge\cdots\wedge t_d)\ptensor[m]$. Then local weight $\phi^{(m)}_{X/Y\!,L}$ of the relative $m$-Bergman kernel metric $h^{(m)}_{X/Y\!,L}$ is given by
\begin{equation}
\label{eq_def-Bergman}
e^{\phi_{X/Y\!,L}^{(m)}(x)}=\sup_{\|u\|_{m,y,L}\leqslant 1}|F_u(x)|^2\,,
\end{equation}
where $\phi_L$ denotes the local weight of the metric $h_L$. Let us remark that if $h_L|_{X_y}\equiv+\infty$ \eqref{eq_def-Bergman} is equal to $0$ by convention and thus $\phi_{X/Y\!,L}(x)=-\infty$. 
The metric $h^{(m)}_{X/Y\!,L}=e^{-\phi^{(m)}_{X/Y\!,L}}$ can also be described in an intrinsic way as follows: for $\xi\in (K_{X/Y}\ptensor[(-m)]\otimes L\inv)_x$, we have
\[
|\xi|_{h^{(m)\ast}_{X/Y\!,L},x}=\sup_{\|u\|_{m,y,L}\leqslant 1}|\xi(u(x))|\,.
\]
 
 
Suppose in the sequel of this subsection that $f$ is a Kähler fibre space with $X$ and $Y$ complex manifolds. By using the Ohsawa-Takegoshi extension theorem with optimal estimate (c.f. {\hyperref[thm_Cao_OT]{Theorem \ref*{thm_Cao_OT}}}) Junyan Cao proved in \cite{Cao17} that the relative $m$-Bergman kernel metric constructed above is semipositively curved (the Kähler hypothesis on $X$ and $Y$ in the original statement of \cite[Theorem 3.5]{Cao17} is in fact not necessary, and can be replaced by the hypothesis that $f$ is a Kähler fibre space, in which case $X$ is only assumed to be Kähler locally over $Y$):
\begin{thm}[{\cite[Theorem 1.2 (Theorem 3.5)]{Cao17}}]
\label{thm_Cao_Bergman}
Let $f: X\to Y$ be a Kähler fibre space with $X$ and $Y$ complex manifolds and $(L,h_L)$ be a holomorphic line bundle on $X$ equipped with a singular Hermitian metric $h_L$ whose curvature current is positive. Let $m$ be a positive integer. Suppose that for a general point $y_0\in Y$ there exists a non-zero section $u\in\Coh^0(X_{y_0},K_{X_{y_0}}\ptensor[m]\otimes L|_{X_{y_0}})$ satisfying
\[
\int_{X_{y_0}}|u|^{\frac{2}{m}}e^{-\frac{1}{m}\phi_L}<+\infty\,,
\]
then the curvature current of the relative $m$-Bergman kernel metric $h_{X/Y\!,L}^{(m)}$ is positive. More precisely, there is an (analytic) Zariski open subset of $f\inv Y_0$ (c.f. {\hyperref[rmq_thm_Cao_Bergman]{Remark \ref*{rmq_thm_Cao_Bergman}}} below) such that the local weight $\phi_{X/Y\!,L}^{(m)}$ of the metric $h_{X/Y\!,L}^{(m)}$ defined above is a psh function uniformly bounded from above, thus it admits a unique (psh) extension on $X$.
\end{thm}

\begin{rmq}
\label{rmq_thm_Cao_Bergman}
Define for every (quasi-)psh function $\phi$ and for every integer $m>0$ the ideal sheaf $\scrJ_m(\phi)$ by taking
\[
\scrJ_m(\phi)_x:=\left\{\,f\in\scrO_{X,x}\,\bigg|\,|f|^{\frac{2}{m}}e^{-\frac{1}{m}\phi}\in L_{\text{loc}}^1\,\right\},
\]
which is proved to be coherent in \cite{Cao17}. Then the integrability condition in {\hyperref[thm_Cao_Bergman]{Theorem \ref*{thm_Cao_Bergman}}} is equivalent to the non-vanishing condition that $f_\ast(K_{X/Y}\ptensor[m]\otimes L\otimes\scrJ_m(h_L))\neq 0$. Moreover, the open subset mentioned in {\hyperref[thm_Cao_Bergman]{Theorem \ref*{thm_Cao_Bergman}}} can be taken to be $f\inv U$ where $U\subseteq Y_0$ is the (analytic) Zariski open subset consist of all point $t\in U$ such that
\[
\dimcoh^0(X_t,(K_{X/Y}\ptensor[m]\otimes L\otimes\scrJ_m(h_L)|_{X_t})=\rank f_\ast(K_{X/Y}\ptensor[m]\otimes L\otimes\scrJ_m(h_L)). 
\]
\end{rmq}

By an explicit local calculation as in \cite[Theorem 2.3]{CP17} or \cite[3.33.Theorem]{Pau16} we obtain (in virtue of {\hyperref[thm_Cao_Bergman]{Theorem \ref*{thm_Cao_Bergman}}} the proof in \cite{CP17} apparently does not require the projectivity of $f$):

\begin{prop}[Kähler version of {\cite[Remark 2.5]{CP17}} or {\cite[3.35.Remark]{Pau16}}]
\label{prop_Bergman-discriminant}
Let $f: X\to Y$ be a Kähler fibre space with $X$ and $Y$ complex manifolds and $(L,h_L)$ be a holomorphic line bundle on $X$ equipped with a singular Hermitian metric $h_L$ whose curvature current is positive. Let $m$ be a positive integer. Suppose that for a general point $y_0\in Y$ there exists a non-zero section $u\in\Coh^0(X_{y_0},K_{X_{y_0}}\ptensor[m]\otimes L|_{X_{y_0}})$ satisfying
\[
\int_{X_{y_0}}|u|^{\frac{2}{m}}e^{-\frac{1}{m}\phi_L}<+\infty\,,
\]
(as in the hypothesis of {\hyperref[thm_Cao_Bergman]{Theorem \ref*{thm_Cao_Bergman}}}). Then we have
\begin{equation}
\label{eq_Bergman>disc+L}
\Theta_{h_{X/Y\!,L}^{(m)}}(K_{X/Y}\ptensor[m]\otimes L)\geqslant m[\Sigma_f]
\end{equation}
in the sense of currents, where the divisor is defined as follows: let $\Sigma_Y$ be a divisor containing $Y\backslash Y_0$ and write 
\[
f^\ast\Sigma_Y=\sum_{i\in I}b_i W_i
\]
with the $b_i$'s positive integers and $W_i$'s prime divisors over $X$, then we define
\[
\Sigma_f:=\sum_{i\in I^{\divisor}}(b_i-1)W_i
\]
with $I^{\divisor}$ the set of indices $i\in I$ such that $f(W_i)$ is a divisor over $Y$. This definition is clearly does not depends on the choice of $\Sigma_Y$. In particular, the current $\Theta_{h_{X/Y\!,L}^{(m)}}(K_{X/Y}\ptensor[m]\otimes L)$ is singular along the multiple fibres of $f$ in codimension $1$.
\end{prop}
\begin{proof}
Let us remark that in \cite{CP17} the proof of inequality \eqref{eq_Bergman>disc+L} is only sketched for $m=1$. For the convenience of the readers let us give a detailed proof for the general case here. Since a positive $(1,1)$-current extends across analytic subsets of codimension $2$, it suffices to check the inequality around a general point of $W_i$ for every $i\in I_{\divisor}$ (so that one can assume that every $W_i$ is smooth). Say $i=1\in I_{\divisor}$, and let $x$ be a general point of $W_1$. Take a small ball $B_y$ (of radius $<1$) around $y=f(x)$ with holomorphic local coordinates $(t_j)_{j=1,\cdots,d}$ and a small ball $\Omega_x\subset f\inv(B_y)$ around $x_0$ with holomorphic local coordinates $(z_i)_{i=1,\cdots,n}$, such that $W_1$ is locally defined by the equation $z_{e+1}=0$ and that $f(W_1)$ is defined by $t_1=0$. Then $f$ is locally given by the formula (up to reordering the indices):
\[
(z_1\,,\cdots,z_e\,,z_{e+1}\,,\cdots,z_n)\longmapsto(z_{e+1}^{b_1}\,,z_{e+2}\,,\cdots,z_n).
\] 

Now let $y_0\in B_y\backslash(t_1=0)$, and let $u\in\Coh^0(X_{y_0}, K_{X_{y_0}}\ptensor[m]\otimes L|_{X_{y_0}})$ satisfying the $L^{2/m}$ condition as in the hypothesis; up to a normalization one can suppose that $\|u\|_{m,y_0,L}=1$. Then by the construction of $F_u$ we have
\[
1=\|u\|^{\frac{2}{m}}_{m,y_0,L}=\int_{X_{y_0}}|u|^{\frac{2}{m}}e^{-\frac{1}{m}\phi_L}\geqslant\int_{\Omega_x\cap X_{y_0}}\left|\frac{F_u}{z_{e+1}^{m(b_1-1)}}\right|^{\frac{2}{m}}d\mu_{X_{y_0}}
\]
where $d\mu_{X_{y_0}}$ is the Lebesgue measure on $X_{y_0}$ with respect to the $z_i$'s. Notice that
\[
\Omega_x\cap X_{y_0}=\left\{z_{e+1}^{b_1}=t_1(y_0)\,, z_{e+i}=t_i(y_0)\,, 2\leqslant i\leqslant d\right\},
\]
hence by applying the Ohsawa-Takegoshi type extension theorem \cite[Theorem 4.2]{DWZZ18} (or \cite[0.2.Proposition]{BP10}) to $\Omega=\Omega_x$\,, $p=(z_{e+1}\,,\cdots,z_n)$ and $\phi=(b_1-1)\log|z_{e+1}|^2$, the holomorphic function $F_u$ extends to a function $G_u$ defined on $\Omega_x$ satisfying the following $L^{2/m}$-integrability condition: 
\[
\int_{\Omega_x}\left|\frac{G_u}{z_{e+1}^{m(b_1-1)}}\right|^{\frac{2}{m}}d\mu_X\leqslant \mu(B_y)
\]
By valuative integrability criterion \cite[Theorem 10.11]{Bou16} the generic Lelong number of $\log|G_u|$ over $W_1$ is strictly superior to $m(b_1-1)$, implying that
\[
\log|G_u|^2\leqslant m(b_1-1)\log|z_{e+1}|^2+C_{y_0}
\]
for some uniform (the section space $\Coh^0(X_{y_0}\,,K_{X_{y_0}}\ptensor[m]\otimes L|_{X_{y_0}})$ being finite-dimensional) constant $C_{y_0}$ depending on $y_0$. Hence by the construction \eqref{eq_def-Bergman} we have
\[
\phi_{X/Y\!,L}^{(m)}(z)\leqslant m(b_1-1)\log|z_{e+1}|^2+C_{f(z)}\,;
\]
by the mean-value inequality the constant $C_{f(z)}$ can be chosen locally uniform, which proves \eqref{eq_Bergman>disc+L}. 
\end{proof}

\subsection{Positivity of the Canonical \texorpdfstring{$L^2$}{text} Metric on the Direct Image Sheaf}
\label{ss_pos-im-dir_NS}

In this subsection, let $f:X\to Y$ be an analytic fibre space between complex manifolds and let $(L,h_L)$ be a holomorphic line bundle on $X$ equipped with a singular Hermitian metric $h_L$ with curvature current  $\Theta_{h_L}(L)\geqslant 0$. We will show in the sequel that the canonical $L^2$ metric on the direct image sheaf $f_\ast(K_{X/Y}\otimes L\otimes\scrJ(h_L))$ is semipositively curved, that is, to prove the following theorem: 

\begin{thm}[Kähler version of {\cite[Theorem 1(b)]{PT18}}]
\label{thm_pos_im-dir}
let $f:X\to Y$ be an analytic fibre space between complex manifolds and let $(L,h_L)$ be a holomorphic line bundle on $X$ equipped with a singular Hermitian metric $h_L$ with curvature current  $\Theta_{h_L}(L)\geqslant 0$. Then the torsion free sheaf $f_\ast\left(K_{X/Y}\otimes L\otimes\scrJ(h_L)\right)$ admits a canonical semipositively curved singular Hermitian metric $g_{X/Y\!,L}$ which satisfies the $L^2$ extension property.
\end{thm}

The argument is very close to that in \cite[\S 22-24]{HPS18}. For the convenience of the readers, we will nevertheless explain it in details. 
 
First recall the construction of the canonical $L^2$ metric on the direct image of the adjoint line bundle (twisted by the multiplier ideal). Briefly speaking, it is done as following: when $Y=\text{pt}$, then $X$ is compact, and this is nothing other than the natural $L^2$ norm on $\Coh^0(X,K_X\otimes L\otimes\scrJ(h_L))$; for the general case, we just do this construction in family. 

Precisely, $g_{X/Y\!,L}$ is constructed as following: let $y\in Y_0$, take a coordinate neighbourhood $B$ of $y$, so that $K_Y$ is trivial over $B$, then there is a nowhere vanishing holomorphic $d$-form $\eta$ such that $K_B\simeq \scrO_B\cdot\eta$. For any section $u\in\Coh^0(B,f_\ast(K_{X/Y}\otimes L)\otimes\scrJ(h_L))$, one can regard it as a morphism of $\scrO_B$-modules (in virtue of the projection formula)
\[
u:K_B \longrightarrow \left. f_\ast(K_X\otimes L\otimes\scrJ(h_L))\right|_B
\]
Thus we obtain a section $u(\eta)\in\Coh^0(B,f_\ast(K_X\otimes L\otimes\scrJ(h_L)))=\Coh^0(f\inv B, K_X\otimes L\otimes\scrJ(h_L))$. Locally over $f\inv(B\cap Y_0)$ we can write $u(\eta)=\sigma_u\wedge f^\ast\eta$; whilst the choice of $\sigma_u$ depends on $\eta$, its restriction to the fibre $\sigma_u|_{X_y}$ does not. The local sections $\sigma_u|_{X_y}$'s glue together to give rise to a section $\sigma_{u,y}\in\Coh^0(X_y,K_{X_y}\otimes (L\otimes\scrJ(h_L))|_{X_y})$. Then we define the canonical $L^2$ metric as following: for two local sections $u,v$ of $f_\ast(K_{X/Y}\otimes L)$ (resp. of $f_\ast(K_{X/Y}\otimes L\otimes\scrJ(h_L))$), define
\begin{equation}
\label{eq_def-metrique-NS}
g_{X/Y\!,L}(u,v)(y)=\left(\sqrt{-1}\right)^{n^2}\int_{X_y}\sigma_{u,y}\wedge\bar\sigma_{v,y}e^{-\phi_L}\,.
\end{equation}

Before proving the result, let us fix some notations for later use:
\paragraph{Notations:}
\label{notation_Y1}
Set $Y_1$ the (analytic) Zariski open subset of $Y_0$ such that 
\begin{itemize}
\item[(i)]\label{hyp_Y1_1} $f_\ast(K_{X/Y}\otimes L\otimes\scrJ(h_L))$ and the quotient sheaf of $f_\ast(K_{X/Y}\otimes L)$ by $f_\ast(K_{X/Y}\otimes L\otimes\scrJ(h_L))$ are both locally free over $Y_1$\,;
\item[(i\!i)]\label{hyp_Y1_2} $f_\ast(K_{X/Y}\otimes L)$ satisfies the base change property over $Y_1$\,, i.e. $f_\ast(K_{X/Y}\otimes L)\otimes\kappa(y)\simeq \Coh^0(X_y,K_{X_y}\otimes L|_{X_y})$ for every $y\in Y_1$. (e.g. if the function $y\mapsto\dimcoh^0(X_y,K_{X_y}\otimes L|_{X_y})$ is locally constant on $Y_1$, c.f. \cite[Theorem 1.4(3), p.~6]{Uen75}).
\end{itemize}
Set in addition $\scrG_L:=f_\ast(K_{X/Y}\otimes L\otimes\scrJ(h_L))$. With these notations we have the following auxiliary lemma:
\begin{lemme}
\label{lemme_inclusion-sur-Y1}
We have inclusions
\[
\Coh^0(X_y,K_{X_y}\otimes L|_{X_y}\otimes\scrJ(h_L|_{X_y}))\subseteq\scrG_L\otimes\kappa(y)\subseteq f_\ast(K_{X/Y}\otimes L)\otimes\kappa(y)=\Coh^0(X_y,K_{X_y}\otimes L|_{X_y})
\]
for every $y\in Y_1$.
\end{lemme}
 \begin{proof}
The second inclusion is simply a consequence of the condition {\hyperref[hyp_Y1_1]{(i)}} in the definition of $Y_1$, whereas the first one is deduced from the Ohsawa-Takegoshi type extension {\hyperref[thm_Cao_OT]{Theorem \ref*{thm_Cao_OT}}}. Let us explain precisely the proof of the first inclusion: let $y\in Y_1$, if $h_L|_{X_y}\equiv+\infty$, then $\Coh^0(X_y,K_{X_y}\otimes L|_{X_y}\otimes\scrJ(h_L|_{X_y}))=0$, and the inclusion is automatically established; thus we can assume $h_L|_{X_y}\not\equiv+\infty$. Take a section $u\in\Coh^0(X_y,K_{X_y}\otimes L|_{X_y}\otimes\scrJ(h_L|_{X_y}))$, then for a small ball $B\subseteq Y$ of centre $y$, by {\hyperref[thm_Cao_OT]{Theorem \ref*{thm_Cao_OT}}} (note that $Y_1\subseteq Y_0$, which implies that $X_y$ is smooth) there exists a section $U\in\Coh^0(f\inv B, K_X\otimes L\otimes\scrJ(h_L))$ such that $U|_{X_y}=u$. $K_B$ being a trivial line bundle, we can write $U=\sigma_U\wedge f^\ast\eta$ where $\eta=t_1\wedge\cdots\wedge t_d$ is a nowhere vanishing holomorphic section of $K_B$ (i.e. a holomorphic $d$-form) and $\sigma_U\in\Coh^0(f\inv B, K_{X/Y}\otimes L\otimes\scrJ(h_L))=\Coh^0(B,\scrG_L)$, hence the result is proved.
\end{proof}
 
For any $y\in Y_1$, since $f_\ast(K_{X/Y}\otimes L)$ satisfies the base change property, the expression of the metric $g_{X/Y\!,L}$ is simpler: for $u\in\scrG_L\otimes\kappa(y)$, $u$ can be regarded as a section in $\Coh^0(X_y,K_{X_y}\otimes (L\otimes\scrJ(h_L))|_{X_y})\subseteq\Coh^0(X_y,K_{X_y}\otimes L|_{X_y})$, and we have
\begin{equation}
\label{eq_formula-Y1-metrique-NS}
|u|^2_{g_{X/Y\!,L},y}=\int_{X_y}|u|^2e^{-\phi_L}.
\end{equation}
In particular, $|u|^2_{g_{X/Y\!,L},y}$ ($y\in Y_1$) is finite if and only if $u\in\Coh^0(X_y,K_{X_y}\otimes L|_{X_y}\otimes\scrJ(h_L|_{X_y}))$. Now let us prove the following result which assures that $g_{X/Y\!,L}$ is well-behaved:

\begin{prop}
\label{prop_metrique_im-dir}
The metric $g_{X/Y\!,L}$ defined above on $f_\ast(K_{X/Y}\otimes L\otimes\scrJ(h_L))$ is measurable, and is non-degenerate and bounded almost everywhere.
\end{prop} 
\begin{proof}
We check successively:
\begin{itemize}
\item[(a)]\label{prop_metrique_im-dir_demo-a} $g_{X/Y\!,L}$ is measurable: this is morally automatic, but we give the details in order to fix some notations for later use. Let $s\in\Coh^0(B,\scrG_L)$ be a local section on $B$ with $B$ a small ball in $Y$, we will show that $\Lambda_s:=|s|_{g_{X/Y\!,L}}^2$ is a measurable function. To this end, we can assume that $B$ is contained in $Y_0$; in addition, $s$ can be regarded as a section in $\Coh^0(f\inv B,K_{X/Y}\otimes L)$, and thus $s(y)\in\Coh^0(X_y,K_{X_y}\otimes L|_{X_y})$; $s\wedge f^\ast\eta\in\Coh^0(f\inv B,K_X\otimes L)$ where $\eta$ is a nowhere vanishing holomorphic $d$-form, giving rise to a trivialization $K_B\simeq\scrO_B\cdot \eta$. By definition, for any $y\in B\cap Y_1$ we have 
\[
\Lambda_s(y)=\int_{X_y}|s(y)|^2e^{-\phi_L},
\]
By Ehresmann's theorem (c.f. for example \cite[\S 9.1.1, Proposition 9.3, pp.~209-210]{Voi02}) we have a diffeomorphism $X_0\times B\xrightarrow{\tau}f\inv B$ such that $\tau|_{X_0\times\{0\}}\circ i_0=\id_{X_0}$ where $i_y:X_0\to X_0\times B$ is the natural inclusion which identifies $X_0$ à $X_0\times\{y\}$ in $X_0\times B$. Then we can write
\begin{equation}
\label{eq_metric-g-section}
\Lambda_s(y)=\int_{X_0}G_s(y,\,\cdot\,)\text{Vol}_{X_0} 
\end{equation}
where $\text{Vol}_{X_0}$ is a fixed volume form on $X_0$ and $G_s$ is a function such that 
\begin{equation}
\label{eq_def-fonction-G}
G_s(y,\,\cdot\,)\text{Vol}_{X_0}=\left|\tau^\ast(s\wedge f^\ast\eta)\big|_{X_0\times\{y\}}\right|^2e^{-\phi_L}.
\end{equation}
$\phi_L$ being a psh function, the function $G_s$ is lower semi-continuous and is well defined on $X_0\times(B\cap Y_1)$, in particular it is measurable. Hence by Fubini's theorem, $\Lambda_s$ is measurable.

\item[(b)]\label{prop_metrique_im-dir_demo-b} $g_{X/Y\!,L}$ is non-degenerate and bounded almost everywhere (c.f. also \cite[3.29.Remark]{Pau16}): First one notices that by the formula \eqref{eq_formula-Y1-metrique-NS} the metric $g_{X/Y\!,L}$ is non-degenerate over $Y_1$ since $\phi_L$ is a psh function. In order to show that $g_{X/Y\!,L}$ is bounded almost everywhere, it suffices to prove:
\begin{lemme}
\label{lemme_Fubini_multiplier-ideal}
The natural inclusion
\[
\Coh^0(X_y,K_{X_y}\otimes L|_{X_y}\otimes\scrJ(h_L|_{X_y})\hookrightarrow\scrG_L\otimes\kappa(y)
\]
is an isomorphism for $y\in Y_1$ almost everywhere. 
\end{lemme}
\begin{proof}[Proof of {\hyperref[lemme_Fubini_multiplier-ideal]{Lemma \ref*{lemme_Fubini_multiplier-ideal}}}]
Take a small neighbourhood $B$ of $y$ in $Y$ (so that there is a trivialization $K_B\simeq\scrO_B\cdot\eta$ with $\eta$ a nowhere vanishing holomorphic $d$-form on $B$). $\scrG_L$ begin coherent, there is a finite family of sections, namely $\{\,s_1,\cdots,s_k\,\}$, which generate $\scrG_L$ over $B$; the $s_i$'s can be regarded as sections in $\Coh^0(f\inv B,K_{X/Y}\otimes L)$, satisfying the following $L^2$ integrability condition: for every coordinate open subset $U$ in $f\inv B$ we have
\[
\int_{U}\left|s_i\wedge f^\ast\eta\right|^2e^{-\phi_L}<+\infty.
\]
In particular, $\left|s_i\wedge f^\ast\eta\right|^2e^{-\phi_L}\in L^1_{\text{loc}}(f\inv B)$ (with an abuse of notations), hence Fubini's theorem implies that $\left|s_i|_{X_y}\right|^2e^{-\phi_L}\in L^1_{\text{loc}}(X_y)$ for $y\in B$ almost everywhere, which proves the result.
\end{proof}
\end{itemize}
\end{proof}
 
By virtue of {\hyperref[prop_metrique_im-dir]{Proposition \ref*{prop_metrique_im-dir}}}, in order to prove that $g_{X/Y\!,L}$ defined above extends to a semipositively curved singular Hermitian metric on $\scrG_L$\,, it remains to show: for $U\subseteq Y$ an open subset, and for $\alpha\in\Coh^0(U,\scrG_L^\ast)$ a non-zero section, $\psi_\alpha:=\log|\alpha|^2_{g_{X/Y\!,L}^\ast}$ (a function well-defined on $U\cap Y_0$) extends to a psh function on $U$. To this end, we will successively establish (by {\hyperref[prop_metrique_im-dir]{Proposition \ref*{prop_metrique_im-dir}}}, $\psi_\alpha\not\equiv-\infty$ on $U\cap Y_0$):
\begin{itemize}
\item[(A)]\label{pos_im-dir_A} $\psi_\alpha$ is locally uniformly bounded from above on $U_1:=U\cap Y_1$; \item[(B)]\label{pos_im-dir_B} $\psi_\alpha$ is upper semi-continuous on $U_1$\,;
\item[(C)]\label{pos_im-dir_C} $\psi_\alpha$ satisfies the mean value inequality on any disc in $U_1$. 
\end{itemize}
In fact, the points {\hyperref[pos_im-dir_B]{(B)}} and {\hyperref[pos_im-dir_C]{(C)}} imply that $\psi_\alpha$ is a psh function over $U_1$; and the point {\hyperref[pos_im-dir_A]{(A)}} implies moreover that $\psi_\alpha|_{U_1}$ admits a unique psh extension to $U$. In addition, let us remark that up to replacing $Y$ par $U$, one can suppose that $\alpha$ is a global section; in this case $\psi_\alpha$ is a function well defined over $Y_0$. The proof of theses three points relies on the Ohsawa-Takegoshi type extension {\hyperref[thm_Cao_OT]{Theorem \ref*{thm_Cao_OT}}}, which permits us to extend a section on the fibre to a neighbourhood along with an $L^2$ estimate (in some cases we should require this estimate to be optimal). 
 
\paragraph{Demonstration of {\hyperref[pos_im-dir_A]{(A)}}:}
Let $y_0\in Y$, we will prove that $y_0$ admits a neighbourhood on whose intersection with $Y_1$ the function $\psi_\alpha$ is uniformly upper bounded. To this end, take a small open ball $B_0$ of centre $y_0$ in $Y$ and denote $B_1:=\frac{1}{2}B_0$, $B=B_2:=\frac{1}{4}B_0$ and $R_0=$ radius of $B_0$. We will prove in the sequel that $\psi_\alpha$ est uniformly upper bounded on $B\cap Y_1$. This proceeds in two steps:
\begin{itemize}
\item[(A1)]\label{pos_im-dir_demoA1} Firstly we prove that
\begin{equation}
\label{eq_demo_a1}
\psi_\alpha|_{B\cap Y_1}\leqslant\text{ punctual supremum of the family of functions }\left\{\log|\alpha(s)|^2\right\}_{s\in S_{M_0}}
\end{equation}
where $S_{M_0}$ denotes the set of sections $s\in\Coh^0(B_1,\scrG_L)=\Coh^0(f\inv B_1, K_{X/Y}\otimes L\otimes\scrJ(h_L))$ satisfying the following $L^2$ condition:
\begin{equation}
\label{eq_cond_s-psi}
\int_{f\inv B_1}\left|s\wedge f^\ast\eta|_{B_1}\right|^2e^{-\phi_L}\leqslant \left(\frac{3}{4}\right)^d\!\mu(B_0):=M_0,
\end{equation}
where $\mu(B_0)$ denotes the Lebesgue measure of $B_0$ and $\eta$ a nowhere vanishing holomorphic $n$-form on $B_0$ (which gives rise to a trivialization $K_{B_0}\simeq\scrO_{B_0}\cdot\eta$). 

For every $y\in B\cap Y_1$ such that $h_L|_{X_y}\not\equiv+\infty$ (if $h_L|_{X_y}\equiv+\infty$, then $\psi_\alpha(y)=-\infty$ and \eqref{eq_demo_a1} is automatically established at $y$), we have 
\[
\psi_\alpha(y)=\log\left|\alpha(y)\right|_{g_{X/Y\!,L}^\ast,y}^2=\sup_{\|u\|_{y,L}\leqslant 1}\log\left|\alpha(y)(u)\right|^2\,.
\]
The set $\left\{\,u\in\Coh^0(X_y,K_{X_y}\otimes L|_{X_y})\,\big|\,\|u\|_{y,L}\leqslant 1\right\}$ being compact, the supremum is attained by a vector $v_y\in\scrG_L\otimes\kappa(y)$ satisfying $\|v_y\|_{y,L}=|v_y|_{g_{X/Y\!,L},y}=1$ (we denote $\|\cdot\|_{1,y,L}=\|\cdot\|_{y,L}$, compare \eqref{eq_def-norm-fibre} and \eqref{eq_formula-Y1-metrique-NS}); in particular $v_y\in\Coh^0(X_y,K_{X_y}\otimes L|_{X_y}\otimes\scrJ(h_L|_{X_y}))$. Consider the open ball $B_y:=B(y,\frac{3}{4}R_0)$ of centre $y$ and of radius $=\frac{3}{4}R_0$. Then $B\subseteq B_1\subseteq B_y\subseteq B_0$. By {\hyperref[thm_Cao_OT]{Theorem \ref*{thm_Cao_OT}}} we get a section $s_y\in\Coh^0(B_y,\scrG_L)$ such that $s_y|_{X_y}=v_y$ and satisfies the following $L^2$ condition:
\[
\int_{f\inv B_y}\left|s_y\wedge f^\ast\eta|_{B_y}\right|^2e^{-\phi_L}\leqslant \mu(B_y)\cdot\|v_y\|_{y,L}=\mu(B_y)=\left(\frac{3}{4}\right)^d\!\mu(B_0)=M_0\,.
\]
In particular, $s_y|_{B_1}$ satisfies the condition \eqref{eq_cond_s-psi}, then $s_y|_{B_1}\in S_{M_0}$. In addition, we have \[\psi_\alpha(y)=\left(\log\left|\alpha(s_y)\right|^2\right)(y),\]
which proves \eqref{eq_demo_a1}.
 
\item[(A2)]\label{pos_im-dir_demoA2} By the previous step, it remains to prove that the functions $\log\left|\alpha(s)\right|^2$ ($s\in S_{M_0}$) are all uniformly upper bounded over $\bar B$ by a uniform constant. In fact we can prove the following more general: 
\begin{lemme}
\label{lemme_hol-L2-compact}
For a fixed $M\geqslant 0$, define
\[
S_M:=\left\{\,s\in\Coh^0(B_1,\scrG_L)=\Coh^0(f\inv B_1,K_{X/Y}\otimes L\otimes\scrJ(h_L))\,\bigg|\,\int_{f\inv B_1}\left|s\wedge f^\ast\eta|_{B_1}\right|^2e^{-\phi_L}\leqslant M\,\right\},
\]
then for every compact $K\subseteq B_1$, there exists a constant $C_K\geqslant 0$ (independent of $s$) such that \[
\sup_{K}\left|\alpha(s)\right|\leqslant C_K
\]
for every $s\in S_M$.
\end{lemme}
\begin{proof}
The lemma is deduced from the following well known facts about the Fréchet space structure on the cohomology spaces of coherent sheaves over compact complex :
\begin{itemize}
\item[(a)]\label{lemme_hol-L2-compact_demo-a} The section $\alpha$, regarded as a morphism $\scrG_L\to\scrO_Y$, induces continuous map between Fréchet spaces
\[
\alpha|_{B_1}: \Coh^0(B_1,\scrG_L)\longrightarrow\Coh^0(B_1,\scrO_X),
\]
where the topology on the two vector spaces are induced by the uniform convergence on all the compacts (abbr. compact convergence). In general, one can endow the section space of any coherent sheaf with such a topology. Moreover, although the topologies on the two spaces $\Coh^0(B_1,\scrG_L)$ and $\Coh^0(f\inv B_1,K_{X/Y}\otimes L\otimes\scrJ(h_L))$ are different a priori, they are in fact homeomorphic under the usual identification.  
\item[(b)]\label{lemme_hol-L2-compact_demo-b} $S_M\subseteq\Coh^0(B_1,\scrG_L)$ is a compact with respect to the topology of compact convergence. This is a result of Montel's Theorem.
\item[(c)]\label{lemme_hol-L2-compact_demo-c} The compacts in $\Coh^0(B_1,\scrO_X)$ are closed and bounded, c.f.\cite[\S V.4.2, Proposition 2.1, pp.~165-166]{Car61}.
 \end{itemize}
 \end{proof}
 \end{itemize}
 
\paragraph{Demonstration of {\hyperref[pos_im-dir_B]{(B)}}:}
Let $y_0\in Y_1$, and let $\{y_k\}_{k>0}$ be any sequence in $Y_1$ convergent to $y_0$, we will prove that
\[
\limsup_{k\to +\infty}\psi_\alpha(y_k)\leqslant\psi_\alpha(y_0).
\]
The problem being local, we can replace $Y$ by $B_0$ a small open ball of centre $y_0$ ($y_0=0$ in $B_0$) in $Y$. Note $R_0:=$ the radius of $B_0$ and $B_i:=\left(\frac{1}{2}\right)^iB_0$. Since there is a subsequence of $\{\psi_\alpha(y_k)\}_{k>0}$ which converges to the limit superior of $\{\psi_\alpha(y_k)\}_{k>0}$, we can assume that the sequence $\{\psi_\alpha(y_k)\}_{k>0}$ is convergent. In addition, up to shifting the numbering of the sequence we can assume that $\{y_k\}_{k>0}\subseteq B_3$; we can also assume that $\psi_\alpha(y_k)\neq-\infty,\;\forall k$ (in particular, $h_L|_{X_{y_k}}\not\equiv+\infty$). As in the step {\hyperref[pos_im-dir_demoA1]{(A1)}} above, there exists for every $k\in\ZZ_{>0}$ a vector $v_k\in\Coh^0(X_{y_k},K_{X_{y_k}}\otimes L|_{X_{y_k}}\otimes\scrJ(h_L|_{X_{y_k}}))$ such that $\|v_k\|_{y_k,L}=1$ and
\[
\psi_\alpha(y_k)=\log\left|\alpha(y_k)(v_k)\right|^2.
\]   
Consider $B_{y_k}:=B(y_k,\frac{7}{8}R_0)$ the open ball of centre $y_k$ and of radius $\frac{7}{8}R_0$, then $B_3\subseteq B_2\subseteq B_1\subseteq B_{y_k}\subseteq B_0$. Still by {\hyperref[thm_Cao_OT]{Theorem \ref*{thm_Cao_OT}}}, we obtain a section $s_k\in\Coh^0(B_{y_k},\scrG_L)=\Coh^0(f\inv B_{y_k},K_{X/Y}\otimes L\otimes\scrJ(h_L))$ such that $s_k|_{X_{y_k}}=v_k$ and
\[
\int_{f\inv B_{y_k}}|s_k|^2e^{-\phi_L}\leqslant\left(\frac{7}{8}\right)^d\!\mu(B_0):=M'_0.
\]
Denote $F_k=\alpha(s_k)|_{B_1}$ and $\theta_k:=\log\left|F_k\right|^2$, then $F_k$ is a holomorphic function on $B_1$ and $\theta_k$ is a psh function (with analytic singularities); in addition, we have that $\psi_\alpha(y_k)=\theta_k(y_k)$. By {\hyperref[lemme_hol-L2-compact]{Lemma \ref*{lemme_hol-L2-compact}}} (taking $M=M'_0$ and $K=\bar B_2$), there is a constant $C_{\bar B_2}$ independent of $k$ such that $|F_k|\leqslant C_{\bar B_2}$ on $\bar B_2$ for every $k$; in consequence, the derivatives of $F_k$ satisfy 
\[
|\nabla F_k|^2\leqslant \tilde{C}_{\bar B_2}:=\frac{16\sqrt{n}}{R_0}C_{\bar B_2}
\]   
on $\bar B_3$ (c.f. \cite[\S V.1.2, Lemme, p.~146]{Car61}). In particular, since $\{y_k\}_{k>0}\subseteq B_3$, we have
\[
\left||F_k(0)|-|F_k(y_k)|\right|\leqslant\left|F_k(0)-F_k(y_k)\right|\leqslant\tilde{C}_{\bar B_2}|y_k-0|\to 0\text{ when }k\to+\infty,
\]
hence we get 
\begin{equation}
\label{eq_psi-theta}
\lim_{k\to+\infty}\theta_k(y_k)=\lim_{k\to+\infty}\left(\log|F_k(y_k)|\right)=\lim_{k\to+\infty}\left(\log|F_k(0)|\right)=\lim_{k\to+\infty}\theta_k(0)
\end{equation}
By definition, we have 
\[|\alpha(s_k)|\leqslant|\alpha|_{g_{X/Y\!,L}^\ast}\left|s_k\right|_{g_{X/Y\!,L}}\quad\Rightarrow\quad\psi_\alpha+\log\lambda_k\geqslant\theta_k\,,\]
where $\lambda_k:=\Lambda_{s_k}=|s_k|_{g_{X/Y\!,L}}^2$. By passing to the limit superior we obtain (in virtue of \eqref{eq_psi-theta})
\[
\psi_\alpha(0)+\limsup_{k\to+\infty}\left(\log\lambda_k(0)\right)\geqslant\limsup_{k\to+\infty}\theta_k(0)=\lim_{k\to+\infty}\theta_k(0)=\lim_{k\to+\infty}\theta_k(y_k)=\lim_{k\to+\infty}\psi_\alpha(y_k).
\]
It remains thus to show
\[
\limsup_{k\to+\infty}\left(\log\lambda_k(0)\right)\leqslant 0,
\]
and this amounts to show (the function $\log$ being increasing and continuous)
\[
\limsup_{k\to+\infty}\lambda_k(0)\leqslant 1.
\]
Now up to taking an extraction, we can assume that the sequence $\{\lambda_k(0)\}_{k>0}$ is convergent. By the compacity of $S_{M'_0}$ (the point {\hyperref[lemme_hol-L2-compact_demo-b]{(b)}} in the proof of {\hyperref[lemme_hol-L2-compact]{Lemma \ref*{lemme_hol-L2-compact}}}), up to taking a subsequence, we can further assume that $\{s_k\}_{k>0}$ converges uniformly on all compacts in $B_1$ to a section $s\in S_{M'_0}$. By \eqref{eq_metric-g-section} (c.f. point {\hyperref[prop_metrique_im-dir_demo-a]{(a)}} in the proof of {\hyperref[prop_metrique_im-dir]{Proposition \ref*{prop_metrique_im-dir}}}) we have for $y\in B_1\cap Y_1$ that 
\begin{align*}
\lambda_k(y) &= \int_{X_0}G_{s_k}(y,\,\cdot\,)\text{Vol}_{X_0}\,, \\
\Lambda_s(y) &= \int_{X_0}G_s(y,\,\cdot\,)\text{Vol}_{X_0}.
\end{align*}
By \eqref{eq_def-fonction-G} the compact convergence $\{s_k\}_{k>0}$ implies that $\left\{G_{s_k}\right\}_{k>0}$ converges uniformly over all compacts to $G_s$ (especially over $\bar B_3$).By the point {\hyperref[prop_metrique_im-dir_demo-a]{(a)}} in the proof of {\hyperref[prop_metrique_im-dir]{Proposition \ref*{prop_metrique_im-dir}}}, the $G_{s_k}$'s as well as $G_s$ are all lower semi-continuous functions, thus 
\begin{align*}
G_{s_k}(0,\,\cdot\,) &\leqslant \liminf_{l\to+\infty}G_{s_k}(y_l,\,\cdot\,)\,,\\
G_s(0,\,\cdot\,) &\leqslant \liminf_{l\to+\infty}G_s(y_l,\,\cdot\,)\,,
\end{align*} 
and in consequence (by a diagonal process)
\[
G_s(0,\,\cdot\,)\leqslant\liminf_{k\to+\infty}G_{s_k}(y_k,\,\cdot\,).
\]
Then Fatou's lemma implies that,
\begin{align*}
\lim_{k\to+\infty}\lambda_k(0) &=\Lambda_s(0)=\int_{X_0}G_s(0,\,\cdot\,)\text{Vol}_{X_0}\leqslant\int_{X_0}\liminf_{k\to+\infty}G_{s_k}(y_k,\,\cdot\,)\text{Vol}_{X_0} \\
&\leqslant\liminf_{k\to+\infty}\int_{X_0}G_{s_k}(y_k,\,\cdot\,)=\liminf_{k\to+\infty}\lambda_k(y_k)=1,
\end{align*}
which proves the result.
 
\paragraph{Demonstration of {\hyperref[pos_im-dir_C]{(C)}}:}
Let $\Delta$ be any disc contained in $Y_1$, we will prove that 
\begin{equation}
\label{eq_inegalite-moyenne}
\psi_\alpha(0)\leqslant\frac{1}{\mu(\Delta)}\int_\Delta \psi_\alpha d\mu.
\end{equation}
We can assume that $Y=\Delta\,(=Y_1=Y_0)$, in particular, $f$ is a smooth fibration. If $\psi_\alpha(0)=-\infty$, then the inequality \eqref{eq_inegalite-moyenne} is automatically established; hence we can assume that $\psi_\alpha(0)\neq-\infty$, in particular $h_L|_{X_0}\not\equiv+\infty$. As in the step {\hyperref[pos_im-dir_demoA1]{(A1)}}, there is a section $v\in\Coh^0(X_0, K_{X_0}\otimes L|_{X_0}\otimes\scrJ(h_L|_{X_0}))$ such that $\|v\|_{0,L}=1$ and
\[
\psi_\alpha(0)=\log\left|\alpha(0)(v)\right|^2.
\]
Again by {\hyperref[thm_Cao_OT]{Theorem \ref*{thm_Cao_OT}}} we get a section $s\in\Coh^0(X,K_{X/\Delta}\otimes L\otimes\scrJ(h_L))$ such that $s|_{X_0}=v$ and
\[
\int_\Delta\Lambda_s(t)dt=\int_X |s|^2e^{-\phi_L}\leqslant \mu(\Delta).
\] 
In particular $\left(\log|\alpha(s)|^2\right)(0)=\psi_\alpha(0)$. By definition we have
\[
|\alpha(s)|\leqslant |\alpha|_{g_{X/Y\!,L}^\ast}|s|_{g_{X/Y\!,L}}\quad\Rightarrow\quad\psi_\alpha+\log\Lambda_s\geqslant\log|\alpha(s)|^2.
\]
The function $\log|\alpha(s)|$ being psh on $\Delta$, it satisfies the mean value inequality, hence we have 
\[
\frac{1}{\mu(\Delta)}\int_{\Delta}\psi_\alpha d\mu+\frac{1}{\mu(\Delta)}\int_\Delta\log\Lambda_s d\mu\geqslant\frac{1}{\mu(\Delta)}\int_\Delta\log|\alpha(s)|d\mu\geqslant\left(\log|\alpha(s)|^2\right)(0)=\psi_\alpha(0). 
\]
It remains to show that
\[
\int_{\Delta}\log\Lambda_s d\mu\leqslant 0,
\]
but the function $\log$ being concave, this is a result of Jensen's inequality: $\Lambda_s$ being integrable, we have
\[
\int_\Delta\log\Lambda_s\frac{d\mu}{\mu(\Delta)}\leqslant \log\left(\int_\Delta\Lambda_s\frac{d\mu}{\mu(\Delta)}\right)=\log 1=0.
\]
This proves \eqref{eq_inegalite-moyenne}, and thus finishes the proof of the step {\hyperref[pos_im-dir_C]{(C)}}. Hence $g_{X/Y\!,L}$ is a semipositively curved singular Hermitian metric on $\scrG_L$.

\paragraph{}
In order to finish the proof of {\hyperref[thm_pos_im-dir]{Theorem \ref*{thm_pos_im-dir}}}, it remains to show that $(\scrG_L\,,g_{X/Y\!,L})$ satisfies the $L^2$ extension property. To this end, take an open subset $U$ of $X$ and $Z$ an analytic subset of $U$ and , and take a local section $s\in\Coh^0(U\backslash Z,\scrG_L)$ satisfying the $L^2$ integrability condition, we will show that $s$ extends to a section over $U$. The problem being local, we can replace $U$ by a small ball $B$ in $Y$ (with $t_1,\cdots,t_d$ the standard coordinates). 
Then $s\in\Coh^0(B\cap Y_1,\scrG_L)=\Coh^0(f\inv Y_1,K_{X/Y}\otimes L\otimes\scrJ(h_L))$ satisfies the following $L^2$ condition: 
\[
\int_B\left(|s|_{g_{X/Y\!,L}}^2\right)\eta=\int_{f\inv B} |s\wedge f^\ast\eta|^2e^{-\phi_L}<+\infty,
\]
where $\eta=dt_1\wedge\cdots\wedge dt_d$ is a nowhere vanishing holomorphic $d$-form (giving rise to a trivialization $K_B\simeq\scrO_B\cdot\eta$). Then it is an elementary consequence of Riemann extension that $s$ extends to a section in $\Coh^0(f\inv B,K_{X/Y}\otimes L\otimes\scrJ(h_L))=\Coh^0(B,\scrG_L)$, meaning that $(\scrG_L\,,g_{X/Y\!,L})$ satisfies the $L^2$ extension property. This finishes the proof of {\hyperref[thm_pos_im-dir]{Theorem \ref*{thm_pos_im-dir}}}.

\subsection{Positivity of Direct Images of Twisted Pluricanonical Bundles}
\label{ss_pos-im-dir_pluri-can}

In this subsection, we will apply {\hyperref[thm_Cao_Bergman]{Theorem \ref*{thm_Cao_Bergman}}} and {\hyperref[thm_pos_im-dir]{Theorem \ref*{thm_pos_im-dir}}} to prove {\hyperref[thm_pos_im-dir]{Theorem \ref*{thm_pos_im-dir}}}, which will serve as a key ingredient in the demonstration of our {\hyperref[main-thm]{Main Theorem}}.

\begin{proof}[Proof of {\hyperref[thm_pos-im-pluri-can]{Theorem \ref*{thm_pos-im-pluri-can}}}]
Recall that
\[
\scrF_{m,\Delta}:=f_\ast\left(K_{X/Y}\ptensor[m]\otimes\scrO_X(m\Delta)\right).
\]
If $\scrF_{m,\Delta}=0$, then there is nothing to prove; hence we assume that $\scrF_{m,\Delta}=0$. Since $(X,\Delta)$ is klt (implying that $(X_y,\Delta_y)$ is klt for $y$ general by \cite[\S 9.5.D, Theorem 9.5.35, pp.~210-211, vol.II]{Laz04}) and $\scrF_{m,\Delta}\neq 0$, the condition in the hypothesis of {\hyperref[thm_Cao_Bergman]{Theorem \ref*{thm_Cao_Bergman}}} is satisfied for $L=\scrO_X(m\Delta)$ and $h_L=h_\Delta\ptensor[m]$ where $h_\Delta$ is the canonical (singular) Hermitian metric defined by the local equations of $\Delta$, then we obtain a singular Hermitian metric $h_{X/Y\!,m\Delta}^{(m)}$ over $K_{X/Y}\ptensor[m]\otimes\scrO_X(m\Delta)$ whose curvature current is positive. However one cannot directly apply {\hyperref[thm_pos_im-dir]{Theorem \ref*{thm_pos_im-dir}}} to obtain a semipositively curved singular Hermitian metric on $\scrF_{m,\Delta}$. In order to overcome this difficulty, we introduce the line bundle
\[
L_{m-1}=K_{X/Y}\ptensor[(m-1)]\otimes\scrO_X(m\Delta),
\]
equipped with the metric 
\[
h_{L_{m-1}}:=(h_{X/Y\!,m\Delta}^{(m)})^{\otimes\frac{m-1}{m}}\otimes h_\Delta.
\]
Then the curvature current of $h_{L_{m-1}}$ is positive. We are now ready to apply {\hyperref[thm_pos_im-dir]{Theorem \ref*{thm_pos_im-dir}}} to $L=L_{m-1}$, except that we need to establish in addition that the natural inclusion 
\[
f_\ast\left(K_{X/Y}\otimes L_{m-1}\otimes\scrJ(h_{L_{m-1}})\right)\hookrightarrow\scrF_{m,\Delta}
\]
is generically an isomorphism. 

To this end, let $Y_2$ be the (analytic) Zariski open subset of $Y_0$ satisfying the conditions {\hyperref[hyp_Y1_1]{(i)}}{\hyperref[hyp_Y1_2]{(i\!i)}} in the definition of $Y_1$ for $L=L_{m-1}$ (see {\hyperref[notation_Y1]{Notations}}) and such that the pair $(X_y,\Delta_y)$ is klt for $\forall y\in Y_2$ (c.f. \cite[\S 9.5.D, Theorem 9.5.35, pp.~210-211, vol.II]{Laz04}). By virtue of the base change property of $\scrF_{m,\Delta}$ over $Y_2$ and {\hyperref[lemme_inclusion-sur-Y1]{Lemme \ref*{lemme_inclusion-sur-Y1}}}, it suffices to prove:

\begin{lemme}
\label{lemme_gen-iso_im-pluri-can}
The natural inclusion 
\begin{equation}
\label{eq_incl-mult-ideal}
\Coh^0(X_y,K_{X_y}\otimes L_{m-1}|_{X_y}\otimes\scrJ(h_{L_{m-1}}|_{X_y}))\hookrightarrow\Coh^0(X_y,K_{X_y}\otimes L_{m-1}|_{X_y})
\end{equation}
is an isomorphism (or equivalently, surjective) for $y\in Y_2$.
\end{lemme}
\begin{proof}[Proof of {\hyperref[lemme_gen-iso_im-pluri-can]{Lemma \ref*{lemme_gen-iso_im-pluri-can}}}]
 
 that  Let $v\in\Coh^0(X_y,K_{X_y}\otimes L_{m-1}|_{X_y})=\Coh^0(X_y,K_{X_y}\ptensor[m]\otimes\scrO_{X_y}(m\Delta_y))$, then with the same notations as in {\hyperref[ss_pos-im-dir_Bergman]{\S\ref*{ss_pos-im-dir_Bergman}}} we can write
\[
v\wedge(dt_1\wedge\cdots\wedge dt_d)\ptensor[m]=F_v\cdot(dz_1\wedge\cdots\wedge dz_n)\ptensor[m].
\]
Since $\Delta_y$ is klt, we have
\begin{equation}
\label{eq_klt-int-L2/m}
\|v\|_{m,y,m\Delta}^{\frac{2}{m}}=\int_{X_y}|v|^{\frac{2}{m}}e^{-\phi_\Delta}=\int_{X_y}\left(|F_v|^{\frac{2}{m}}e^{-\phi_\Delta}\right)\text{Vol}_{X_y}<+\infty,
\end{equation}
where $\phi_\Delta$ denotes the local weight of the metric $h_\Delta$. By \eqref{eq_def-Bergman} (c.f. also \cite[\S A.2, p.~8]{BP10}) the local weight $\phi_{X/Y\!,m\Delta}^{(m)}$ satisfies 
\[
\phi_{X/Y\!,m\Delta}^{(m)}=\log\left(\sup_{\|u\|_{m,y,m\Delta}\leqslant 1}|F_u|^2\right)\geqslant \log\left(\frac{|F_v|^2}{\|v\|^2_{m,y,m\Delta}}\right),
\]
and thus
\begin{equation}
\label{eq_bound-Fv}
\log|F_v|^2\leqslant \phi_{X/Y\!,\Delta}^{(m)}+O(1)\quad\Rightarrow\quad |F_v|^{\frac{2(m-1)}{m}}e^{-\frac{m-1}{m}\phi_{X/Y\!,m\Delta}^{(m)}}\leqslant O(1).
\end{equation}
Regarded as a holomorphic $n$-form with values in the line bundle $L_{m-1}|_{X_y}$, the section $v$ satisfies 
\[
|v|^2e^{-\phi_{L_{m-1}}}=|F_v|^2e^{-\frac{m-1}{m}\phi_{X/Y\!,m\Delta}^{(m)}-\phi_\Delta}\cdot\text{Vol}_{X_y},
\]
where $\phi_{L_{m-1}}$ denotes the local weight of the metric $h_{L_{m-1}}$, hence by \eqref{eq_klt-int-L2/m} and \eqref{eq_bound-Fv} we have
\begin{align*}
\|v\|_{y,L_{m-1}}^2 &=\int_{X_y}|v|^2e^{-\phi_{L_{m-1}}}
=\int_{X_y}\left(|F_v|^2e^{-\frac{m-1}{m}\phi_{X/Y\!,\Delta}^{(m)}-\phi_\Delta}\right)\text{Vol}_{X_y} \\
&=\int_{X_y}\left(|F_v|^{\frac{2}{m}}e^{-\phi_\Delta}\right)\cdot\left(|F_v|^{\frac{2(m-1)}{m}}e^{-\frac{m-1}{m}\phi_{X/Y\!,\Delta}^{(m)}}\right)\text{Vol}_{X_y} \\
&\leqslant C\int_{X_y}\left(|F_v|^{\frac{2}{m}}e^{-\phi_\Delta}\right)\text{Vol}_{X_y}=C\cdot\|v\|_{m,y,\Delta}^{\frac{2}{m}}<+\infty,
\end{align*}
where $C$ is a constant given by \eqref{eq_bound-Fv}. Therefore $v\in\Coh^0(X_y,K_{X_y}\otimes L|_{X_y}\otimes\scrJ(h_{L_{m-1}}|_{X_y}))$, which proves the lemma.
\end{proof}
\end{proof}

\begin{rmq}
\label{rmq_klt-gen-iso}
In general, let $N$ be a $\QQ$-line bundle endowed with a semipositively curved singular Hermitian metric $h_N$ such that $\scrJ(h_N)=\scrO_X$. Then by a Fubini type argument as in {\hyperref[lemme_Fubini_multiplier-ideal]{Lemma \ref*{lemme_Fubini_multiplier-ideal}}} we have that $\scrJ(h_N|_{X_y})=\scrO_{X_y}$ for almost every $y\in Y_1$. In consequence, if the direct image sheaf $f_\ast(K_{X/Y}\ptensor[m]\otimes N\ptensor[m])\neq 0$, then we can construct the relative $m$-Bergman kernel metric on $K_{X/Y}\ptensor[m]\otimes N\ptensor[m]$ whose curvature current is positive. Then by the same argument as above, we can prove that the natural inclusion
\[
\Coh^0(X_y,K_{X_y}\otimes N_{m-1}|_{X_y}\otimes\scrJ(h_{N_{m-1}}|_{X_y}))\hookrightarrow\Coh^0(X_y,K_{X_y}\otimes N_{m-1}|_{X_y})
\]
is an isomorphism for almost every (a posteriori, for general) $y\in Y_1$\,, where \[
N_{m-1}:=K_{X/Y}\ptensor[(m-1)]\otimes N\ptensor[m]
\]
and 
\[
h_{N_{m-1}}:=(h_{X/Y\!,mN}^{(m)})\ptensor[\frac{m-1}{m}]\otimes h_N.
\]
In consequence, the natural inclusion
\[
f_\ast\left(K_{X/Y}\otimes N_{m-1}\otimes\scrJ(h_{N_{m-1}})\right)\hookrightarrow f_\ast\left(K_{X/Y}\ptensor[m]\otimes N\ptensor[m]\right) 
\]
is generically an isomorphism.
\end{rmq}

By combining {\hyperref[thm_pos-im-pluri-can]{Theorem \ref*{thm_pos-im-pluri-can}}} and {\hyperref[thm_det-triv_herm-plat]{Theorem \ref*{thm_det-triv_herm-plat}}} we immediately get:

\begin{cor}
\label{cor_im-pluri-can_plat}
Let $f:X\to Y$ and $\Delta$ as in {\hyperref[thm_pos-im-pluri-can]{Theorem \ref*{thm_pos-im-pluri-can}}}. Suppose that the determinant of $\scrF_{m,\Delta}$ is numerically trivial. Then $\left(\scrF_{m,\Delta}\,,g_{X/Y\!,\Delta}^{(m)}\right)$ is a Hermitian flat vector bundle. 
\end{cor}

\section{Log Kähler Version of Results of Kawamata and of Viehweg}
\label{sec_Kah-Viehweg-Kawamata}

In this section we will apply to prove {\hyperref[thm_Viehweg_Iitaka-type-gen]{Theorem \ref*{thm_Viehweg_Iitaka-type-gen}}} and {\hyperref[thm_Viehweg_Iitaka-det-gros]{Theorem \ref*{thm_Viehweg_Iitaka-det-gros}}}, and then use {\hyperref[thm_Viehweg_Iitaka-type-gen]{Theorem \ref*{thm_Viehweg_Iitaka-type-gen}}} to deduce {\hyperref[thm_Kawamata_Ab-Var]{Theorem \ref*{thm_Kawamata_Ab-Var}}}. Let us remark that in \cite{Kaw81} a result equivalent to {\hyperref[thm_Kawamata_Ab-Var]{Theorem \ref*{thm_Kawamata_Ab-Var}}} with $\Delta=0$ is also stated (\cite[Theorem 25]{Kaw81}).

Classically the proof of {\hyperref[thm_Viehweg_Iitaka-type-gen]{Theorem \ref*{thm_Viehweg_Iitaka-type-gen}}} and of {\hyperref[thm_Viehweg_Iitaka-det-gros]{Theorem \ref*{thm_Viehweg_Iitaka-det-gros}}} are based on Viehweg's weak positivity theorem on the direct image; here we will take a new argument which only depends on the Ohsawa-Takegoshi type extension {\hyperref[thm_Deng_OT]{Theorem \ref*{thm_Deng_OT}}}. Precisely, {\hyperref[thm_Deng_OT]{Theorem \ref*{thm_Deng_OT}}} is used to ensure the effectivity of the twisted relative canonical bundle up to adding an ample line bundle from the base, in virtue of the following auxiliary result:
\begin{lemme}
\label{lemme_kod-eff+pull-ample}
Let $f:X\to Y$ be an analytic fibre space with $X$ a normal complex variety and $Y$ a projective variety. Let $L$ be a holomorphic line bundles on $X$ such that $\kappa(L)\geqslant 0$ and let $A$ be a ample line bundle on $Y$. Then 
\[
\kappa(X,L\otimes f^\ast A)=\kappa(F,L|_F)+\dim Y
\]
where $F$ denotes the general fibre of $f$.
\end{lemme}

Before giving the proof, let us remark that this simple but useful result has been implicitly used in the works on $C_{n,m}$\,, e.g. \cite{Esn81, Vie83}; it is explicitly formulated in \cite[Lemma 4.9]{Cam04} but without proof. For the convenience of the readers, we will give the detailed proof.

\begin{proof}[Proof of {\hyperref[lemme_kod-eff+pull-ample]{Lemma \ref*{lemme_kod-eff+pull-ample}}}]
Up to multiplying $L$ and $A_Y$ by a sufficiently large and divisible integer, we can assume that $\Coh^0(X,L)\neq 0$ and $A$ is very ample; we can further assume that the closure of the image of the meromorphic mapping 
\[
\Phi:=\Phi_{|L\otimes f^\ast A|}:X\dashrightarrow \PP V
\]
with $V:=\Coh^0(X,L\otimes f^\ast A)$ is of dimension $\kappa(X,L\otimes f^\ast A)$. Up to blowing up $X$ we can assume that $\Phi$ is an analytic fibre space (c.f. \cite[Lemma 5.3, pp.~51-52, and Corollary 5.8, p.~57]{Uen75}). Then consider the sub-linear series defined by the inclusion
\[
\Coh^0(Y,A)\simeq\Coh^0(X,f^\ast A)\hookrightarrow \Coh^0(X,L\otimes f^\ast A)\simeq \Coh^0(\PP V,\scrO_{\PP V}(1)),
\]
this gives rise to a meromorphic mapping
\[
\PP V\dashrightarrow \PP\Coh^0(Y,A).
\]
On the other hand, since $A$ is very ample, the linear series $|A|$ defines an closed embedding $i:=\Phi_{|A|}:Y\hookrightarrow \PP\Coh^0(Y,A)$, thus we have the following "commutative" diagram:
\begin{center}
\begin{tikzpicture}[scale=2.5]
\node (A) at (0,0) {$Y$};
\node (B) at (0,1) {$X$};
\node (C) at (2,0) {$\PP\Coh^0(Y,A)$.};
\node (D) at (2,1) {$\PP V$};
\path[->,font=\scriptsize,>=angle 90]
(B) edge node[left]{$f$} (A)
(B) edge node[above right]{$\Phi_{|f^\ast A|}$} (C)
(B) edge node[above]{$\Phi:=\Phi_{|L\otimes f^\ast A|}$} (D);
\path[right hook->,font=\scriptsize,>=angle 90]
(A) edge node[below]{$i:=\Phi_{|A|}$} (C);
\path[dashed,->,font=\scriptsize,>=angle 90]
(D) edge (C);
\end{tikzpicture}
\end{center}
In particular, the general fibre $G$ of $\Phi$ is contracted by $f$, hence we get an analytic fibre space 
\[
\Phi|_F: F\to\Image(\Phi|_F),
\]
whose general fibre is isomorphic to $G$. $\Phi|_F$ is defined by the linear series $|L\otimes f^\ast A|$ restricted to $F$, which is a sub-linear series of $|(L\otimes f^\ast A)|_F|\simeq|L|_F|$, hence we have
\[
\kappa(F,L|_F)\geqslant \dim\Image(\Phi|_F)=\dim\Image\Phi-\dim Y=\kappa(X,L\otimes f^\ast A)-\dim Y.
\]
In addition, by applying the easy inequality \cite[Theorem 5.11, pp.~59-60]{Uen75} to $\Phi|_F$ and $(L\otimes f^\ast A_Y)|_F$ we get
\[
\kappa(F,L|_F)=\kappa(F,(L\otimes f^\ast A)|_F)\leqslant \kappa(G,(L\otimes f^\ast A)|_G)+\dim\Image(\Phi|_F)=\dim\Image(\Phi|_F),
\]
therefore $\kappa(X,L\otimes f^\ast A)=\kappa(F,L|_F)+\dim Y$.
\end{proof}

\subsection{Kähler Version of \texorpdfstring{$C_{n,m}^{\log}$}{text} over General Type Bases}
\label{ss_Kah-Viehweg-Kawamata_type-gen}
In this subsection we will apply the Ohsawa-Takegoshi type extension {\hyperref[thm_Deng_OT]{Theorem \ref*{thm_Deng_OT}}} to recover the result that $C_{n,m}^{\log}$ holds for fibre spaces over general type bases, i.e. to give a new proof of the following theorem:
\begin{thm}[Kähler version of {\cite[Theorem 3]{Kaw81}}, {\cite[Theorem III]{Vie83}}]
\label{thm_Viehweg_Iitaka-type-gen}
Let $f: X\to Y$ be a surjective morphism between compact Kähler manifolds whose general fibre $F$ is connected. And let $\Delta$ be an $\QQ$-effective divisor on $X$ such that $(X,\Delta)$ is klt. Suppose that $Y$ of general type (thus projective). Then 
\[
\kappa(X,K_X+\Delta)\geqslant\kappa(F,K_F+\Delta_F)+\dim Y,
\]
where $\Delta_F:=\Delta|_F$
\end{thm}

Let us remark that by virtue of the easy inequality \cite[Theorem 5.11, pp.~59-60]{Uen75}, the inequality in the theorem is in fact an equality. In order to establish {\hyperref[thm_Viehweg_Iitaka-type-gen]{Theorem \ref*{thm_Viehweg_Iitaka-type-gen}}}, we first prove the following lemma, which can be regarded as a (log) Kähler version of \cite[Corollary 7.1]{Vie83}:

\begin{lemme}
\label{lemme_Viehweg_serie-lin}
Let $f:X\to Y$ be an analytic fibre space with $X$ a (compact) Kähler manifold and $Y$ a smooth projective variety. Let $\Delta$ be an effective $\QQ$-divisor on $X$ such that the pair $(X,\Delta)$ is klt. Then for any ample $\QQ$-line bundle $A_Y$ on $Y$, 
we have
\begin{equation}
\label{eq_Kod-fibre+dim-base}
\kappa(X,K_{X/Y}+\Delta+f^\ast A_Y)=\kappa(F,K_F+\Delta_F)+\dim Y. 
\end{equation}
where $F$ denotes the general fibre of $f$, and $\Delta_F:=\Delta|_F$.
\end{lemme}

\begin{proof}
If $\kappa(F,K_F+\Delta_F)=-\infty$, then for any integer $\mu>0$ sufficiently large and divisible (so that $A_Y\ptensor[\mu]$ is a line bundle and $\mu\Delta$ is an integral divisor) we have 
\[
\scrF_{\mu,\Delta}:=f_\ast\left(K_{X/Y}\ptensor[\mu]\otimes\scrO_X(\mu\Delta)\right)=0, 
\]
thus $\scrF_{\mu,\Delta}\otimes A_Y\ptensor[\mu]=0$, and in particular
\[
\Coh^0(X,K_{X/Y}\ptensor[\mu]\otimes\scrO_X(\mu\Delta)\otimes f^\ast A_Y\ptensor[\mu])=\Coh^0(Y,\scrF_{\mu,\Delta}\otimes A_Y\ptensor[\mu])=0,
\]
therefore $\kappa(X,K_{X/Y}+\Delta+f^\ast A_Y)=-\infty$, hence the equality \eqref{eq_Kod-fibre+dim-base}.

Suppose in the sequel that $\kappa(F,K_F+\Delta_F)\geqslant 0$. Let $m$ be a sufficiently large and divisible positive integer, so that $A_Y\ptensor[m]$ is a line bundle, $m\Delta$ is an integral divisor, $\scrF_{m,\Delta}\neq 0$ and that there is a very ample line bundle $A'_Y$ on $Y$ which satisfies $(A'_Y)\ptensor[2]\simeq A_Y\ptensor[m]$ and the following inequality for Seshadri constant
\[
\epsilon(A'_Y\otimes K_Y\inv,y)>\dim Y,\quad\text{for general }y\in Y.
\]
By {\hyperref[thm_Cao_Bergman]{Theorem \ref*{thm_Cao_Bergman}}} the relative $m$-Bergman kernel metric $h_{X/Y\!,m\Delta}^{(m)}$ on $K_{X/Y}\ptensor[m]\otimes\scrO_X(m\Delta)$ is semi-positively curved. Then as in the proof of {\hyperref[thm_pos-im-pluri-can]{Theorem \ref*{thm_pos-im-pluri-can}}} we consider the line bundle 
\[
L_{m-1}:=K_{X/Y}\ptensor[(m-1)]\otimes\scrO_X(m\Delta)
\]
equipped with the semi-positively curved metric
\[
h_{L_{m-1}}:= (h_{X/Y\!,m\Delta}^{(m)})\ptensor[\frac{m-1}{m}]\otimes h_\Delta,
\]
where $h_{\Delta}$ denotes the singular Hermitian metric whose local weight is defined by the local equation of $\Delta$. Then apply {\hyperref[thm_Deng_OT]{Theorem \ref*{thm_Deng_OT}}} to $L=L_{m-1}$ (by virtue of {\hyperref[lemme_gen-iso_im-pluri-can]{Lemma \ref*{lemme_gen-iso_im-pluri-can}}}) and we get a surjection
\[
\Coh^0(X,K_X\otimes L_{m-1}\otimes f^\ast(A'_Y\otimes K_Y\inv))\twoheadrightarrow \Coh^0(F,K_F\otimes L_{m-1}|_F),
\]
i.e.  
\[
\Coh^0(X,K_{X/Y}\ptensor[m]\otimes\scrO_X(m\Delta)\otimes f^\ast A'_Y)\twoheadrightarrow \Coh^0(F,K_F\ptensor[m]\otimes\scrO_F(m\Delta_F)),
\]
which implies that 
\begin{equation}
\label{eq_nonvanishing-can+delta+A'}
\Coh^0(X,K_{X/Y}\ptensor[m]\otimes\scrO_X(m\Delta)\otimes f^\ast A'_Y)\neq 0.
\end{equation} 
By \eqref{eq_nonvanishing-can+delta+A'} we can apply {\hyperref[lemme_kod-eff+pull-ample]{Lemma \ref*{lemme_kod-eff+pull-ample}}} to $L=K_{X/Y}\ptensor[m]\otimes\scrO_X(m\Delta)\otimes f^\ast A'_Y$ and $A=A'_Y$ and we get
\begin{align*}
\kappa(X,K_{X/Y}+\Delta+f^\ast A_Y) &=\kappa(X,(mK_{X/Y}+m\Delta+f^\ast A'_Y)+f^\ast A'_Y) \\
&=\kappa(F, (mK_{X/Y}+m\Delta+f^\ast A'_Y)|_F)+\dim Y \\
&=\kappa(F,K_F+\Delta_F)+\dim Y. 
\end{align*}
\end{proof}

By virtue of {\hyperref[lemme_Viehweg_serie-lin]{Lemma \ref*{lemme_Viehweg_serie-lin}}}, one easily deduces {\hyperref[thm_Viehweg_Iitaka-type-gen]{Theorem \ref*{thm_Viehweg_Iitaka-type-gen}}}: 

\begin{proof}[Proof of {\hyperref[thm_Viehweg_Iitaka-type-gen]{Theorem \ref*{thm_Viehweg_Iitaka-type-gen}}}]
Since $Y$ is of general type, it is projective. Thus fix an ample line bundle $H$ on $Y$; its canonical bundle $K_Y$ being big, the Kodaira Lemma (c.f. \cite[Lemma 2.60, pp.~67-68]{KM98}) implies that there exists an integer $b>0$ such that $ K_Y\ptensor[b]\otimes H\inv$ is effective. Now by {\hyperref[lemme_Viehweg_serie-lin]{Lemma \ref*{lemme_Viehweg_serie-lin}}} we obtain 
\begin{align*}
\kappa(X, K_X+\Delta) &\geqslant \kappa(X, bK_{X/Y}+b\Delta+H) \\
&=\kappa(F, K_F+\Delta_F)+\dim Y,
\end{align*}
thus we prove {\hyperref[thm_Viehweg_Iitaka-type-gen]{Theorem \ref*{thm_Viehweg_Iitaka-type-gen}}}.
\end{proof}

\subsection{Iitaka Conjecture for Kähler Fibre Spaces with Big Determinant Bundle of the Direct Image of Relative Pluricanonical Bundles}
 \label{ss_Kah-Viehweg-Kawamata_det-gros}
The proof of {\hyperref[main-thm_I]{Main Theorem, Part (I)}} is obtained by combining {\hyperref[lemme_kod-eff+pull-ample]{Lemma \ref*{lemme_kod-eff+pull-ample}}} and {\hyperref[thm_Deng_OT]{Theorem \ref*{thm_Deng_OT}}} plus the following result:

\begin{thm}[Kähler version of {\cite[Theorem 3.4]{CP17}}]
\label{thm_CP_canonique:det-im-dir}
Let $f:X\to Y$ surjective morphism with $X$ a compact Kähler manifold and $Y$ a smooth projective variety such that the general fibre $F$ of $f$ is connected. Let $L$ be a holomorphic $\QQ$-line bundle on $X$ equipped with a singular Hermitian metric $h_L$ such that its curvature current $\Theta_{h_L}(L)\geqslant 0$ and that $\scrJ(h_L)\simeq\scrO_X$. Suppose that there is an integer $m>0$ such that $L\ptensor[m]$ is a line bundle and that
\begin{equation}
\label{eq_cond_thm_CP_canonique:det-im-dir}
f_\ast\left(K_{X/Y}\ptensor[m]\otimes L\ptensor[m]\right)\neq 0.
\end{equation}
Such $m$ exists if and only if $\kappa(F,K_F+L|_F)\geqslant 0$. Suppose that there is a SNC divisor $\Sigma_Y$ containing $Y\backslash Y_0$ where $Y_0$ is the (analytic) Zariski open subset over which $f$ is smooth; suppose further that $f^\ast\Sigma_Y$ has SNC support (in other word, $f$ is prepared in the sense of \cite{Cam04}). Then there exists a constant $\epsilon_0>0$ and an $f$-exceptional effective $\QQ$-divisor $E$ such that the $\QQ$-line bundle 
\begin{equation}
\label{eq_canonique:det-im-dir}
K_{X/Y}+L+E-\epsilon_0f^\ast\det\!f_\ast\left( K_{X/Y}\ptensor[m]\otimes L\ptensor[m]\right)
\end{equation}
is pseudoeffective.
\end{thm}

Before giving the proof, let us remark that:
\begin{rmq}
The condition \eqref{eq_cond_thm_CP_canonique:det-im-dir} concerning the positivity of the Kodaira dimension of the general fibre does not appear in the original statement of \cite[Theorem 3.14]{CP17}, but is indispensable. In fact, consider for example the case where $Y=\text{pt}$, $X$ is a smooth Fano variety (or more generally a smooth uniruled projective variety) with $\Delta=0$, $f$ is the structural morphism $X\to \text{pt}$ and $L=\scrO_X$; $f$ being a smooth morphism, there is no $f$-exceptional divisors, and the direct image (space of global sections) of $K_X\ptensor[m]$ is always $0$, then the $\QQ$-line bundle \eqref{eq_canonique:det-im-dir} is equal to $K_X$, which can never be pseudoeffective for $X$ Fano (or uniruled projective, by \cite{BDPP13}).
\end{rmq}

\begin{proof}[Proof of {\hyperref[thm_CP_canonique:det-im-dir]{Theorem \ref*{thm_CP_canonique:det-im-dir}}}]
The proof follows the same idea as that of \cite[Theorem 3.4]{CP17}; in fact, the algebraicity of $f$ (or equivalently, the algebraicity of $X$) is not essential in the original proof: it is only used in \cite{CP17} to apply the Ohsawa-Takegoshi extension theorem and \cite[III.5.10.Lemma, pp.~107-108]{Nak04}; as have been seen in {\hyperref[ss_pos-im-dir_OT]{\S \ref*{ss_pos-im-dir_OT}}} and {\hyperref[ss_preliminary_ref-hull]{\S \ref*{ss_preliminary_ref-hull}}} respectively, both of them can be generalized to the Kähler case. Nevertheless, the proof being highly technical, we will give more details for the convenience of the readers. Let us summarize the central idea of the proof as follows: from the natural inclusion of the determinant into the tensor product, we can construct, by the diagonal method of Viehweg, a non-zero section on $X^{(r)}$ (where $X^{(r)}$ denotes the resolution of some fibre product $X^r$ of $X$ over $Y$) of a line bundle of the form \eqref{eq_canonique:det-im-dir} (with $X$ replaced by $X^r$ and $\epsilon_0=1$); and then we "restrict" this section to the diagonal so that we get a section of the line bundle \eqref{eq_canonique:det-im-dir} on $X$. However one cannot deduce the effectivity of the line bundle \eqref{eq_canonique:det-im-dir}, since the section constructed as above can vanish along the diagonal. To overcome this difficulty, we have to take a twisted approach: at the cost of tensoring by an ample divisor on $Y$, we can use the Ohsawa-Takegoshi extension {\hyperref[thm_Deng_OT]{Theorem \ref*{thm_Deng_OT}}} to extend pluricanonical forms on the general fibre $F$ (by virtue of the condition \eqref{eq_cond_thm_CP_canonique:det-im-dir}) to sections of the line bundle of the form \eqref{eq_canonique:det-im-dir} on $X^{(r)}$, then one can restrict them to the diagonal and get non-zero sections. However, these sections usually have poles, due to the singularities of $f$; in order to get rid of them, one has to carefully analyse these singularities (this analysis takes up a technical part of the proof), then it turns out that one can use {\hyperref[prop_Bergman-discriminant]{Proposition \ref*{prop_Bergman-discriminant}}} to control these poles. Finally one use an approximation argument to conclude the pseudoeffectivity of the line bundle \eqref{eq_canonique:det-im-dir}. The proof of the theorem proceeds in six steps: 
  
\paragraph{\quad (A) Analysis of singular fibres of $f$.\\}
\label{thm_CP_canonique:det-im-dir_demo-A}
In this step, we will use a standard argument to show that the (analytic Zariski) open subset of $y\in Y$ such that $X_y$ is Gorenstein is of codimension $\geqslant 2$ (whilst the generic smoothness only ensure this to be analytic Zariski open). To this end, note 
\[
Y_{\ff}:=Y_{\text{flat}}\cap Y_{\scrF_{m,L}}
\]
the (analytic) Zariski open subset over which $f$ is flat and $\scrF_{m,L}:=f_\ast\left( K_{X/Y}\ptensor[m]\otimes L\ptensor[m]\right)$ is locally free; and denote $X_{\ff}:=f\inv Y_{\ff}$\,.  since $X$ and $Y$ are reduced, $\codim_Y(Y\backslash Y_{\text{f}})\geqslant 2$ (c.f. \cite[Corollary 5.5.15, p.~147]{Kob87} and \cite[Example A.5.4, p.~416]{Ful84}). By \cite[Theorem 23.4, p.~181]{Mat89}, for every $y\in Y_{\ff}$, the fibre $X_y$ is Gorenstein.

\paragraph{\quad(B) Construction of the fibre product $X^r$ and the canonical section.\\}
\label{thm_CP_canonique:det-im-dir_demo-B}
Over $Y_{\ff}$ one has a natural morphism (injection of vector bundles)
\begin{equation}
\label{eq_im-dir--det-tensor}
\det\!f_\ast\left( K_{X/Y}\ptensor[m]\otimes L\ptensor[m]\right)\hookrightarrow\bigotimes^r f_\ast\left( K_{X/Y}\ptensor[m]\otimes L\ptensor[m]\right),
\end{equation}
where $r:=\rank\scrF_{m,L}$\,, which gives rise to a non-trivial section of 
\begin{equation}
\label{eq_ptensor-im-dir:det}
\left(\bigotimes^r f_\ast\left( K_{X/Y}\ptensor[m]\otimes L\ptensor[m]\right)\right)\otimes\left(\det\!f_\ast\left( K_{X/Y}\ptensor[m]\otimes L\ptensor[m]\right)\right)\inv.
\end{equation}
over $Y_{\ff}$. In order to get a section of a line bundle of the form \eqref{eq_canonique:det-im-dir}, we will apply the diagonal method of Viehweg (c.f. for example \cite[\S 6.5, pp.~192-196]{Vie95}). Let 
\[
X^r:=\underbrace{X\underset{Y}{\times}X\underset{Y}{\times}\cdots\underset{Y}{\times}X}_{r\text{ times}}
\]
be the $r$-fold fibre product of $X$ over $Y$, equipped with a morphism (a Kähler fibration) $f^r:X^r\to Y$ as well as the natural projections $\pr_i:X^r\to X$ to the $i$-th factor. Denote $X_{\ff}^r:=(f^r)\inv Y_{\ff}$, then $f^r|_{X_{\ff}}$ is plat; moreover, since $Y$ and $X_y^r=X_y\times\cdots\times X_y$ are Cohen-Macaulay for every $y\in Y_{\ff}$, $X^r_{\ff}$ is also Cohen-Macaulay (by \cite[(21.C) Corollary 2, p.~154]{Mat70}). By the base change formula for relative canonical sheaves 
we see that $X^r_{\ff}$ is Gorenstein and
\begin{equation}
\label{eq_caconique-rel-prod-fib}
\omega_{X^r}\otimes f^{r\ast}K_Y\inv=\omega_{X^r/Y}\simeq\bigotimes_{i=1}^r\pr_i^\ast K_{X/Y}
\end{equation}
Note
\[
L_r:=\bigotimes^r_{i=1}\pr_i^\ast L\,,
\]
then by an induction argument, the projection formula together with the base change formula imply that (c.f. \cite[Lemma 3.15]{Hor10})
\[
\bigotimes^r f_\ast\left( K_{X/Y}\ptensor[m]\otimes L\ptensor[m]\right)\simeq (f^r)_\ast\left(\omega_{X^r/Y}\ptensor[m]\otimes L_r\ptensor[m]\right)\quad\text{over }Y_{\ff}.
\]
In consequence, the morphism \eqref{eq_im-dir--det-tensor} gives rise to a non-zero section
\begin{align}
s_0 &\in \Coh^0(X^r_{\ff},\omega_{X^r/Y}\ptensor[m]\otimes L_r\ptensor[m]\otimes(f^r)^\ast\left(\det\!f_\ast\left( K_{X/Y}\ptensor[m]\otimes L\ptensor[m]\right)\right)\inv) \nonumber \\
&=\Coh^0(Y_{\ff},\left(\bigotimes^r f_\ast\left( K_{X/Y}\ptensor[m]\otimes L\ptensor[m]\right)\right)\otimes\left(\det\!f_\ast\left( K_{X/Y}\ptensor[m]\otimes L\ptensor[m]\right)\right)\inv). \label{eq_sec-canonique:det-im-dir--ff}
\end{align}

\paragraph{\quad (C) Analysis of the singularities of $X^r$.\\}
\label{thm_CP_canonique:det-im-dir_demo-C}
Take a desingularization $\mu:X^{(r)}\to X^r$ which is an isomorphism over the smooth locus of $X^r$. Note $f^{(r)}:=f^r\circ\mu$ and $X_{\ff}^{(r)}:=\mu\inv X_{\ff}^r$. The natural morphism
\begin{equation}
\label{eq_canonique-resolution}
\mu_\ast K_{X^{(r)}}\to\omega_{X^r}\,,
\end{equation}
which is an isomorphism over $X^r_{\rat}$ where $X^r$ denotes the (analytic Zariski) open subset of point with rational singularities on $X^r$, gives rise to a meromorphic section of the line bundle (by virtue of \eqref{eq_caconique-rel-prod-fib})
\[
K_{X^{(r)}/Y}\inv\otimes\mu^\ast\left(\bigotimes_{i=1}^r\pr_i^\ast K_{X/Y}\right),
\]
whose zeros and poles are contained in $X^{(r)}\backslash \mu\inv X^r_{\rat}$. In consequence, there are two effective divisors  $D_1$ and $D_2$ over $X^{(r)}$ such that $\Supp(D_1),\Supp(D_2)\subseteq X^{(r)}\backslash\mu\inv X^r_{\rat}$ and that 
\begin{equation}
\label{eq_canonique-desing:pull-back}
K_{X^{(r)}/Y}\otimes\scrO_{X^{(r)}}(D_1)=\mu^\ast\left(\bigotimes_{i=1}^r\pr_i K_{X/Y}\right)\otimes\scrO_{X^{(r)}}(D_2).
\end{equation}

Now let us further analyse the rational singularities locus $X^r_{\rat}$ by virtue of our hypothesis on $\Sigma_Y$ and $f^\ast \Sigma_Y$. Write
\begin{equation}
\label{eq_pull-back-SigmaY}
f^\ast\Sigma_Y=\sum_i W_i+\sum_j a_j V_j
\end{equation}
with the $W_i$'s and $V_j$'s prime divisors over $X$ and $a_i\geqslant 2$; by hypothesis, 
\[
W:=\sum_i W_i\quad\text{ et }\quad V:=\sum_j V_j
\]
are (reduced) SNC divisors. As is explained in {\hyperref[rmq_lemme_cov-trick]{Remark \ref*{rmq_lemme_cov-trick}}}, the fibre product 
\[
(X_{\ff}\backslash (V\cup f\inv\Sing(\Sigma_Y)))^r:=\underbrace{(X_{\ff}\backslash (V\cup f\inv\Sing(\Sigma_Y)))\underset{Y_{\ff}\backslash\Sing(\Sigma_Y)}{\times}\cdots\underset{Y_{\ff}\backslash\Sing(\Sigma_Y)}{\times}(X_{\ff}\backslash(V\cup f\inv\Sing(\Sigma_Y)))}_{r\text{ times}}
\]
is contained in $X^r_{\rat}$\,.

In consequence, both $D_1$ and $D_2$ are contained in the set $\scrD$ where $\scrD$ denotes the set of divisors $D$ on $X^{(r)}$ such that every component $\Gamma$ of $D$ satisfies (at least) one of the following three conditions:
\begin{itemize}
\item[($\scrD$1)]\label{eq_defn-ens-scrD_1}
$f^{(r)}(\Gamma)\subseteq Y\backslash Y_{\ff}$ (in particular, $\Gamma$ is $f^{(r)}$-exceptional);
\item[($\scrD$2)]\label{eq_defn-ens-scrD_2} $\Gamma$ is $\pr_i\circ\mu$-exceptional for some $i$;
\item[($\scrD$3)]\label{eq_defn-ens-scrD_3} $\pr_i\circ\mu(\Gamma)=V_j$ for some $i$ and $j$.
\end{itemize}

\paragraph{\quad (D) Extension of pluricanonical forms on $X^{(r)}_y$ by Ohsawa-Takegoshi.\\}
\label{thm_CP_canonique:det-im-dir_demo-D}
The section $s_0$ (c.f. \eqref{eq_sec-canonique:det-im-dir--ff}) gives rise the section
\[
\mu^\ast s_0\in\Coh^0(X^{(r)}_{\ff}, K_{X^{(r)}/Y}\ptensor[m]\otimes\mu^\ast L_r\ptensor[m]\otimes\scrO_{X^{r)}}(mD_1)\otimes f^{(r)\ast}\left(\det\!f_\ast\left( K_{X/Y}\ptensor[m]\otimes L\ptensor[m]\right)\right)\inv).
\]
Since $\codim_Y Y_{\ff}\geqslant 2$, the section $\mu^\ast s_0$, regarded as a section of the torsion free sheaf \eqref{eq_ptensor-im-dir:det} over $Y_{\ff}$, extends to a global section $\bar s_0$ of the reflexive hull
\begin{align*}
&\quad \left[\left(\bigotimes^r f_\ast\left( K_{X/Y}\ptensor[m]\otimes L\ptensor[m]\right)\right)\otimes\left(\det\!f_\ast\left( K_{X/Y}\ptensor[m]\otimes L\ptensor[m]\right)\right)\inv\right]^\wedge \\
&= \left[f^{(r)}_\ast\left( K_{X^{(r)}/Y}\ptensor[m]\otimes\mu^\ast L_r\ptensor[m]\otimes\scrO_{X^{r)}}(mD_1)\otimes f^{(r)\ast}\left(\det\!f_\ast\left( K_{X/Y}\ptensor[m]\otimes L\ptensor[m]\right)\right)\inv\right)\right]^\wedge. 
\end{align*}
By {\hyperref[thm_env-ref]{Theorem \ref*{thm_env-ref}}}, there is an $f^{(r)}$-exceptional effective divisor $D_3$ such that 

\begin{align*}
&\quad \left[f^{(r)}_\ast\left(K_{X^{(r)}/Y}\ptensor[m]\otimes\mu^\ast L_r\ptensor[m]\otimes\scrO_{X^{r)}}(mD_1)\otimes f^{(r)\ast}\left(\det\!f_\ast\left( K_{X/Y}\ptensor[m]\otimes L\ptensor[m]\right)\right)\inv\right)\right]^\wedge \\
&= \left[f^{(r)}_\ast\left( K_{X^{(r)}/Y}\ptensor[m]\otimes\mu^\ast L_r\ptensor[m]\otimes\scrO_{X^{r)}}(mD_1)\right)\right]^\wedge\otimes f^{(r)\ast}\left(\det\!f_\ast\left( K_{X/Y}\ptensor[m]\otimes L\ptensor[m]\right)\right)\inv \\
&= f^{(r)}_\ast\left( K_{X^{(r)}/Y}\ptensor[m]\otimes\mu^\ast L_r\ptensor[m]\otimes\scrO_{X^{r)}}(mD_1+D_3)\right)\otimes f^{(r)\ast}\left(\det\!f_\ast\left( K_{X/Y}\ptensor[m]\otimes L\ptensor[m]\right)\right)\inv \\
&= f^{(r)}_\ast\left( K_{X^{(r)}/Y}\ptensor[m]\otimes\mu^\ast L_r\ptensor[m]\otimes\scrO_{X^{r)}}(mD_1+D_3)\otimes f^{(r)\ast}\left(\det\!f_\ast\left( K_{X/Y}\ptensor[m]\otimes L\ptensor[m]\right)\right)\inv\right)\,,
\end{align*}
hence $\bar s_0$ can be regarded as a (global) section of the line bundle
\[
K_{X^{(r)}/Y}\ptensor[m]\otimes\mu^\ast L_r\ptensor[m]\otimes\scrO_{X^{r)}}(mD_1+D_3)\otimes f^{(r)\ast}\det\!f_\ast\left( K_{X/Y}\ptensor[m]\otimes L\ptensor[m]\right)\inv\,.
\]
Moreover, since the torsion free sheaf \eqref{eq_ptensor-im-dir:det} is locally free on $Y_{\ff}$\,, hence
\[
f^{(r)}\left(\Supp(D_3)\right)\subseteq Y\backslash Y_{\ff}\,,
\]
in particular, $D_3\in\scrD$. Now choose $\epsilon\in\QQ_{>0}$ small enough such that $\Delta_0:=\epsilon\divisor(\bar s_0)$ is klt on $X^{(r)}$.The $\QQ$-line bundle $\scrO_{X^{(r)}}(\Delta_0)$ is equipped with a canonical singular Hermitian metric $h_{\Delta_0}$ whose local weight is given by
\[
\phi_{\Delta_0}=\frac{\epsilon}{2}\log|g_{\bar s_0}|^2,
\]
where $g_{\bar s_0}$ denotes a local equation of $\divisor(\bar s_0)$. Denote $L_0:=\mu^\ast L_r\otimes\scrO_{X^{(r)}}(\Delta_0)$, this $\QQ$-line bundle is equipped with the singular Hermitian metric
\[
h_{L_0}:=h_{\Delta_0}\otimes \bigotimes_{i=1}^r \mu^\ast\pr_i^\ast h_L\,.
\]
whose curvature current is positive. Since $\scrJ(h_{L_0})=\scrO_{X^{(r)}}$, by the same argument as in the point {\hyperref[prop_metrique_im-dir_demo-b]{(b)}} of the proof of {\hyperref[prop_metrique_im-dir]{Proposition \ref*{prop_metrique_im-dir}}}, 
we have that $\scrJ(h_{L_0}|_{X^{(r)}_y})=\scrO_{\!\!X^{(r)}_y}$ for $y\in Y_0$ almost everywhere (since $\mu$ is supposed to be an isomorphism over $Y_0$\,, we have $X^{(r)}_y\simeq X_y\times\cdots\times X_y$ for $y\in Y_0$, c.f. Step {\hyperref[thm_CP_canonique:det-im-dir_demo-E1]{(E1)}} below). 

Let $A_Y$ be an ample line bundle over $Y$ such that the line bundle $A_Y\otimes K_Y\inv$ is ample and that the Seshadri constant $\epsilon(A_Y\otimes K_Y\inv, y)>d:=\dim Y$ for every $y\in Y_0$. Claim that the restriction map 
\begin{equation}
\label{eq_restriction_prod-fib}
\Coh^0(X^{(r)}, K_{X^{(r)}/Y}\ptensor[k]\otimes L_0\ptensor[k]\otimes f^{(r)\ast} A_Y)\longrightarrow\Coh^0(X^{(r)}_y, K_{\!\!X^{(r)}_y}\ptensor[k]\otimes L_0\ptensor[k]|_{X^{(r)}_y})
\end{equation}
is surjective for any $k$ sufficiently large and divisible and for every $y\in Y_0$ such that $\scrJ(h_{L_0}|_{X^{(r)}_y})=\scrO_{\!\!X^{(r)}_y}$\,. In fact, $\Delta_0$ being effective, the hypothesis \eqref{eq_cond_thm_CP_canonique:det-im-dir} implies that 
\begin{align*}
f^{(r)}_\ast\left( K_{X^{(r)}/Y}\ptensor[k]\otimes L_0\ptensor[k]\right)
&= f^{(r)}_\ast\left( K_{X^{(r)}/Y}\ptensor[k]\otimes \mu^\ast L_r\ptensor[k]\otimes\scrO_{X^{(r)}}(k\Delta_0)\right) \\ 
&\supseteq f^{(r)}_\ast\left( K_{X^{(r)}/Y}\ptensor[k]\otimes\mu^\ast L_r\ptensor[k]\right)\neq 0
\end{align*}
for $k$ sufficiently large and divisible (e.g. such that $\epsilon k\in\ZZ_{>0}$ and $k$ divisible by $m$). Moreover, since $h_{L_0}$ is klt and $\Theta_{h_{L_0}}(L_0)\geqslant 0$, {\hyperref[thm_Cao_Bergman]{Theorem \ref{thm_Cao_Bergman}}} implies that the $k$-Bergman kernel metric $h^{(k)}_{X^{(r)}/Y,kL_0}$ is semi-positively curved (by virtue of {\hyperref[lemme_Fubini_multiplier-ideal]{Lemma \ref*{lemme_Fubini_multiplier-ideal}}}). Set $M_k:= K_{X^{(r)}/Y}\ptensor[(k-1)]\otimes L_0\ptensor[k]$, equipped with a singular Hermitian metric
\[h_{M_k}:=\left(h_{X^{(r)}/Y,kL_0}^{(k)}\right)^{\frac{k-1}{k}}\otimes h_{L_0}\]
whose curvature current is positive. Then by {\hyperref[lemme_gen-iso_im-pluri-can]{Lemma \ref*{lemme_gen-iso_im-pluri-can}}} and {\hyperref[rmq_klt-gen-iso]{Remark \ref*{rmq_klt-gen-iso}}} one has 
\begin{equation}
\label{eq_klt-metrique-M_k}
\Coh^0(X^{(r)}_y,K_{X^{(r)}_y}\ptensor[k]\otimes L_0\ptensor[k]|_{X^{(r)}}\otimes\scrJ(h_{M_k}|_{X^{(r)}_y}))=\Coh^0(X^{(r)}_y,K_{X^{(r)}_y}\ptensor[k]\otimes L_0\ptensor[k]|_{X^{(r)}})
\end{equation}
for general $y\in Y$.
Hence we can apply {\hyperref[thm_Deng_OT]{Theorem \ref*{thm_Deng_OT}}} to 
\[
K_{X^{(r)}}\otimes M_k\otimes f^{(r)\ast}(A_Y\otimes K_Y\inv)=  K_{X^{(r)}/Y}\ptensor[k]\otimes L_0\ptensor[k]\otimes f^{(r)\ast}A_Y
\] 
to obtain the surjectivity of the restriction morphism  \eqref{eq_restriction_prod-fib}
for general $y\in Y_0$. Moreover, set $H_k:=A_Y\otimes\det\!f_\ast\left( K_{X/Y}\ptensor[m]\otimes L\ptensor[m]\right)\ptensor[-\epsilon k]$, then we can rewrite \eqref{eq_restriction_prod-fib} as 
\begin{align}
&\Coh^0(X^{(r)}, \left( K_{X^{(r)}/Y}\otimes\mu^\ast L_r\right)\ptensor[\,(1+\epsilon m)k]\otimes\scrO_{X^{(r)}}(\epsilon kmD_1+\epsilon kD_3)\otimes f^{(r)\ast}H_k) \nonumber \\
&\xtwoheadrightarrow{\text{restriction}} \Coh^0(X^{(r)}_y,\left( K_{\!\!X^{(r)}_y}\otimes \mu^\ast L_r|_{X^{(r)}_y}\right)\ptensor[\,(1+\epsilon m)k])\label{eq_restriction_prod-fib_2}
\end{align}
for general $y\in Y_0$ and for $k$ sufficiently large and divisible.

\paragraph{\quad (E) Extension of pluricanonical forms on $X_y$ via restriction to the diagonal.\\}
\label{thm_CP_canonique:det-im-dir_demo-E}
For general $y\in Y_0$ take a section
\[
u\in\Coh^0(X_y,\left( K_{X_y}\otimes L|_{X_y}\right)\ptensor[\,(1+\epsilon m)k])
\]
with $k$ sufficiently large and divisible, we will construct a section $s$ in 
\[
\Coh^0(X, \left( K_{X/Y}\otimes L\right)\ptensor[\,(1+\epsilon m)kr]\otimes\scrO_X(CkV+kE_0)\otimes f^\ast H_k\ptensor[r])\,,
\]
for $C>0$ a constant and $E_0$ an $f$-exceptional effective divisor, both independent of $k$, such that $s|_{X_y}=u\ptensor[r]$\,. 

\subparagraph{\quad (E1) Extending the section $u$ to a section over $X^{(r)}$ by Step {\hyperref[thm_CP_canonique:det-im-dir_demo-D]{(D)}} \\}
\label{thm_CP_canonique:det-im-dir_demo-E1}
Note
\[
X_0^r:=X_0\underset{Y_0}{\times}X_0\underset{Y_0}{\times}\cdots\underset{Y_0}{\times}X_0\,\subseteq X^r\,,
\]
then $X_0^r$ is smooth, hence $\mu\inv X^0_r\xrightarrow{\overset{\mu}{\sim}}X^r_0$ is an isomorphism\,. In particular, we have 
\begin{equation}
\label{eq_iso-mu-Y0}
X_y^{(r)}\xrightarrow{\underset{\sim}{\mu}}X_y^r=\underbrace{X_y\times X_y\times\cdots\times X_y}_{r\text{ times}}\,.
\end{equation}
Hence $u$ gives rise to a section 
\begin{equation}
\label{eq_section-u(r)}
u^{(r)}:=\mu^\ast\left(\bigotimes_{i=1}^r\pr_i^\ast u\right) \in\;\Coh^0(X^{(r)}_y, \left( K_{\!\!X^{(r)}_y}\otimes \mu^\ast L_r|_{X^{(r)}_y}\right)\ptensor[\,(1+\epsilon m)k])\,,
\end{equation}
such that the restriction of $u^{(r)}$ to the diagonal is equal to $u\ptensor[r]$. Using the surjection \eqref{eq_restriction_prod-fib_2} we obtain a section $\sigma^{(r)}$ of the line bundle
\begin{equation}
\label{eq_fed_desing-prod-fib}
\left( K_{X^{(r)}/Y}\otimes\mu^\ast L_r\right)\ptensor[\,(1+\epsilon m)k]\otimes\scrO_{X^{(r)}}(\epsilon kmD_1+\epsilon kD_3)\otimes f^{(r)\ast}H_k\,,
\end{equation}
such that $\sigma^{(r)}|_{X^{(r)}_y}=u^{(r)}$. 

\subparagraph{\quad (E2) Restricting the section $\sigma^{(r)}|_{\mu\inv X^r_0}$ to the diagonal \\}
\label{thm_CP_canonique:det-im-dir_demo-E2}
In order to restrict $\sigma^{(r)}|_{\mu\inv X^r_0}$ to the diagonal, use \eqref{eq_canonique-desing:pull-back} to rewrite the line bundle \eqref{eq_fed_desing-prod-fib} as follows:
\begin{align}
&\quad\left( K_{X^{(r)}/Y}\otimes\mu^\ast L_r\right)\ptensor[\,(1+\epsilon m)k]\otimes\scrO_{X^{(r)}}(\epsilon kmD_1+\epsilon kD_3)\otimes f^{(r)\ast}H_k \nonumber \\
&=\mu^\ast\left(\bigotimes_{i=1}^r\pr_i^\ast( K_{X/Y}\otimes L)\right)\ptensor[\,(1+\epsilon m)k]\otimes\scrO_{X^{(r)}}(-kD_1+(1+\epsilon m)kD_2+\epsilon kD_3)\otimes f^{(r)\ast}H_k\,. \label{eq_canonique-desing:pull-back_2}
\end{align}
In consequence, $\sigma^{(r)}$ can be regarded as a meromorphic section  of the line bundle
\begin{equation}
\label{eq_fed-prod-canonique-rel}
\mu^\ast\left(\bigotimes_{i=1}^r\pr_i^\ast( K_{X/Y}\otimes L)\right)\ptensor[\,(1+\epsilon m)k]\otimes f^{(r)\ast}H_k
\end{equation}
whose poles are contained $\Supp(D_2)\cup\Supp(D_3)$. Locally, by choosing a trivialization of the line bundle \eqref{eq_fed-prod-canonique-rel}, the section $\sigma^{(r)}$ can be written as a meromorphic function $F^{(r)}$ such that
\begin{equation}
\label{eq_local-section-s'}
g_{D_1}^{-k}g_{D_2}^{(1+\epsilon m)k}g_{D_3}^{\epsilon k}\cdot F^{(r)}
\end{equation}
is holomorphic, where $g_{D_l}$ is a local equation of the divisor $D_l$ ($l=1,2,3$). 

By construction, $D_1,D_2,D_3\in\scrD$ (in particular, $D_3$ is $f^{(r)}$-exceptional), hence there exist constants $C_1$ et $C_2$ such that 
\begin{equation}
\label{eq_comparaison-D12-V}
D_l\leqslant C_l\cdot\mu^\ast\sum_{i=1}^r\pr_i^\ast V,\quad\text{pour }l=1,2
\end{equation}
over $X^{(r)}_{\ff}\backslash S$ where $S\subseteq X^{(r)}$ denotes the union of the components in $D_1+D_2$ which are $\pr_i\circ\mu$-exceptional for every $i=1,\cdots,r$. By Step {\hyperref[thm_CP_canonique:det-im-dir_demo-D]{(D)}}) we have 
\[f^{(r)}\left(\Supp(D_3)\right)\subseteq Y\backslash Y_{\ff}\,,
\]
hence locally over $X^{(r)}_{\ff}\backslash S$ the meromorphic function
\[
F^{(r)}\cdot\prod_{i=1}^r\big((\pr_i\circ\mu)^\ast g_{V}\big)^{C_2(1+\epsilon m)k} =F^{(r)}\cdot\prod_{i=1}^r(\pr_i\circ\mu)^\ast\left(g_V^{C_2(1+\epsilon m)k}\right)
\]
is holomorphic where $g_V=\prod_j g_{V_j}$ is a local equation of $V$.

Note $\delta_{X,r}:X\to X^r$ the inclusion of the diagonal. Then $\pr_i\circ\delta_{X,r}=\id_X$ for $\forall i=1,\cdots, r$. Since the $D_l$'s ($l=1,2,3$) are disjoint to 
\[
\mu\inv X^r_{\rat}\supseteq\mu\inv X^r_0\supseteq \mu\inv(\delta_{X,r}(X_0))\,,
\]
then locally the meromorphic function $F^{(r)}$ is holomorphic over $\mu\inv X^r_0$\,. Therefore we can restrict $\sigma^{(r)}|_{\mu\inv X^r_0}$ to the diagonal and obtain a section 
\[
s_1:=(\mu|_{X^r_0}\inv\circ\delta_{X,r}|_{X_0})^\ast \left(\sigma^{(r)}|_{\mu\inv X^r_0}\right)
\]
over $X_0$ of the line bundle
\begin{equation}
\label{eq_diagonale_canonique+pull-back}
( K_{X/Y}\otimes L)\ptensor[\,(1+\epsilon m)kr]\otimes f^\ast H_k\ptensor[r]\,.
\end{equation}
Locally over an open subset of $X_0$ trivializing the line bundle \eqref{eq_diagonale_canonique+pull-back} the section $s_1$ is given by a holomorphic function 
\[
F_1:=(\mu|_{X^r_0}\inv\circ\delta_{X,r}|_{X_0})^\ast \left(F^{(r)}|_{\mu\inv X^r_0}\right)\,.
\]

\subparagraph{\quad (E3) Extending the section $s_1$ across the singular fibres of $X$ \\}
\label{thm_CP_canonique:det-im-dir_demo-E3}
In order to extend $s_1$ across $f\inv\Sigma_Y$, one needs to know its behaviour around the $W_i$'s and the $V_j$'s; this can be done by analysing the poles along the $D_l$'s of $\sigma^{(r)}$, regarded as a meromorphic section of the line bundle \eqref{eq_fed-prod-canonique-rel}, as we explain in the sequel:

\begin{description}
\item[(E3-i)]\label{thm_CP_canonique:det-im-dir_demo-E3i}
By Step {\hyperref[thm_CP_canonique:det-im-dir_demo-C]{(C)}} $(X_{\ff}\backslash(V\cup\Sing(W)))^r$ is contained in $X^r_{\rat}$, thus disjoint to the $D_l$'s ($l=1,2,3$); regarding $F_1$ as a holomorphic function on $\delta_{X,r}(X_0)$), one has
\[
\mu^\ast F_1=F^{(r)}|_{\mu\inv(\delta_{X,r}(X_0))},
\]
but the poles of $F^{(r)}$ are contained in  $\Supp(D_2)\cup\Supp(D_3)$, hence the function $F_1$ is bounded around the $X_{\ff}\backslash(V\cup f\inv\Sing(\Sigma_Y))$, and thus $F_1$ can be extended to $X_{\ff}\backslash(V\cup f\inv\Sing(\Sigma_Y))$ by Riemann extension; moreover, by Hartogs extension, $F_1$ extends to a holomorphic function over $X_{\ff}\backslash V$.

\item[(E3-i\!i)]\label{thm_CP_canonique:det-im-dir_demo-E3ii} 
In general, $F_1$ is not bounded around $V$. Nevertheless, by Step {\hyperref[thm_CP_canonique:det-im-dir_demo-E2]{(E2)}} the meromorphic function
\[
F^{(r)}\cdot\mu^\ast\prod_{i=1}^r\pr_i^\ast\left(g_V^{C_2(1+\epsilon m)k}\right)
\]
is holomorphic over $X^r_{\ff}\backslash S$ with the restriction of $S$ to the diagonal is an analytic subset of codimension $\geqslant 2$ (c.f. {\hyperref[thm_CP_canonique:det-im-dir_demo-E2]{(E2)}} for the definition of $S$), hence the function
\[
F_1\cdot g_V^{C_2(1+\epsilon m)kr}
\]
is bounded around a general point of $V\cap X_{\ff}$. By Riemann extension (as well as Hartogs extension) $F_1$ extends across $V\cap X_{\ff}$ as a holomorphic local section of the line bundle
\[
( K_{X/Y}\otimes L)\ptensor[\,(1+\epsilon m)kr]\otimes \scrO_X(CkV)\otimes f^\ast H_k\ptensor[r],
\]
where $C:=C_2(1+\epsilon m)r$ is a constant independent of $k$. Combining this with {\hyperref[thm_CP_canonique:det-im-dir_demo-E3i]{(E3-i)}} we obtain an extension of $s_1$ to a section over $X_{\ff}$:
\[
\bar s_1\in\Coh^0(X_{\ff},( K_{X/Y}\otimes L)\ptensor[\,(1+\epsilon m)kr]\otimes \scrO_X(CkV)\otimes f^\ast H_k\ptensor[r])\,.
\]

\item[(E3-i\!i\!i)]\label{thm_CP_canonique:det-im-dir_demo-E3iii} 
At last, we will extend $\bar s_1$ to a global section, which provides the section $s$ that we search for. In fact, $\bar s_1$ can be regarded as a section  of the direct image sheaf
\begin{equation}
\label{eq_diagonale_canonique+pull-back+V}
f_\ast\left(( K_{X/Y}\otimes L)\ptensor[\,(1+\epsilon m)kr]\otimes\scrO_X(CkV)\otimes f^\ast H_k\ptensor[r]\right);
\end{equation}
but $\codim_Y(Y_{\ff})\geqslant 2$, hence $\bar s_1$ extends to a section $s$ of the (torsion free) sheaf \eqref{eq_diagonale_canonique+pull-back+V}. By {\hyperref[thm_env-ref]{Theorem \ref*{thm_env-ref}}}, there is an $f$-exceptional effective divisor $E_0$\,, independent of $k$, such that
\begin{align*}
&\quad f_\ast\left(( K_{X/Y}\otimes L)\ptensor[\,(1+\epsilon m)kr]\otimes \scrO_X(CkV)\otimes f^\ast H_k\ptensor[r]\right)^\wedge \\
&=f_\ast\left(( K_{X/Y}\otimes L)\ptensor[\,(1+\epsilon m)kr]\otimes \scrO_X(CkV+kE_0)\otimes f^\ast H_k\ptensor[r]\right)\,,
\end{align*}
hence
\[
s\in\Coh^0(X,( K_{X/Y}\otimes L)\ptensor[\,(1+\epsilon m)kr]\otimes \scrO_X(CkV+kE_0)\otimes f^\ast H_k\ptensor[r]).
\]
Moreover, by \eqref{eq_iso-mu-Y0} as well as the construction of the section $u^{(r)}$ (c.f. \eqref{eq_section-u(r)})we have
\[
s|_{X_y}=s_1|_{X_y}=(\delta\circ\mu)^\ast u^{(r)}=u\ptensor[r].
\]
This finishes {\hyperref[thm_CP_canonique:det-im-dir_demo-E3]{(E3)}} and thus the Step {\hyperref[thm_CP_canonique:det-im-dir_demo-E]{(E)}}.  
\end{description}

\paragraph{\quad (F) Conclusion.\\}
\label{thm_CP_canonique:det-im-dir_demo-F}
By the hypothesis \eqref{eq_cond_thm_CP_canonique:det-im-dir}, for any general $y\in Y$ and for any integer $k$ sufficiently large and divisible (e.g. such that $\epsilon k\in\ZZ_{>0}$ and that $k$ divisible par $m$), we have a non-zero section 
\[
u\in\Coh^0(X_y,\left( K_{X_y}\otimes L|_{X_y}\right)\ptensor[\,(1+\epsilon m)k]).
\]  
Assume further that $y\in Y_0$ and $\scrJ(h_{L_0}|_{X^{(r)}_y})=\scrO_{\!\!X^{(r)}_y}$\,, then by Step {\hyperref[thm_CP_canonique:det-im-dir_demo-E]{(E)}} above, we can construct a section 
\[
s\in\Coh^0(X,( K_{X/Y}\otimes L)\ptensor[\,(1+\epsilon m)kr]\otimes\scrO_X(CkV+kE_0)\otimes f^\ast H_k),
\]
for $C$ and $E_0$ independent of $k$ such that $s|_{X_y}=u\ptensor[r]$. In particular $s\neq 0$, implying that the line bundle
\begin{equation}
\label{eq_diagonale_canonique+pull-back+V+E0}
( K_{X/Y}\otimes L)\ptensor[\,(1+\epsilon m)kr]\otimes \scrO_X(CkV+kE_0)\otimes f^\ast H_k\ptensor[r]
\end{equation}
is effective. By writing $V=V_{\text{hor}}+V_{\text{ver}}$ with $V_{\text{hor}}$ (resp. $V_{\text{ver}}$) the horizontal (resp. vertical) part of $V$ with respect to $f$, one can rewrite the line bundle \eqref{eq_diagonale_canonique+pull-back+V+E0} as follows:
\begin{align*}
&\quad ( K_{X/Y}\otimes L)\ptensor[\,(1+\epsilon m)kr]\otimes \scrO_X(CkV+kE_0)\otimes f^\ast H_k\ptensor[r] \\
&= ( K_{X/Y}\otimes L)\ptensor[\,(1+\epsilon m)kr]\otimes \scrO_X(CkV_{\text{hor}}+kE_1)\otimes f^\ast H_k\ptensor[r]
\end{align*}
where $E_1=CV_{\text{ver}}+E_0$ is $f$-exceptional.
 
In addition, the hypothesis \eqref{eq_cond_thm_CP_canonique:det-im-dir} implies that the relative $m$-Bergman kernel metric $h^{(m)}_{X/Y\!,L}$ on $ K_{X/Y}\ptensor[m]\otimes L\ptensor[m]$ is semi-positively curved, hence by {\hyperref[prop_Bergman-discriminant]{Proposition \ref*{prop_Bergman-discriminant}}} and  \eqref{eq_pull-back-SigmaY} the line bundle 
\[
 K_{X/Y}\otimes L\otimes\scrO_X(-bV_{\text{hor}})
\]
is pseudoeffective, where $b:=\min_{j}\{a_j-1\}$. There the $\QQ$-line bundle
\[
\left((1+\epsilon m)kr+\frac{Ck}{b}\right)( K_{X/Y}+L)+kE_1+rf^\ast A_Y-\epsilon krf^\ast\det\!f_\ast\left( K_{X/Y}\ptensor[m]\otimes L\ptensor[m]\right)
\]
is pseudoeffective. By letting $k\to+\infty$ and by putting
\[
E:=\frac{b}{(1+\epsilon m)br+C}E_1\quad\text{et}\quad\epsilon_0:=\frac{\epsilon br}{(1+\epsilon m)br+C}
\]
we obtain the pseudoeffectivity of the $\QQ$-line bundle \eqref{eq_canonique:det-im-dir}, thus prove the {\hyperref[thm_CP_canonique:det-im-dir]{Theorem \ref{thm_CP_canonique:det-im-dir}}}.
\end{proof}

Now turn to the proof of {\hyperref[main-thm_II]{Main Theorem, Part (I)}}. In fact one can prove a stronger result as following, whose proof is quite similar to \cite[Corollary 4.1]{CP17}:

\begin{thm}
\label{thm_Viehweg_Iitaka-det-gros}
Let $f: X\to Y$ be a surjective morphism between compact Kähler manifolds such that its general fibre $F$ is connected. Let $\Delta$ be an effective $\QQ$-divisor on $X$ such that $(X,\Delta)$ is klt. Suppose that there exists an integer $m>0$ such that $m\Delta$ is an integral divisor and the determinant line bundle $\det\!f_*(K_{X/Y}\ptensor[m]\otimes\scrO_X(m\Delta))$ is big. Then
\begin{equation}
\label{eq_inegalite-conj-Iitaka_det-gros}
\kappa(X,K_X+\Delta)\geqslant \kappa(Y)+\kappa(F,K_F+\Delta_F).
\end{equation}
where $\Delta_F=\Delta|_F$. Moreover, if $\kappa(Y)\geqslant 0$ then we have
\[
\kappa(X,K_X+\Delta)\geqslant\kappa(F,K_F+\Delta_F)+\dim Y.
\]
\end{thm}

\begin{proof}
The key point of the proof has already been proved in {\hyperref[thm_CP_canonique:det-im-dir]{Theorem \ref*{thm_CP_canonique:det-im-dir}}}, the rest is  quite similar to that of {\hyperref[thm_Viehweg_Iitaka-type-gen]{Theorem \ref*{thm_Viehweg_Iitaka-type-gen}}}.
Nevertheless, in order to apply {\hyperref[thm_CP_canonique:det-im-dir]{Theorem \ref*{thm_CP_canonique:det-im-dir}}}, one should be able to add an "exceptional" positivity to the pluricanonical bundle; therefore we take a diagram as in {\hyperref[lemme_Viehweg-aplatissement]{Lemma \ref*{lemme_Viehweg-aplatissement}}}\,:
\begin{center}
\begin{tikzpicture}[scale=2.0]
\node (A) at (0,0) {$Y$,};
\node (B) at (0,1) {$X$};
\node (A') at (-1,0) {$Y'$};
\node (B') at (-1,1) {$X'$};
\path[->,font=\scriptsize,>=angle 90]
(B) edge node[right]{$f$} (A)
(B') edge node[left]{$f'$} (A')
(A') edge node[below]{$\pi_Y$} (A)
(B') edge node[above]{$\pi_X$} (B);
\end{tikzpicture}
\end{center}
and take $\Delta'$ an effective $\QQ$-divisor on $X'$ as in {\hyperref[lemme_preservation-klt]{Lemma \ref*{lemme_preservation-klt}}}, so that every $f'$-exceptional divisor is also $\pi_X$-exceptional and that $(X',\Delta')$ is klt. By construction, the morphism $f'$ is smooth over $Y'_0:=\pi_X\inv Y_0$ where $Y_0$ denotes the (analytic) Zariski open subset of $Y$ over which $f$ is smooth; $\pi_X|_{X'_0}:X'_0\to X_0$ with $X'_0:=(f')\inv Y'_0$ and $X_0:=f\inv Y_0$ is an isomorphism. In particular, for $y'\in Y'_0$, we have an isomorphism $X'_{y'}\simeq X_y$ (with $y:=\pi_Y(y')$) between complex manifolds, implying that $F'\simeq F$ where $F'$ denotes the general fibre of $f'$; moreover this isomorphism identifies $\Delta'_{F'}:=\Delta'|_{F'}$ to $\Delta_F$.

In addition, we have the following (non-trivial) morphism of base change
\begin{equation}
\label{eq_morphsime-change-base}
\pi_Y^\ast f_\ast\left( K_{X/Y}\ptensor[m]\otimes\scrO_X(m\Delta)\right)\rightarrow f'_\ast\left(\pi_X^\ast\left( K_{X/Y}\ptensor[m]\otimes\scrO_X(m\Delta)\right)\right)\hookrightarrow f'_\ast\left( K_{X'/Y'}\ptensor[m]\otimes\scrO_{X'}(m\Delta')\right),
\end{equation}
where the first morphism is an isomorphism over $Y'_0$, and the second morphism is injective, which is a result of the fact that $ K_{Y'/Y}$ is $\pi_Y$-exceptional and effective; $\pi_Y$ being birational, the line bundle 
\[
\pi_Y^\ast\det\!f_\ast\left( K_{X/Y}\ptensor[m]\otimes\scrO_X(m\Delta)\right)
\]
is big over $Y'$, therefore the morphism \eqref{eq_morphsime-change-base} implies that the determinant line bundle $\det\!f'_\ast\left( K_{X'/Y'}\ptensor[m]\otimes\scrO_{X'}(m\Delta')\right)$ is also big over $Y'$. In particular
\begin{equation}
\label{eq_non-vanishing-f'}
f'_\ast\left( K_{X'/Y'}\ptensor[m]\otimes\scrO_{X'}(m\Delta')\right)\neq 0.
\end{equation}
Hence we can apply {\hyperref[thm_CP_canonique:det-im-dir]{Theorem \ref{thm_CP_canonique:det-im-dir}}} to $f'$, and we get an $f'$-exceptional $\QQ$-divisor $E'$ and $\epsilon_0\in\QQ_{>0}$ such that the $\QQ$-line bundle
\[
K_{X'/Y'}+\Delta'+E'-\epsilon_0(f')^\ast\det\!f'_\ast\left(K_{X'/Y'}\ptensor[m]\otimes\scrO_{X'}(m\Delta')\right)
\]
is pseudoeffective. Let us fix a very ample line bundle $A_{Y'}$ on $Y'$ such that $A_{Y'}\otimes K_{Y'}\inv$ is ample and that the Seshadri constant $\epsilon(A_{Y'}\otimes K_{Y'}\inv, y)>\dim Y$ for general $y\in Y'$. Then by Kodaira's Lemma (c.f. \cite[Lemma 2.60, pp.~67-68]{KM98}), there exists a  integer $m_1>0$ sufficiently large and divisible and a pseudoeffective line bundle $L_0$ on $X$ such that $m_1\Delta'$ and $m_1E'$ are integral divisors and that
\[
K_{X'/Y'}\ptensor[m_1]\otimes\scrO_{X'}(m_1(\Delta'+E'))=(f')^\ast A_{Y'}\ptensor[2]\otimes L_0.
\]

Now $L_0$ being pseudoeffective, we can equip it with a singular Hermitian metric $h_{L_0}$ whose curvature current is positive. Since $\Delta'$ is klt, we can find $m_2\in\ZZ_{>0}$ sufficiently large and divisible such that 
\[
\scrJ\left(h_{\Delta'}\otimes h_{L_0}\ptensor[\frac{1}{m_2}]\right)=\scrO_{X'}.
\]
Now we can endow $K_{X'/Y'}\ptensor[m_2]\otimes\scrO_{X'}(m_2\Delta')\otimes L_0$ with the relative $m_2$-Bergman kernel metric $h_{X'/Y'\!,m_2\Delta'+L_0}^{(m_2)}$\,, then by applying {{\hyperref[rmq_klt-gen-iso]{Remark \ref*{rmq_klt-gen-iso}}}} to the $\QQ$-line bundle $N=\Delta'+\frac{1}{m_2}L_0$ we have 
\[
\Coh^0(F',K_{F'}\otimes \left.N_{m_2-1}\right|_{F'}\otimes\scrJ(h_{N_{m_2-1}}\big|_{F'}))=\Coh^0(F',K_{F'}\otimes\left.N_{m_2-1}\right|_F),
\]
where $N_{m_2-1}:=K_{X'/Y'}\ptensor[(m_2-1)]\otimes \scrO_{X'}(m_2\Delta')\otimes L_0$ equipped with singular Hermitian metric
\[
h_{N_{m_2-1}}:=\left(h_{X'/Y'\!,m_2\Delta'+L_0}^{(m_2)}\right)\ptensor[\frac{m_2-1}{m_2}]\otimes h_{\Delta'}\otimes h_{L_0}\ptensor[\frac{1}{m_2}]\,.
\]
Now by {\hyperref[thm_Deng_OT]{Theorem \ref*{thm_Deng_OT}}} we have a surjection
\[
\Coh^0(K_{X'}\otimes N_{m_2-1}\otimes(f')^\ast(A_{Y'}\otimes K_{Y'}\inv))\twoheadrightarrow \Coh^0(K_{F'}\otimes\left.N_{m_2-1}\right|_{F'})
\]
which amounts to:
\[
\Coh^0(X',K_{X'/Y'}\ptensor[m_2]\otimes\scrO_{X'}(m_2\Delta')\otimes L_0\otimes (f')^\ast A_{Y'})\twoheadrightarrow\Coh^0(F',K_{F'}\ptensor[(m_1+m_2)]\otimes \scrO_{F'}((m_1+m_2)\Delta'_{F'}))
\]
where the space on the right hand is non-vanishing by \eqref{eq_non-vanishing-f'}.

Then we can apply {\hyperref[lemme_kod-eff+pull-ample]{Lemma \ref*{lemme_kod-eff+pull-ample}}} to $L=K_{X'/Y'}\ptensor[m_2]\otimes\scrO_{X'}(m_2\Delta')\otimes L_0\otimes (f')^\ast A_{Y'}$ and obtain the following equality:
\begin{align}
\kappa(X,K_{X/Y}+\Delta) & =\kappa(X', (m_1+m_2)(K_{X'/Y'}+\Delta')+m_1E') \nonumber\\
&=\kappa(X', m_2K_{X'/Y'}+m_2\Delta'+L_0+2(f')^\ast A_{Y'}) \nonumber\\
&=\kappa(F', K_{F'}+\Delta'_{F'}))+\dim Y' \nonumber\\
&=\kappa(F,K_F+\Delta_F)+\dim Y. \label{eq_inegalite-canonique-rel:Iitaka}
\end{align}
If $\kappa(Y)=-\infty$ then the inequality \eqref{eq_inegalite-conj-Iitaka_det-gros} is automatically established; otherwise, there is an integer $k>0$ such that $K_Y\ptensor[k]$ is effective, then by \eqref{eq_inegalite-canonique-rel:Iitaka} we get
\begin{align*}
\kappa(X, K_X+\Delta) &= \kappa(X, kK_{X/Y}+k\Delta+kf^\ast K_Y) \\
&\geqslant \kappa(X, kK_{X/Y}+k\Delta) \\
&\geqslant \kappa(F, K_F+\Delta_F)+\dim Y.
\end{align*}
\end{proof}


\section{Albanese Maps of Compact Kähler Manifolds of log Calabi-Yau Type}
\label{sec_Alb-kod=0}
Having demonstrated {\hyperref[thm_Viehweg_Iitaka-type-gen]{Theorem \ref*{thm_Viehweg_Iitaka-type-gen}}}, one can follow the same argument as that in \cite{Kaw81} to deduce {\hyperref[thm_Kawamata_Ab-Var]{Theorem \ref*{thm_Kawamata_Ab-Var}}}. The first step of the proof, as in \cite{Kaw81}, is to obtain the following proposition, which generalize \cite[Theorem 10.9, pp.120-123]{Uen75} and can be regarded as an analytic version of \cite[Theorem 13]{Kaw81}: 
\begin{prop}
\label{prop_Kaw_fini-tore}
Let $p: V\to T$ be a finite morphism with $V$ a compact normal complex variety and $T$ a complex torus. Then $\kappa(V)\geqslant 0$, and there is a subtorus $S$ of $T$ and a (projective) normal variety of general type $W$, which is finite over $T/S$, such that 
\begin{itemize}
\item[\rm(a)]\label{prop_Kaw_fini-tore_a} there is an analytic fibre space $\phi_p: V\to W$ whose general fibre is equal to $\tilde S$, a complex torus which admits a finite étale cover $\tilde S\to S$ over $S$.
\item[\rm(b)]\label{prop_Kaw_fini-tore_b} $\kappa(W)=\kappa(V)=\dim W$;
\end{itemize}
\end{prop} 
Before demonstrating the proposition, let us recall of the following lemma, which can be proved by following the same argument as in \cite{Mil86} (combined with an analytic version of \cite[Proposition (1.3)]{Art86})
\begin{lemme}[analytic version of {\cite[Theorem 3.1]{Mil86}}]
\label{lemme_app-mero-tore}
A meromorphic mapping from a complex manifold to a complex torus is always defined everywhere, thus gives rise to a morphism.
\end{lemme}

\begin{proof}[Proof of {\hyperref[prop_Kaw_fini-tore]{Proposition \ref*{prop_Kaw_fini-tore}}}]
By \cite[Lemma 6.3, pp.~66-67]{Uen75}, we have $\kappa(V)$ $\geqslant\kappa(T)=0$. Let $\Phi_V: V'\to W'$ be the Iitaka fibration of $V$ where $V'$ is smooth model lying over $V$ and $W'$ a complex manifold. For a general point $w'$ in $W'$, $V_{w'}$ and $V'_{w'}$ are bimeromorphic and thus $\kappa(V_{w'})=\kappa(V'_{w'})=0$, where $V_{w'}$ is the image of $V'_{w'}$ in $V$. Denote $S_{w'}=p(V_{w'})$ for $w'\in W'$, then by \cite[Theorem 10.9, pp.120-123]{Uen75} we have $\kappa(S_{w'})\geqslant 0$; on the other hand, $p$ being a finite morphism, \cite[Lemma 6.3, pp.~66-67]{Uen75} implies that $\kappa(S_{w'})\leqslant \kappa(V_{w'})=0$ for $w'\in W'$ general, hence $\kappa(S_{w'})=0$ pour $w'$ general. Again by \cite[Theorem 10.9, pp.120-123]{Uen75}, $S_{w'}$ is a translate of a subtorus of $T$ for $w'$ general (in particular, $S_{w'}$ is isomorphic to a complex torus for $w'$ general). Therefore $\{S_{w'}\}_{w'\in W'}\subseteq T\times W'$ forms an analytic family of complex varieties over $W'$ whose general fibre is isomorphic to a complex torus; but $T$ has only countably many subtori, hence there exists a subtorus $S$ of $T$ such that for very general $w'$ we have $S_{w'}\simeq S$. Now by (the analytic version of) \cite[Lemma 14]{Kaw81} (applied to $f=(V'\to V\to T/S)$ and $g=\Phi_V$), this implies that we have a meromorphic mapping $q':W'\dashrightarrow T/S$; but $W'$ is smooth, then by {\hyperref[lemme_app-mero-tore]{Lemma \ref{lemme_app-mero-tore}}} the meromorphic mapping $q'$ is everywhere defined, hence it is a morphism and makes the following diagram commutative:     
\begin{center}
\begin{tikzpicture}[scale=2.0]
\node (A) at (0,0) {$V$};
\node (T) at (0,-1) {$T$};
\node (T') at (0,-2) {$T/S$.};
\node (A') at (-2,0) {$V'$};
\node (B') at (-2,-2) {$W'$};
\node (B) at (-1,-1) {$W$};
\path[->,font=\scriptsize,>=angle 90]
(A) edge node[right]{$p$} (T)
(A) edge node[above left]{$\exists\,\phi_p$} (B)
(T) edge node[right]{quotient}(T')
(A') edge node[left]{$\Phi_V$} (B')
(A') edge node[above]{bimeromorphic} (A)
(B') edge node[below]{$q'$} (T')
(B') edge (B)
(B) edge node[below left]{$q$} (T');
\end{tikzpicture}
\end{center}
Note $W'_0=q'(W')=\text{image of }V\text{ in }T/S$. Since we have 
\[
\dim W'=\dim V'-\dim V'_w=\dim p(V)-\dim S_w=\dim W'_0\,, 
\]
$q'$ is generically finite. Take a Stein factorization of $q'$:  $q:W\to T/S$ is a finite morphism and $W'\to W$ an analytic fibre space; in addition, $W$ is normal by our construction. Since $q'$ is generically finite, $W'\to W$ is a fortiori bimeromorphic, in particular we have
\begin{equation}
\label{eq_cond_fib-Iitaka}
\dim W=\dim W'=\kappa(V).
\end{equation}
By construction $q:W\to T/S$ also gives a Stein factorization of the proper morphism $V'\xrightarrow{\Phi_V}W'\xrightarrow{q'}T/S$ since $\Phi_{V\ast}\scrO_{V'}=\scrO_{W'}$\,; $V'\to V$ being bimeromorphic morphism, the fibres of the morphism $V'\to V$ are connected, hence they are contracted by $V'\xrightarrow{\Phi_V}W'\to W$, by \cite[\S 1.3, Lemma 1.15, pp.12-13]{Deb01} there is a morphism $\phi_p:V\to W$ such that $q\circ\phi_p$ is equal to the morphism $V\xrightarrow{p}T\to T/S$. Moreover, since $V'\to V$ is bimeromorphic, Zariski's Main Theorem (c.f. \cite[Corollary 1.14, p.~12]{Uen75}) implies that $\phi_{p\ast}\scrO_V=\scrO_W$, hence $\phi_p$ is an analytic fibre space; by our construction $\phi_p$ and $q$ provide a Stein factorization of the proper morphism $V\to T\to T/S$\,.  
In order to prove {\hyperref[prop_Kaw_fini-tore_a]{(a)}} it suffices to apply \cite[Theorem 22]{Kaw81}, which is an analytic version of {\cite[Main Theorem]{KV80}}. In fact, since $\kappa(V_w)=0$ for $w\in W$ general ($W'\to W$ bimeromorphic), \cite[Theorem 22]{Kaw81} implies that the finite surjective morphism $p|_{V_w}: V_w\to p(V_w)=S_w\simeq S$ is a finite étale cover, hence $V_w$ is isomorphic to a (disjoint) union of copies of $\tilde S$ with $\tilde S$ a complex torus admitting a finite étale cover over $S$; $V_w$ being connected, we must have $V_w\simeq \tilde S$. In other word, $\phi_p$ is an analytic fibre space whose general fibre equals to $\tilde S$. Let us remark that one can further prove that $\phi_p$ a principle $\tilde S$-bundle, for this it suffices to apply \cite[Theorem 8]{AS60} which ensures that the deformation of a complex torus is still a complex torus.
 
In order to establish {\hyperref[prop_Kaw_fini-tore_b]{(b)}}, it remains, by virtue of \eqref{eq_cond_fib-Iitaka}, to show that that $W$ is of general type, i.e. $\kappa(W)=\dim W$. To see this, we will follow the same argument as in \cite[Proof of Theorem 10.9, p.~122]{Uen75}. Assume by contradiction that $\kappa(W)<\dim W$, then one can apply the above argument to the finite morphism $q:W\to T/S$ and get the following commutative diagram
\begin{center}
\begin{tikzpicture}[scale=2.5]
\node (A1) at (0,0) {$T/S_1$,};
\node (B1) at (0,1) {$W_1$};
\node (A) at (-1,0) {$T/S$};
\node (B) at (-1,1) {$W$};
\node (C) at (-2,0) {$T$};
\node (D) at (-2,1) {$V$};
\path[->,font=\scriptsize,>=angle 90]
(B1) edge node[right]{$q_1$} (A1)
(B) edge node[right]{$q$} (A)
(D) edge node[left]{$p$} (C)
(A) edge node[below]{quotient} (A1)
(B) edge node[above]{$\phi_q$} (B1)
(C) edge node[below]{quotient} (A)
(D) edge node[above]{$\phi_p$} (B);
\end{tikzpicture}
\end{center}
where $\dim W_1=\kappa(W)<\dim W$, $S_1$ is a subtorus of $T$ containing $S$, $\phi_q$ is an analytic fibre space whose general fibre is equal to $\tilde S_1$, a complex torus admitting a finite étale cover over $S_1/S$, and $q_1$ is a finite morphism. Then $\phi_q\circ\phi_p: V\to W_1$ is an analytic fibre space whose general fibre $F$ admits an analytic fibre space $\phi_p|_F: F\to \tilde S_1$ whose general fibre is equal to $\tilde S$. $K_{\tilde S_1}$ being trivial, consider the relative Bergman kernel metric $h_{F/\tilde S_1}$ on $ K_F\simeq K_{F/\tilde S_1}$ (c.f. {\hyperref[ss_pos-im-dir_Bergman]{\S\ref*{ss_pos-im-dir_Bergman}}}). Since $K_{F_t}\simeq K_{\tilde S}\simeq\scrO_{\tilde S}$ is trivial for general $t\in\tilde S_1$, then by \eqref{eq_def-Bergman} the local weight of $h_{F/\tilde S_1}$ is a constant psh function, hence $(K_F,h_{F/\tilde S_1})$ is an Hermitian flat line bundle. In particular $\nu(F)=\nu(K_F)=0$, implying that $\kappa(F)\leqslant \nu(F)=0$ (in fact, one can further prove that $\kappa(F)=0$, c.f. {\hyperref[thm_Abundance_kod=0]{Theorem \ref*{thm_Abundance_kod=0}}}). By the easy inequality \cite[Theorem 5.11, pp.~59-60]{Uen75} we have
\[
\kappa(V)\leqslant \kappa(F)+\dim W_1\leqslant \dim W_1<\dim W=\kappa (V),
\]
which is absurd. Therefore we must have $\kappa(W)=\dim W=\kappa (V)$.
\end{proof}
 
\begin{proof}[Proof of {\hyperref[thm_Kawamata_Ab-Var]{Theorem \ref*{thm_Kawamata_Ab-Var}}}]
Take a Stein factorization of the Albanese map of $X$: $f:X\to Y$ is an analytic fibre space and $p:Y\to T:=\Alb_X$ is a finite morphism. Then by {\hyperref[prop_Kaw_fini-tore]{Proposition \ref*{prop_Kaw_fini-tore}}}, one can find a subtorus $S$ of $T$ and a projective variety $Z$ of general type which admits a finite morphism $q:Z\to T/S$ such that there is an Kähler fibre space $\phi_p:Y\to Z$ whose general fibre $\tilde S$ is a complex torus, which is a finite étale cover over $S$.   
\begin{center}
\begin{tikzpicture}[scale=2.5]
\node (T) at (0,0) {$T=\Alb_X$};
\node (B) at (0,1) {$Y$};
\node (A) at (-1,1) {$X$};
\node (S) at (1.25,0) {$T/S$.};
\node (C) at (1.25,1) {$Z$};
\path[->,font=\scriptsize,>=angle 90]
(A) edge node[above]{$f$} (B)
(A) edge node[below left]{$\alb_X$} (T)
(B) edge node[left]{$p$} (T)
(T) edge node[below]{quotient} (S)
(C) edge node[right]{$q$} (S)
(B) edge node[above]{$\phi_p$} (C)
(T) edge [bend right] node[right]{$u$} (B);
\end{tikzpicture}
\end{center}
Since $Z$ is of general type, apply  {\hyperref[thm_Viehweg_Iitaka-type-gen]{Theorem \ref*{thm_Viehweg_Iitaka-type-gen}}} as well as the easy inequality \cite[Theorem 5.11, pp.~59-60]{Uen75} to the Kähler fibre space $f\circ\phi_p:X\to Z$ and we get:
\[
0=\kappa(X,K_X+\Delta)=\kappa(X_z,K_{X_z}+\Delta_z)+\dim Z\geqslant\dim Z,
\]
where $z\in Z$ is a general point and $\Delta_z:=\Delta|_{X_z}$\,. Hence $Z$ must be a singleton. In consequence $Y=\tilde S$ is a complex torus. By the universal property of the Albanese map, we obtain a unique morphism $u:T\to Y$ of complex tori, such that $u\circ\alb_X=f$ (up to change the base point of $\alb_X$); in particular, the fibres of $\alb_X$ are connected, hence $\alb_X$ is also an analytic fibre space, hence a Kähler fibre space, thus proves {\hyperref[thm_Kawamata_Ab-Var]{Theorem \ref*{thm_Kawamata_Ab-Var}}}. Let us remark that $\alb_X$ being an analytic fibre space, then so is $p$ (all its fibres are connected); $p$ is thus a fortiori an isomorphism by Zariski's Main Theorem (c.f. \cite[Theorem 1.11, pp.~9-10]{Uen75}).
\end{proof}

\section{Pluricanonical Version of the Structure Theorem for Cohomology Jumping Loci}
\label{sec_Simpson}
In this section we will prove {\hyperref[thm_Simpson_pluri-klt]{Theorem \ref*{thm_Simpson_pluri-klt}}} by combining the {\hyperref[lemme_cov-trick]{covering trick}} and the main result in \cite{Wan16}. First let us recall some notions: let $V$ be a complex manifold, and let $\scrF$ be a coherent sheaf on $V$, for every $k>0$ denote 
\[
V^i_k(\scrF):=\left\{\;\rho\in\Pic^0(V)\;\big|\;\dimcoh^i(V,\scrF\otimes\rho)\geqslant k\;\right\},
\]
the "$k$-th jumping locus of the $i$-th cohomology". With the help of the Poincaré line bundle on $V\times\Pic^0(V)$, one can express this as the locus where a certain coherent sheaf (in fact, some higher direct image sheaf) of $\Pic^0(V)$ has rank $\geqslant k$, hence $V^{i}_k(\scrF)$ is a closed analytic subspace of $\Pic^0(V)$. The study of the cohomology jumping loci is initiated by the works of Green-Lazarsfeld \cite{GL87, GL91} where they treat the case $\scrF=\scrO_V$. When $\scrF=\Omega_V^p$ for $V$ a smooth projective variety (resp. a compact Kähler manifold) these cohomology jumping loci are described by the result of Simpson \cite{Sim93} (resp. of Wang \cite{Wan16}). Especially the case $g=\id_X$, $m=1$ and $\Delta=0$ is treated in \cite{Wan16}. Now turn to the proof of the theorem. First we reduce to the proof of {\hyperref[thm_Simpson_pluri-klt]{Theorem \ref*{thm_Simpson_pluri-klt}}} to a "key lemma":
 
\begin{proof}[Reduction to the {\hyperref[lemme_cle_Simpson]{Key Lemma}}]
The idea of the proof is similar to that of {\cite[Theorem 10.1]{HPS18}}. In fact, when $\Delta=0$, {\hyperref[thm_Simpson_pluri-klt]{Theorem \ref*{thm_Simpson_pluri-klt}}} is nothing other than the Kähler version of {\cite[Theorem 10.1]{HPS18}}; moreover, as in \cite{HPS18} the theorem is proved by a Baire category theorem argument combined with the following "key lemma":
\begin{lemme}[Key Lemma]
\label{lemme_cle_Simpson}
Every irreducible component of $V^0_k\left(g_\ast\left(K_X\ptensor[m]\otimes\scrO_X(m\Delta)\right)\right)$ is a union of torsion translates of subtori in $\Pic^0(Y)$\,. 
\end{lemme}
Assuming that {\hyperref[lemme_cle_Simpson]{Key Lemma}} is true, let us prove {\hyperref[thm_Simpson_pluri-klt]{Theorem \ref*{thm_Simpson_pluri-klt}}}. Since $\Pic^0(Y)$ is compact, the jumping locus 
\[
V^0_k\left(g_\ast\left(K_X\ptensor[m]\otimes\scrO_X(m\Delta)\right)\right),
\]
as a closed analytic subspace of $\Pic^0(Y)$, has only finite many irreducible components, thus it suffices to prove that every irreducible component of 
\[
V^0_k\left(g_\ast\left(K_X\ptensor[m]\otimes\scrO_X(m\Delta)\right)\right)
\]
is a torsion translate of a subtorus. Let $Z$ be a irreducible component of 
\[
V^0_k\left(g_\ast\left(K_X\ptensor[m]\otimes\scrO_X(m\Delta)\right)\right).
\]
Since $\Pic^0(Y)$ has only countably many subtori (c.f. \cite[Ch.1, Exercise (1-b), p.~20]{BL04}) and countably many torsion points, hence the set of torsion translates of subtori is countable, and then by the {\hyperref[lemme_cle_Simpson]{Key Lemma}} $Z$ is a (at most) countable union of torsion translates of subtori: we can write $Z=\bigcup_{n\in\NN} E_n$\,. By the Baire category theorem ($Z$ is (locally) compact, hence it is a Baire space: every countable union of closed subsets of empty interior is of empty interior), there is one $E_n$\,, say $E_1$\,, which dominates $Z$, a fortiori $Z=E_1$\,. This proves {\hyperref[thm_Simpson_pluri-klt]{Theorem \ref*{thm_Simpson_pluri-klt}}}.
\end{proof} 

In the following two subsections we will prove the "Key Lemma".

\begin{rmq}
\label{rmq_lemme_cle_Simpson}
Remark that in order to prove {\hyperref[lemme_cle_Simpson]{Key Lemma}} it suffices to show that every point of
\[
V^0_k \left(g_\ast\left(K_X\ptensor[m]\otimes\scrO_X(m\Delta)\right)\right)
\]
is in a torsion translate of a subtorus contained in $V^0_k\left(g_\ast\left(K_X\ptensor[m]\otimes\scrO_X(m\Delta)\right)\right)$. In fact, assume this to be true, and let $Z$ be an irreducible component of
\[
V^0_k\left(g_\ast\left(K_X\ptensor[m]\otimes\scrO_X(m\Delta)\right)\right),
\]
with $Z_0$ be the dense (analytic Zariski) open subset of $Z$ complementary to the other irreducible components of
\[
V^0_k\left(g_\ast\left(K_X\ptensor[m]\otimes\scrO_X(m\Delta)\right)\right);
\]
then $Z_0$ is contained in a union of torsion translates of subtori: $Z_0\subseteq\bigcup_{\lambda}E_\lambda$\,, with each $E_\lambda\subseteq Z$ being a torsion translate of a subtorus. Hence $Z=\bigcup_\lambda E_\lambda$ by the density of $Z_0$\,; moreover, this union must be countable, as explained in the proof above. Then by a Baire category theorem argument we get {\hyperref[lemme_cle_Simpson]{Key Lemma}}.
\end{rmq}

\subsection{Result of Wang and Reduction to the Case \texorpdfstring{$g=\id$}{text}}
\label{ss_Simpson_Wang}

In this subsection we consider the case where $m=1$ and $\Delta=0$, this is also the case considered by Simpson and Wang. In particular, if $g=\id$, {\hyperref[thm_Simpson_pluri-klt]{Theorem \ref*{thm_Simpson_pluri-klt}}} is proved by Botong Wang in \cite{Wan16}; effectively, he proves the more general: 
\begin{prop}[{\cite[Corollary 1.4]{Wan16}}]\label{prop_Wang}
Let $V$ a compact Kähler manifold, then each $V_k^i(\Omega_V^p)$ is a finite union of torsion translates of subtori in $\Pic^0(V)$. 
\end{prop}

In the sequel we concentrate on the case $i=0$, as in {\hyperref[thm_Simpson_pluri-klt]{Theorem \ref*{thm_Simpson_pluri-klt}}}. For every integer $k>0$ and for every coherent sheaf $\scrF$ on $X$, by the projection formula we have:
\begin{align}
V_k^0(g_\ast\scrF) &=\left\{\,\rho\in\Pic^0(Y)\,\big|\,\dimcoh^0(Y,g_\ast\scrF\otimes\rho)\geqslant k\,\right\}=\left\{\,\rho\in\Pic^0(Y)\,\big|\,\dimcoh^0(X, \scrF\otimes g^\ast\rho)\geqslant k\,\right\} \nonumber\\
&=(g^\ast)\inv\left(V_k^0(\scrF)\cap\Image g^\ast\right) \label{eq_lieu-saut--formule-proj}
\end{align}
where $g^\ast: \Pic^0(Y)\to\Pic^0(X)$ is the morphism of complex tori given by $L\mapsto g^\ast L$. Then the following lemma permit us to reduce to the case $g=\id$:

\begin{lemme}
\label{lemme_morphisme-tores}
Let $\alpha: T_1\to T_2$ a morphism of complex tori. Let $t\in T_2$ a torsion point and $S\subseteq T_2$ a subtorus. Then $\alpha\inv(t+S)$ is also a torsion translate of a subtorus in $T_1$.
\end{lemme}

\begin{proof}
By \cite[\S 1.2, Théorème 2.3, p.~7]{Deb99} $\alpha$ can be factorized as 
\begin{displaymath}
T_1\xtwoheadrightarrow{\text{quotient}}T_1/(\Ker\alpha)^0\xrightarrow[\text{isogeny}]{\bar \alpha}T_1/\Ker\alpha=\Image\alpha\xhookrightarrow{\text{inclusion}}T_2\,.
\end{displaymath}
Thus it suffices to prove the lemma in the following three cases: 
\begin{itemize}
\item $\alpha$ is the quotient by a subtorus, 
\item $\alpha$ is an isogeny,
\item $\alpha$ is the inclusion of a subtorus.
\end{itemize}
Each of theses cases can be done by elementary linear algebra.
\end{proof}

In particular we obtain immediately:
\begin{prop}
\label{prop_Simpson_kah_i=0}
Let $g:X\to Y$ a morphism between compact Kähler manifolds. Then for every $k>0$ and for every $0\leqslant p\leqslant n$, $V^0_k(g_\ast \Omega_X^p)$ is a finite union of torsion translates of subtori in $\Pic^0(Y)$\,.
\end{prop}

\subsection{Proof of the \texorpdfstring{"Key Lemma"}{text}}
\label{ss_Simpson_lemme-cle}
  
Not turn to the demonstration of the {\hyperref[lemme_cle_Simpson]{Key Lemma}}. It proceeds in four steps:

\paragraph{\quad (A) Reduction to the case $g=\id$.\\}
\label{lemme_cle_Simpson_demo-A}
First apply the formula \eqref{eq_lieu-saut--formule-proj} to $\scrF=K_X\ptensor[m]\otimes\scrO_X(m\Delta)$ and then by {\hyperref[lemme_morphisme-tores]{Lemma \ref*{lemme_morphisme-tores}}} we see that the {\hyperref[lemme_cle_Simpson]{Key Lemma}} is true for $V_k^0\left(g_\ast\left(K_X\ptensor[m]\otimes\scrO_X(m\Delta)\right)\right)$ as soon as it holds for $V_k^0\left(K_X\ptensor[m]\otimes\scrO_X(m\Delta)\right)$. In consequence we can suppose that $g=\id$ (and $X=Y$). 

\paragraph{\quad (B) Case $m=1$ and $\Delta=0$.\\}
\label{lemme_cle_Simpson_demo-B}
This is nothing other than {\hyperref[prop_Simpson_kah_i=0]{Proposition \ref*{prop_Simpson_kah_i=0}}} for $p=n$.


\paragraph{\quad (C) Case $m=1$ and $\Delta$ is of SNC support. \\}
\label{lemme_cle_Simpson_demo-C}
In this step, we consider the case where $m=1$ and $\Delta$ is an effective $\QQ$-divisor of SNC support; in addition, we do not require $\Delta$ to be an integral divisor, but only assume that it is given by a line bundle $L^+$, i.e. there is a line bundle $L^+$, $(L^+)\ptensor[N]\simeq\scrO_X(N\Delta)$ for any $N\in\ZZ_{>0}$ which makes $N\Delta$ an integral divisor. In this case, the {\hyperref[lemme_cle_Simpson]{Key Lemma}} can be deduced from the {\hyperref[lemme_cov-trick]{Covering Trick}} combined with the following auxiliary result (c.f. also \cite[Lemma 3.1]{Wan16}):
\begin{lemme}[analytic  version of {\cite[Lemma 10.3]{HPS18}}]
\label{lemme_semi-cont}
Let $\scrF$ and $\scrG$ be coherent sheaves on $X$ such that $\scrF$ is a direct summand of $\scrG$. Then for $\forall i\in\NN$ and $\forall k\in\ZZ_{>0}$\,, each irreducible component of $V^i_k (\scrF)$ is also an irreducible component of $V^i_l(\scrG)$ for some $l\geqslant k$.
\end{lemme}
\begin{proof}
This is simply a result of Grauert's  semi-continuity theorem (c.f. \cite[\S III.4, Theorem 4.12(i), p.~134]{BS76})
\end{proof}

Now let $L^+$ be the line bundle given by $\Delta$. Since $(X,\Delta)$ is a klt pair, then $\lfloor\Delta\rfloor=0$. Moreover, $\Delta$ being a $\QQ$-divisor of SNC support, then for any $N$ making $N\Delta$ an integral divisor, we can construct by {\hyperref[lemme_cov-trick]{Lemma \ref*{lemme_cov-trick}}} a generically finite morphism $f:V\to X$ of
compact Kähler manifolds such that
\[
f_\ast K_V\simeq\bigoplus_{i=0}^{N-1}K_X\otimes(L^+)\ptensor[i]\otimes\scrO_X(-\lfloor i\Delta\rfloor).
\]
By {\hyperref[lemme_semi-cont]{Lemma \ref*{lemme_semi-cont}}} each irreducible component of $V_k^0\left(K_X\otimes L^+\right)$ is also a irreducible component of a certain $V^0_l(f_\ast K_V)$ for some $l>0$. Then by Step {\hyperref[lemme_cle_Simpson_demo-B]{(B)}} (or {\hyperref[prop_Simpson_kah_i=0]{Proposition \ref*{prop_Simpson_kah_i=0}}}), every irreducible component of $V^0_k\left(K_X\otimes L^+\right)$ is a torsion translate of a subtorus in $\Pic^0(X)$\,.

\paragraph{\quad (D) General case.\\}
\label{lemme_cle_Simpson_demo-D}
In order to prove the general case we use a reduction to the case of Step {\hyperref[lemme_cle_Simpson_demo-C]{(C)}}. This reduction process is inspired by\cite[\S 1.A-1.C]{CKP12}, whose idea has already appeared in \cite{Bud09}.
Let $L$ be a point in 
\[
V_k^0\left(K_X\ptensor[m]\otimes\scrO_X(m\Delta)\right)\subseteq\Pic^0(X)\,,
\]
we will prove in the sequel that there exists a torsion translate of a subtorus contained in
\[V_k^0\left(K_X\ptensor[m]\otimes\scrO_X(m\Delta)\right)\]
which contains $L$. $\Pic^0(X)$ being complex torus, thus divisible, then we can write $L=mL_0=L_0\ptensor[m]$ for some $L_0\in\Pic^0(Y)$. Then we have $\dimcoh^0(X,L_{m,\Delta})\geqslant k$,
where
\[
L_{m,\Delta}:=K_X\ptensor[m]\otimes\scrO_X(m\Delta)\otimes L_0\ptensor[m].
\]
Now take a log resolution $\mu:X'\to X$ for both $\Delta$ and the linear series $\left|L_{m,\Delta}\right|$.
Then we can write
\begin{align}
K_{X'}\ptensor[m]\otimes\scrO_{X'}(m\Delta') &\simeq \mu^\ast(K_X\ptensor[m]\otimes\scrO_X(m\Delta))\otimes\scrO_{X'}(\sum_{i\in I^+}ma_iE_i)\,, \label{eq_lemme_cle_Simpson_demo_log-discrepance}\\
\mu^\ast\left|L_{m,\Delta}\right| &= \left|\mu^\ast L_{m,\Delta}\right| = F_{m,\Delta}+\left|M_{m,\Delta}\right|\,,\nonumber
\end{align}
where:
\begin{itemize}
\item $\left\{\,E_i\,\big|\,i\in I\right\}$ denotes the set of $\mu$-exceptional prime divisors, and
\[
a_i:=a(E_i,X,\Delta)
\]
denotes the discrepancy of $E_i$ with respect to the pair $(X,\Delta)$; $I^+$ (resp. $I^-$) is the set of indices $i$ such that $a_i>0$ (resp. $a_i<0$). 
\item $\Delta'$ is an effective $\QQ$-divisor on $X'$ as in the proof of {\hyperref[lemme_preservation-klt]{Lemma \ref*{lemme_preservation-klt}}}, i.e.
\[
\Delta':=\mu_\ast\inv\!\!\Delta-\sum_{i\in I^-}a_iE_i\,,
\]
by {\hyperref[lemme_preservation-klt]{Lemma \ref*{lemme_preservation-klt}}} the pair $(X',\Delta')$ is also klt.
\item $F_{m,\Delta}$ (resp. $M_{m,\Delta}$) is the fix part (resp. mobile part) of the linear series $\mu^\ast\left|L_{m,\Delta}\right|$; by construction, $\left|M_{m,\Delta}\right|$ is base point free. 
\end{itemize}
By construction ($\mu$ being a log resolution of $\Delta$ and of $\left|L_{m,\Delta}\right|$), $m\Delta'+\sum_{i\in I}E_i$ and $F_{m,\Delta}+\sum_{i\in I}E_i$ are (integral) divisors of SNC support. Let $H$ be a general member in $\left|M_{m,\Delta}\right|$, then $H$ has no common component either with $F_{m,\Delta}$ or with $\sum_{i\in I}E_i$ or with $\Delta'$; by Bertini's theorem, $H$ is smooth (in particular $H$ is reduced), $H+F_{m,\Delta}+\sum_{i\in I}E_i$ is of SNC support. Denote
\[
F'_{m,\Delta}:=F_{m,\Delta}+\sum_{i\in I^+}ma_iE_i\,,
\]
this is a divisor of SNC support, which is equal to the fix part of the linear series $\left|L'_{m,\Delta}\right|$ where
\[
L'_{m,\Delta}:=K_{X'}\ptensor[m]\otimes\scrO_{X'}(m\Delta')\otimes\mu^\ast L_0\ptensor[m]\,.
\]
Put 
\begin{align*}
\mu_\ast\inv\Delta &:=\sum_{j\in J}d_jD_j\,,\quad d_j\in\QQ_{>0}, \\
b_j &:=\text{coefficient of }D_j\text{ in }F_{m,\Delta}\,,\quad j\in J\,, \\
b_i &:=\text{coefficient of }E_i\text{ in }F_{m,\Delta}\,,\quad i\in I^-\,,
\end{align*}
and take 
\begin{align*}
\bar\Delta &:= \Delta'-\sum_{j\in J}\min(d_j,\frac{b_j}{m})D_j-\sum_{i\in I^-}\min(-a_i,\frac{b_i}{m})E_i\,, \\
\bar F_{m,\Delta} &:= F'_{m,\Delta}-\sum_{j\in J}\min(md_j,b_j)D_j-\sum_{i\in I^-}\min(-ma_i,b_i)E_i\,,
\end{align*}
so that $\bar\Delta$ and $\bar F_{m,\Delta}$ have no common component. We see that $\bar\Delta\leqslant\Delta'$, $\bar F_{m,\Delta}\leqslant F'_{m,\Delta}$\,. Now consider the line bundle 
\[
\bar L_{m,\Delta}:=K_{X'}\ptensor[m]\otimes\scrO_{X'}(m\bar\Delta)\otimes\mu^\ast L_0\ptensor[m],
\]
then $\bar F_{m,\Delta}$ equals to the fix part of the linear series $\left|\bar L_{m,\Delta}\right|$, hence we have
\[
\left|\bar L_{m,\Delta}\right|=\bar F_{m,\Delta}+\left|M_{m,\Delta}\right|.
\]
In addition we have
\begin{align*}
\bar L_{m,\Delta}\otimes\scrO_{X'}(-\bigl\lfloor\frac{m-1}{m}\bar F_{m,\Delta}\bigr\rfloor) &= K_{X'}\ptensor[m]\otimes\scrO_{X'}(m\bar\Delta)\otimes\mu^\ast L_0\ptensor[m]\otimes\scrO_{X'}(-\bigl\lfloor\frac{m-1}{m}\left(\bar F_{m,\Delta}+H\right)\bigr\rfloor) \\
&\simeq K_{X'}\otimes\scrO_{X'}(\Delta^+)\otimes\mu^\ast L_0
\end{align*}
where the $\QQ$-divisor
\[
\Delta^+:=\bar\Delta+\bigl\{\frac{m-1}{m}\left(\bar F_{m,\Delta}+H\right)\bigr\}\,.
\]
Since $H$ has no common component with either $\bar\Delta$ or $\bar F_{m,\Delta}$, hence 
\[
\Delta^+=\bar\Delta+\left\{\frac{m-1}{m}\bar F_{m,\Delta}\right\}+\frac{m-1}{m}H\,;
\]
but $H$ is reduced, $\bar\Delta$ and $\bar F_{m,\Delta}$ have no common components, then the coefficients of the irreducible components in $\Delta^+$ are all $<1$; since $\Delta^+$ is of SNC support, then \cite[Corollary 2.31(3), p.~53]{KM98} implies that the pair $(X',\Delta^+)$ is klt. A priori $\scrO_{X'}(\Delta^+)$ is only a $\QQ$-line bundle, but by our construction $\Delta^+$ is given by a line bundle 
\[L^+:=\scrO_{X'}(\Delta^+)=\bar L_{m,\Delta}\otimes\scrO_{X'}(-\bigl\lfloor\frac{m-1}{m}\bar F_{m,\Delta}\bigr\rfloor)\otimes K_{X'}\inv\otimes\mu^\ast L_0\inv\,.
\]
Moreover, we have 
\[
\dimcoh^0(X',K_{X'}\otimes L^+\otimes\mu^\ast L_0)=\dimcoh^0(X',\bar L_{m,\Delta}\otimes\scrO_{X'}(-\bigl\lfloor\frac{m-1}{m}\bar F_{m,\Delta}\bigr\rfloor))\geqslant\dimcoh^0(X',M_{m,\Delta})\geqslant k,
\]
which means that $\mu^\ast L_0\in V_k^0(K_{X'}\otimes L^+)$.  Let $W'$ an irreducible component  $V_k^0\left(K_{X'}\otimes L^+\right)$ containing $\mu^\ast L_0$\,. By Step {\hyperref[lemme_cle_Simpson_demo-C]{(C)}} $W'$ is a torsion translate of subtorus, then we can write $W'=\beta_{\text{tor}}+T'_0$ with $\beta_{\text{tor}}$ a torsion point in $\Pic^0(X')$ and $T'_0$ a subtorus, thus 
\[
(m-1)\mu^\ast L_0+W'=(m-1)\beta_{\text{tor}}+T'_0
\]
is also a torsion translate of a subtorus as $(m-1)\beta_{\text{tor}}$ is also a torsion point of $\Pic^0(X')$. In addition, $(m-1)\mu^\ast L_0+W'$ contains $\mu^\ast L=m\mu^\ast L_0$ as $\mu^\ast L_0\in W'$. It remains to see that $(m-1)\mu^\ast L_0+W'$ is contained in $V_k^0\left(K_{X'}\ptensor[m]\otimes\scrO_{X'}(m\Delta')\right)$. In fact, for every $\alpha\in W'$, we have (since $W'\subseteq V^0_k(K_{X'}\otimes L^+)$):
\begin{align*}
\dimcoh^0(X',K_{X'}\ptensor[m]\otimes\scrO_{X'}(m\Delta')\otimes\mu^\ast L_0\ptensor[(m-1)]\otimes\alpha)
&\geqslant \dimcoh^0(X',K_{X'}\ptensor[m]\otimes \scrO_{X'}(m\bar\Delta)\otimes\mu^\ast L_0\ptensor[(m-1)]\otimes\alpha) \\
&= \dimcoh^0(X',\bar L_{m,\Delta}\otimes\mu^\ast\!L_0\inv\otimes\alpha) \\
&\geqslant \dimcoh^0(X',\bar L_{m,\Delta}\otimes\scrO_{X'}(-\bigl\lfloor\frac{m-1}{m}\bar F_{m,\Delta}\bigr\rfloor)\otimes\mu^\ast L_0\inv\otimes\alpha) \\
&= \dimcoh^0(X',K_{X'}\otimes L^+\otimes\alpha)\geqslant k\,.
\end{align*}
Therefore $(m-1)L_0+W'\subseteq V^0_k\left(K_{X'}\ptensor[m]\otimes\scrO_{X'}(m\Delta')\right)$.

In virtue of the isomorphism \eqref{eq_lemme_cle_Simpson_demo_log-discrepance} we have
\begin{align*}
V^0_k\left(K_X\ptensor[m]\otimes\scrO_X(m\Delta)\right)
&= \left\{\,\rho\in\Pic^0(X)\,\big|\,\dimcoh^0(X,K_X\ptensor[m]\otimes\scrO_X(m\Delta)\otimes\rho)\geqslant k\,\right\} \\
&= \left\{\,\rho\in\Pic^0(X)\,|\,\dimcoh^0(X',\mu^\ast\left(K_X\ptensor[m]\otimes\scrO_X(m\Delta)\otimes\rho\right))\geqslant k\,\right\} \\
&= \left\{\,\rho\in\Pic^0(X)\,|\,\dimcoh^0(X',K_{X'}\ptensor[m]\otimes\scrO_{X'}(m\Delta')\otimes\mu^\ast\rho)\geqslant k\,\right\} \\
&= (\mu^\ast)\inv\left(V_k^0\left(K_{X'}\ptensor[m]\otimes\scrO_{X'}(m\Delta')\right)\cap\Image\mu^\ast\right)\,.
\end{align*}
Hence by {\hyperref[lemme_cle_Simpson]{Lemma \ref*{lemme_cle_Simpson}}}, 
\[
W:=\left(\mu^\ast\right)\inv\left(((m-1)\mu^\ast L_0+W')\cap\Image\mu^\ast\right)
\]
is a torsion translate of a subtorus in $V_k^0\left(K_X\ptensor[m]\otimes\scrO_X(m\Delta)\right)$ and $L\in W$. This proves the {\hyperref[lemme_cle_Simpson]{Key Lemma}}.

\begin{rmq}
\label{rmq_thm_Simpson_pluri-klt}
If $X$ is a smooth projective variety, then one can prove {\hyperref[thm_Simpson_pluri-klt]{Theorem \ref*{thm_Simpson_pluri-klt}}} for log canonical pair $(X,\Delta)$ as follows: 
\begin{itemize}
\item First apply \cite[Theorem 1.1]{BW15} along with \cite[Théorème 8.35(ii), p.~201]{Voi02} to prove the {\hyperref[lemme_cle_Simpson]{Key Lemma}} (thus also {\hyperref[thm_Simpson_pluri-klt]{Theorem \ref*{thm_Simpson_pluri-klt}}}) for $m=1$ and $\Delta$ a reduced SNC divisor (c.f. also \cite{Kaw13}); 
\item Then by \cite[Lemma 2.1]{CKP12} and {\hyperref[lemme_cov-trick]{Lemma \ref*{lemme_cov-trick}}} one can deduce further the {\hyperref[lemme_cle_Simpson]{Key Lemma}} for the case of $m=1$ and $\Delta$ a log canonical $\QQ$-divisor of SNC support, which is given by a line bundle, but is not necessarily an integral divisor; 
\item Finally one can follow the same argument as in Step {\hyperref[lemme_cle_Simpson_demo-D]{(D)}} above to prove the {\hyperref[lemme_cle_Simpson]{Key Lemma}} and thus {\hyperref[thm_Simpson_pluri-klt]{Theorem \ref*{thm_Simpson_pluri-klt}}}.
\end{itemize}
As for the Kähler case, as soon as \cite[Conjecture 1.2]{BW17} is solved, one can prove {\hyperref[thm_Simpson_pluri-klt]{Theorem \ref*{thm_Simpson_pluri-klt}}} for log canonical pair $(X,\Delta)$.
\end{rmq}

\subsection{Kähler version of a result of Campana-Koziarz-\texorpdfstring{P\u aun}{text}}
\label{ss_Simpson_CKP}
Before ending this section, let us prove the following significant corollary of {\hyperref[thm_Simpson_pluri-klt]{Theorem \ref{thm_Simpson_pluri-klt}}}, which generalizes a result of Campana, Koziarz and P\u{a}un to the Kähler case, and will be used in the proof of the {\hyperref[main-thm]{Main Theorem}}. In the algebraic case, it is proved in \cite[Theorem 3.1]{CP11} for $\Delta=0$, and in \cite[Theorem 0.1]{CKP12} for $\Delta$ log canonical.  

\begin{cor}
\label{cor_thm_Simpson_pluri-klt}
Let $(X,\Delta)$ a klt pair with $X$ a Kähler manifold, and let $L_0$ a numerically trivial line bundle on $X$, i.e. $L_0\in\Pic^0(X)$. Then 
\begin{itemize}
\item[\rm(a)]\label{cor_thm_Simpson_pluri-klt_a} $\kappa(X,K_X+\Delta)\geqslant \kappa(X,mK_X+m\Delta+L_0)$, $\forall m\in\ZZ_{>0}$\,. Namely, for any $\QQ$-line bundle\footnote{In fact, since $\Pic^0(X)$ is divisible, this a priori $\QQ$-line bundle $L$ is an "authentic" line bundle.} $L$ on $X$ such that $c_1(L)=c_1(K_X+\Delta)$, we have $\kappa(X,K_X+\Delta)\geqslant\kappa(X,L)$. 
\item[\rm(b)]\label{cor_thm_Simpson_pluri-klt_b} If there is an integer $m>0$ such that $\kappa(X,K_X+\Delta)=\kappa(X,mK_X+m\Delta+L_0)=0$, then $L_0$ is a torsion point in $\Pic^0(X)$.
\end{itemize}
\end{cor}
\begin{proof}
We will follow the argument in \cite{CP11} with a little simplification. First prove the point {\hyperref[cor_thm_Simpson_pluri-klt_a]{(a)}}\,, the demonstration proceeds in three steps:

{\bf Step 1:} Reduction to the case $\kappa(X,K_X\otimes\scrO_X(\Delta))\leqslant 0$.
Assuming {\hyperref[cor_thm_Simpson_pluri-klt_a]{(a)}} for any klt pair $(X,\Delta)$ with $\kappa(X,K_X+\Delta)\leqslant 0$, we will prove it for any klt pair $(X,\Delta)$ with  $\kappa(X,K_X+\Delta)>0$. Let $g:X\dashrightarrow W$ the Iitaka fibration (c.f. \cite[\S 5, Theorem 5.10, p.~58]{Uen75}) of the $\QQ$-line bundle $K_X+\Delta$ and $f:X\dashrightarrow Y$ that of $mK_X+m\Delta+L_0$\,. By {\hyperref[lemme_preservation-klt]{Lemma \ref*{lemme_preservation-klt}}} the point {\hyperref[cor_thm_Simpson_pluri-klt_a]{(a)}} is preserved by log resolutions of $(X,\Delta)$, we can thus suppose that $f$ and $g$ are morphisms (instead of meromorphic mappings). 
\begin{center}
\begin{tikzpicture}[scale=1.5]
\node (A) at (0,0) {$X$};
\node (B) at (2,0) {$W$};
\node (C) at (0,-2) {$Y$};
\node (D) at (1,-0.6) {$G$};
\path[left hook->]
(D) edge (A);
\path[->, font=\scriptsize, >= angle 90]
(A) edge node[above]{$g$} (B)
(A) edge node[left]{$f$} (C)
(D) edge node[below right]{$f|_G$} (C);
\end{tikzpicture}
\end{center}
By construction we have $\dim Y=\kappa(X,mK_X+m\Delta+L_0)$\,, $\dim W=\kappa(X,K_X+\Delta)$\,. Denote by $F$ (resp. by $G$) the general fibre of $f$ (resp. of $g$), the, 
\[
\kappa(X,K_X+\Delta)\geqslant\kappa(X,mK_X+m\Delta+L_0)\;\Leftrightarrow\;\dim W\geqslant\dim Y\;\Leftrightarrow\; \dim G\leqslant\dim F,
\]
then it suffices to prove that $G$ is contracted by $f$ (i.e. $f(G)=\text{pt}$). By adjunction formula the $\QQ$-line bundle
\[
K_G+\Delta_G\simeq (K_X+\Delta)|_G
\]
where $\Delta_G:=\Delta|_G$\,, hence $f|_G$ is bimeromorphically equivalent to a meromorphic mapping defined by a sub-linear series of $\left|K_G\ptensor[km]\otimes\scrO_G(km\Delta)\otimes L_0|_G\ptensor[k]\right|$ for some $k$ sufficiently large and divisible\footnote{In the proof of \cite[Theorem 3.1]{CP11}, it is said that $f|_G$ is equal to the Iitaka fibration of $mK_G+m\Delta_G+L_0$; but it is not true in general.}. Therefore it suffices to show 
\[
\kappa(G,mK_G+m\Delta+L_0|_G)=0.
\]
But by our construction 
\[
\kappa(G,K_G+\Delta_G)=\kappa(G,(K_X+\Delta)|_G)=0,
\]
hence our assumption implies that {\hyperref[cor_thm_Simpson_pluri-klt_a]{(a)}} holds for the klt pair $(G,\Delta_G)$. Since $L_0|_G\in\Pic^0(G)$ we have
\[
\kappa(G,mK_G+m\Delta_G+L_0|_G)\leqslant\kappa(G,K_G+\Delta_G)=0.
\]

{\bf Step 2:} By the precedent step, we can assume that $\kappa(X,K_X+\Delta)\leqslant 0$. If $\kappa(X,mK_X+m\Delta+L_0)=-\infty$, then the inequality is automatically established, hence we can assume that $\kappa(X,mK_X+m\Delta+L_0)\geqslant 0$; in addition, up to replacing $m$ and $L_0$ by a multiple, we can assume that $m\Delta$ is an integral divisor and
\[
\Coh^0(X, K_X\ptensor[m]\otimes\scrO_X(m\Delta)\otimes L_0)\neq 0.
\]
For every integer $k>0$ denote 
\[
r_k: =\dimcoh^0(X,K_X\ptensor[km]\otimes\scrO_X(km\Delta)\otimes L_0\ptensor[k])>0.
\]
Then $L_0\ptensor[k]\in V^0_{r_k}\left(K_X\ptensor[km]\otimes\scrO_X(km\Delta)\right)\subseteq V^0_1\left(K_X\ptensor[km]\otimes\scrO_X(km\Delta)\right)$, thus by {\hyperref[thm_Simpson_pluri-klt]{Theorem \ref*{thm_Simpson_pluri-klt}}}, $L_0\ptensor[k]\in \beta_{\tor}+T_0\subseteq V_{r_k}^0\left(K_X\ptensor[km]\otimes\scrO_X(km\Delta)\right)$ for $\beta_{\tor}$ a torsion point in $\Pic^0(X)$ and $T_0$ a subtorus; in particular, $\beta_{\tor}\in V_{r_k}^0\left(K_X\ptensor[km]\otimes\scrO_X(km\Delta)\right)$. Let $m_0>0$ an integer such that $\beta_{\tor}\ptensor[m_0]\simeq \scrO_X$. Then 
\begin{equation}
\label{eq_cor_thm_Simpson_pluri-klt_a}
\dimcoh^0(X,K_X\ptensor[kmm_0]\otimes\scrO_X(kmm_0\Delta))
\geqslant\dimcoh^0(X,K_X\ptensor[km]\otimes\scrO_X(km\Delta)\otimes\beta_{\tor})\geqslant r_k\,.
\end{equation}

{\bf Step 3:} We claim that for every $k>0$ we have $r_k=1$. Otherwise there is an integer $k$ such that $r_k>1$, then $\dimcoh^0(X,K_X\ptensor[kmm_0]\otimes\scrO_X(kmm_0\Delta))\geqslant r_k>1$, which implies that 
\[
\kappa(X,K_X+\Delta)=\kappa(X,kmm_0K_X+kmm_0\Delta)>0,
\] 
contradicting the hypothesis that $\kappa(X,K_X+\Delta)\leqslant 0$. Hence $r_k=1$ for every $k>0$, and thus $\kappa(X,mK_X+m\Delta+L_0)=0$. On the other hand, by the inequality \eqref{eq_cor_thm_Simpson_pluri-klt_a} we have 
\[
\dimcoh^0(X,K_X\ptensor[mm_0]\otimes\scrO_X(mm_0\Delta))\geqslant r_1=1,
\]
implying that $\kappa(X,K_X+\Delta)\geqslant 0$. Hence $\kappa(X,K_X+\Delta)\geqslant \kappa(X,mK_X+m\Delta+L_0)$, which proves {\hyperref[cor_thm_Simpson_pluri-klt_a]{(a)}}.

Now turn to the proof of {\hyperref[cor_thm_Simpson_pluri-klt_b]{(b)}}: assume by contradiction that there is a line bundle $L\in\Pic^0(X)$ with $L$ non-torsion such that $\kappa(X,mK_X+m\Delta+L)=\kappa(X,K_X+\Delta)=0$ for some $m>0$. Up to replacing $m$ and $L$ by a multiple, we can assume that $m\Delta$ is an integral divisor and that $\dimcoh^0(X,K_X\ptensor[m]\otimes\scrO_X(m\Delta)\otimes L)=1$, then $L\in V^0_1\left(K_X\ptensor[m]\otimes\scrO_X(m\Delta)\right)$. By {\hyperref[thm_Simpson_pluri-klt]{Theorem \ref*{thm_Simpson_pluri-klt}}} there exists $\beta_{\tor}\in\Pic^0(X)_{\tor}$ and $T_0$ a subtorus in $\Pic^0(X)$ such that $L\in\beta_{\tor}+T_0\subseteq V^0_1\left(K_X\ptensor[m]\otimes\scrO_X(m\Delta)\right)$, then we can write $L=\beta_{\tor}\otimes F$ with $F\in T_0$\,. By our assumption $L$ is not a torsion point in $\Pic^0(X)$, hence $F$ cannot be trivial and thus $T_0$ is not reduced to a singleton. In consequence there is a (non-trivial) one-parameter subgroup $(F_t)_{t\in\RR}$ in $T_0$ passing through $F$ (by choosing an isomorphism $T_0\simeq\CC^q/\Gamma$, we can take $F_t=t\cdot F$), then for every $t\in\RR$\,, $\beta_{\tor}\otimes F_t\in \beta_{\tor}+T_0\subseteq V^0_1\left(K_X\ptensor[m]\otimes\scrO_X(m\Delta)\right)$\, hence there is a non-zero section $s_t$ in 
\[
\Coh^0(X,K_X\ptensor[m]\otimes\scrO_X(m\Delta)\otimes\beta_{\tor}\otimes F_t).
\]
We claim that:
\begin{claim}[$\boldsymbol{\ast}$]
\label{claim_cor_thm_Simpson_pluri-klt-b}
There is a $t\in\RR_{>0}$ such that the sections $s_t\otimes s_{-t}$ and $s_0\ptensor[2]$ are not linearly independent in $\Coh^0(X,K_X\ptensor[2m]\otimes\scrO_X(2m\Delta)\otimes\beta_{\tor}\ptensor[2])$. 
\end{claim}
In fact, this leads to a contradiction: we have immediately 
\[
\dimcoh^0(X, K_X\ptensor[2m]\otimes\scrO_X(2m\Delta)\otimes \beta_{\tor}\ptensor[2])\geqslant 2,
\]
which implies that 
\[
\kappa(X,K_X+\Delta)=\kappa(X,K_X\ptensor[2m]\otimes\scrO_X(2m\Delta)\otimes\beta_{\tor}\ptensor[2])\geqslant 1,
\]
and this contradicts the hypothesis that $\kappa(X,K_X\otimes\scrO_X(\Delta))=0$. Therefore {\hyperref[cor_thm_Simpson_pluri-klt_b]{(b)}} is proved.

Let us prove the {\hyperref[claim_cor_thm_Simpson_pluri-klt-b]{Claim ($\boldsymbol{\ast}$)}}. Assume by contradiction that $s_t\otimes s_{-t}$ are $s_0\ptensor[2]$ are linearly dependent for every $t\in\RR$. Then $\forall t\in\RR$\,, $\divisor(s_t)+\divisor(s_{-t})=2\divisor(s_0)$; in particular, $\divisor(s_t)\leqslant 2\divisor(s_0)$ for every $t\in\RR_{>0}$\,. Take $\epsilon$ sufficiently small such that $t\mapsto F_t$ is injective for $t\in\,]\!-\epsilon, \epsilon[$\,. By Dirichlet's drawer principle, there are $t_1,t_2\in\, ]0,\epsilon[$ such that $\divisor(s_{t_1})=\divisor(s_{t_2})$, hence the divisor 
\[
0=\divisor(s_{t_2})-\divisor(s_{t_1})\in \left|F_{t_2}\otimes F_{t_1}\inv\right|,
\]
which implies that $F_{t_1}=F_{t_2}$ in $\Pic^0(X)$ with $t_1,t_2\in ]0,\epsilon[$; but this contradicts our hypothesis on $\epsilon$. This proves {\hyperref[claim_cor_thm_Simpson_pluri-klt-b]{Claim ($\boldsymbol{\ast}$)}}.
\end{proof}

As a by-product of {\hyperref[cor_thm_Simpson_pluri-klt_a]{Corollary \ref*{cor_thm_Simpson_pluri-klt}(a)}} we obtain the following special case of the Kähler version of the (generalized) log Abundance Conjecture by using the divisorial Zariski decomposition (c.f.\cite[Definition 3.7]{Bou04}):
\begin{thm}
\label{thm_Abundance_kod=0}
Let $(X,\Delta)$ be a klt pair with $X$ a compact Kähler manifold whose numerical dimension $\nu(X,K_X+\Delta)=0$, then $\kappa(X,K_X+\Delta)=0$.
\end{thm}
\begin{proof}
For the definition of the numerical dimension of (non necessarily nef) $\QQ$-line bundles (or cohomology classes in $\Coh^{1,1}(X,\RR)$) over a compact Kähler manifold, c.f. \cite[\S 18.13, p.~198]{Dem10}.
Since $\nu(K_X+\Delta)=0$, the $\QQ$-line bundle $K_X+\Delta$ is pseudoeffective, hence we can consider the divisorial Zariski decomposition (c.f. \cite[Definition 3.7]{Bou04} and \cite[\S 18.12(d), p.~195]{Dem10}) of its first Chern class:
\[
c_1(K_X+\Delta)=\left\{N\bigl(c_1(K_X+\Delta)\bigr)\right\}+\langle c_1(K_X+\Delta)\rangle.
\]
By hypothesis $\nu(c_1(K_X+\Delta))=0$, which means that $\langle c_1(K_X+\Delta)\rangle=0$; in other word, the $\QQ$-line bundle $K_X+\Delta$ is numerically equivalent to the effective $\RR$-divisor $N=N\bigl(c_1(K_X+\Delta)\bigr)$, a fortiori $N$ is an $\QQ$-divisor. Therefore by {\hyperref[cor_thm_Simpson_pluri-klt_a]{Corollary \ref*{cor_thm_Simpson_pluri-klt}(a)}}, we have
\[
\kappa(K_X+\Delta)\geqslant \kappa(N)\geqslant 0.
\]
Finally by \cite[\S 18.15, p.~199]{Dem10} we get $\kappa(K_X+\Delta)=0$.
\end{proof}

\section{Kähler Version of \texorpdfstring{$C_{n,m}^{\log}$}{text} for Fibre Spaces over Complex Tori}
\label{sec_demo-main-thm}
In this section, we will prove our {\hyperref[main-thm]{Main Theorem}}. To this end, we do some reductions by an induction on the dimension of $T$ and by applying {\hyperref[thm_pos_im-dir]{Theorem \ref*{thm_pos_im-dir}}}, {\hyperref[thm_Viehweg_Iitaka-det-gros]{Theorem \ref*{thm_Viehweg_Iitaka-det-gros}}} and {\hyperref[thm_Kawamata_Ab-Var]{Theorem \ref*{thm_Kawamata_Ab-Var}}}; at last, we deduce {\hyperref[main-thm]{Theorem \ref*{main-thm}}} from {\hyperref[cor_thm_Simpson_pluri-klt]{Corollary \ref*{cor_thm_Simpson_pluri-klt}}}. 

\subsection{Reduction to the case \texorpdfstring{$T$}{text} is a simple torus}
\label{ss_demo-main-thm_reduction-tore-simple}
By an induction on $\dim T$ we can assume that $T$ is a simple torus, i.e. admitting no non-trivial subtori. In fact, if $T$ is not simple, take a non-trivial subtorus $S\subseteq T$ and denote by $q:T\to T/S$ the canonical morphism (of complex analytic Lie groups), this is a Kähler fibre space (more precisely a principle $S$-bundle). We obtain thus a Kähler fibre space $f'=q\circ f: X\to T/S$, and then by induction hypothesis we have
\[
\kappa(X,K_X+\Delta)\geqslant\kappa(F',K_{F'}+\Delta_{F'}),
\]
where $\Delta_{F'}:=\Delta|_{F'}$ with $F'$ the general fibre $f'$. In addition, $f|_{F'}:F'\to S$ is also a Kähler fibre space of general fibre $F$ over a complex torus $S$ of dimension $<\dim T$, hence by induction hypothesis we have
\[
\kappa(F',K_{F'}+\Delta_{F'})\geqslant\kappa(F,K_F+\Delta_F),
\]
thus we get
\[
\kappa(X,K_X+\Delta)\geqslant\kappa(F,K_F+\Delta_F).
\]

\subsection{Dichotomy according to the determinant bundle}
\label{ss_demo-main-thm_dichotomie-det}
For positive integer $m$ such that $m\Delta$ is an integral divisor, consider the direct image
\[
\scrF_{m,\Delta}:=f_\ast\left(K_X\ptensor[m]\otimes\scrO_X(m\Delta)\right)=f_\ast\left(K_{X/T}\ptensor[m]\otimes\scrO_X(m\Delta)\right).
\] 
If $\kappa(F,K_F+\Delta_F)=-\infty$ then {\hyperref[main-thm_II]{Part (I\!I)}} of the {\hyperref[main-thm]{Main Theorem}} is automatically established; hence we can assume that $\kappa(F,K_F+\Delta_F)\geqslant 0$. In consequence for $m$ sufficiently divisible $\scrF_{m,\Delta}\neq 0$. Let us denote by $\scrM$ the set of positive integers $m$ such that $m\Delta$ is an integral divisor and that $\scrF_{m,\Delta}\neq 0$, then we can suppose that $\scrM\neq\varnothing$, this is moreover an additive subset of $\ZZ$. By {\hyperref[thm_pos-im-pluri-can]{Theorem \ref*{thm_pos-im-pluri-can}}}, for $\forall m\in\scrM$ the torsion free sheaf $\scrF_{m,\Delta}$ admits a semi-positively curved metric $g_{X/T\!,\Delta}^{(m)}$; in addition, the induced metric $\det\!g_{X/T\!,\Delta}^{(m)}$ on its determinant bundle $\det\!\scrF_{m,\Delta}$ is of curvature current
\[
\theta_{m,\Delta}:=\Theta_{\det\!g_{X/T\!,\Delta}^{(m)}}(\det\!\scrF_{m,\Delta})\geqslant 0.
\]
In particular, the line bundle $\det\!\scrF_{m,\Delta}$ is pseudoeffective on $T$ for every $m\in\scrM$. By {\hyperref[ss_demo-main-thm_reduction-tore-simple]{\S \ref*{ss_demo-main-thm_reduction-tore-simple}}} we can assume that $T$ is a simple torus, hence \cite[Proposition 2.2]{Cao15b} (c.f. also \cite[Theorem 3.3]{CP17}) implies that we fall into the following two cases:
\begin{itemize}
\item Either $\theta_{m,\Delta}\not\equiv 0$, in this case $T$ is an Abelian variety equipped with $\det\!\scrF_{m,\Delta}$ an ample line bundle; 
\item Or $\theta_{m,\Delta}\equiv 0$, in this case $\det\!\scrF_{m,\Delta}$ is a numerically trivial line bundle, and thus {\hyperref[cor_im-pluri-can_plat]{Corollary \ref*{cor_im-pluri-can_plat}}} implies that $(\scrF_{m,\Delta},g_{X/T\!,\Delta}^{(m)})$ is a Hermitian flat vector bundle.
\end{itemize}

\subsection{Reduction to the case where the determinant of the direct image is numerically trivial}
\label{ss_demo-main-thm_reduction-det=0}
We assume that the determinant bundle $\det\!\scrF_{m,\Delta}$ is numerically trivial for every $m\in\scrM$. Otherwise, by {\hyperref[ss_demo-main-thm_dichotomie-det]{\S \ref*{ss_demo-main-thm_dichotomie-det}}} there is an integer $m\in\scrM$ such that the determinant bundle $\det\!\scrF_{m,\Delta}$ is ample, in which case {\hyperref[main-thm]{Main Theorem, Part (I\!I)}} is simply a result of {\hyperref[main-thm_I]{Part (I)}}; 
hence in order to finish the proof of {\hyperref[main-thm_II]{Main Theorem, Part (I\!I)}} one only need to tackle the case that that the determinant bundle $\det\!\scrF_{m,\Delta}$ is numerically trivial for every $m\in\scrM$, which implies that $(\scrF_{m,\Delta},g_{X/T\!,\Delta}^{(m)})$ is a Hermitian flat vector bundle for every $m\in\scrM$.   

\subsection{Reduction to the case \texorpdfstring{$\kappa\leqslant 0$}{text}}
\label{ss_demo-main-thm_reduction-kappa<=0}
In this subsection we will demonstrate that we can reduce to the case $\kappa(X,K_X+\Delta)\leqslant 0$, which is an observation dating back to Kawamata, c.f. \cite[\S 3, Proof of Claim 2, pp.~256-266]{Kaw81}. Suppose that {\hyperref[main-thm_II]{Main Theorem, Part (I\!I)}} holds true for klt pair $(X,\Delta)$ with $\kappa(X,K_X+\Delta)\leqslant 0$. Now take a klt pair $(X,\Delta)$ such that $\kappa(X,K_X+\Delta)\geqslant 1$. Up to replacing $X$ by a superior bimeromorphic model, we can suppose that the Iitaka fibration of $K_X+\Delta$ is a morphism, denoted by
\[
\phi: X\to Y,
\]
whose general fibre is $G$. Then $\dim Y=\kappa(X,K_X+\Delta)>0$ and $\kappa(G,K_G+\Delta_G)=0$ where $\Delta_G:=\Delta|_G$. Consider 
\[
f|_G:G\to f(G)=:S\subseteq T,
\]
and take a Stein factorization of $f|_G$:
\begin{center}
\begin{tikzpicture}[scale=1.5]
\node (S) at (0,0) {$S$.};
\node (G) at (0,2) {$G$};
\node (S') at (1.2,1.1) {$S'$};
\path[->, font=\scriptsize, >=angle 90]
(G) edge node[left]{$f|_G$} (S)
(S') edge (S)
(G) edge (S');
\end{tikzpicture}
\end{center}
\paragraph{\quad Cas 1: $S\neq T$.\\}
$T$ being a simple torus, \cite[Theorem 10.9, pp.~120-123]{Uen75} implies that $S$ is of general type, then so is $S'$ by \cite[Lemma 6.3, p.~66-67]{Uen75}. By {\hyperref[thm_Viehweg_Iitaka-type-gen]{Theorem \ref*{thm_Viehweg_Iitaka-type-gen}}}, for general $s\in S'$ we have
\[
0=\kappa(G,K_G+\Delta_G)=\kappa(G_s\,,K_{G_s}+\Delta_{G_s})+\dim S'=\kappa(G_s\,,K_{G_s}+\Delta_{G_s})+\dim S,
\]
where $\Delta_{G_s}:=\Delta|_{G_s}=\Delta_G|_{G_s}$. This forces $\dim S=\dim S'=0$, hence $f(G)=\text{pt}$, and in consequence $G$ is contained in $F$. Therefore $\phi|_F:F\to h(F)\subseteq Y$ is a Kähler fibre space of general fibre $G$, and thus by the easy inequality \cite[Lemma 5.11, pp.~59-60]{Uen75} we obtain (noting that $\Delta_G=\left.\Delta_F\right|_G$):
\[
\kappa(F,K_F+\Delta_F)\leqslant \kappa(G,K_G+\Delta_G)+\dim h(F)=\dim h(F)\leqslant\dim Y=\kappa(X,K_X+\Delta).
\]

\paragraph{\quad Cas 2: $S=T$.\\}
First we prove that
$S'\to S$ is a finite étale cover (thus $S'$ is also a complex torus) with the help of {\hyperref[thm_Kawamata_Ab-Var]{Theorem \ref*{thm_Kawamata_Ab-Var}}}. In fact, let $\alb_G:G\to \Alb_G$ the Albanese map of $(G,y)$ with base point $y$ such that $f(y)=e\in T$. By the universal property of the Albanese map we get a (unique) morphism $u:\Alb_G\to T$ of complex tori (a morphism of complex analytic Lie groups) making the following diagram commutative:  
\begin{center}
\begin{tikzpicture}[scale=3.0]
\node (T) at (0,0) {$T$};
\node (G) at (0,1) {$G$};
\node (S') at (-0.5,-0.5) {$S'$};
\node (T') at (1,0) {$T'$.};
\node (A) at (1,1) {$\Alb_G$};
\path[->, font=\scriptsize, >=angle 90]
(G) edge node[left]{$f|_G$} (T)
(S') edge (S)
(G) edge[bend right] (S')
(G) edge node[above]{$\alb_G$} (A)
(A) edge (T')
(T') edge (T)
(A) edge node[above left]{$\exists!\;u$} (T)
(S') edge[bend right] node[above left]{$\simeq$} (T');
\end{tikzpicture}
\end{center}
Since $f|_G$ is surjective, then so is $u$. By \cite[Théorème 2.3, p.~7]{Deb99} $u$ can be factorized as $\Alb_G\to T'\to T$ with $\Alb_G\to T'$ the quotient by $\Ker(u)^0$ and $T'\to T$ a finite étale cover. As $\kappa(G,K_G+\Delta_G)=0$, then by {\hyperref[thm_Kawamata_Ab-Var]{Theorem \ref*{thm_Kawamata_Ab-Var}}} the morphism $\alb_G$ is an analytic (Kähler) fibre space, thus so is $G\to T'$. Therefore the construction of Stein factorization implies that $S'$ and $T'$ are isomorphic. In particular, $S'\to T$ is a finite étale cover and thus $S'$ is a complex torus.

Denote by $F'$ the general fibre of $G\to S'$, then for general $t\in T$, we have $G_t\simeq F\cap G$ is finite union of copies of $F'$. Now apply our supposition to $G\to T$ ($\kappa(G,K_G+\Delta_G)=0$) and we get
\[
0=\kappa(G,K_G+\Delta_G)\geqslant\kappa(F',K_{F'}+\Delta_{F'}).
\]
where $\Delta_{F'}:=\Delta|_{F'}=\Delta_G|_{F'}$\,. Furthermore, consider a Stein factorization of $\phi|_F:F\to \phi(F)=:Z\subseteq Y$:
\begin{center}
\begin{tikzpicture}[scale=1.5]
\node (Z) at (0,0) {$Z$.};
\node (F) at (0,2) {$F$};
\node (Z') at (1.2,1.1) {$Z'$};
\path[->, font=\scriptsize, >=angle 90]
(F) edge node[left]{$\phi|_F$} (Z)
(Z') edge (Z)
(F) edge (Z');
\end{tikzpicture}
\end{center}
For $z\in Z$ general $F_z\simeq F\cap G$, hence the general fibre of the analytic fibre space $F\to Z'$ is isomorphic à $F'$. Then by the easy inequality \cite[Lemma 5.11, pp.~59-60]{Uen75} we obtain:
\[
\kappa(F,K_F+\Delta_F)\leqslant\kappa(F',\Delta_{F'}+\Delta_{F'})+\dim Z'\leqslant\dim Z'=\dim Z\leqslant\dim Y=\kappa(X,K_X+\Delta).
\]

\subsection{End of the Proof of the Main Theorem}
\label{ss_demo-main-thm_fin}
By {\hyperref[ss_demo-main-thm_reduction-det=0]{\S \ref*{ss_demo-main-thm_reduction-det=0}}} we have that $(\scrF_{m,\Delta},g_{X/T\!,\Delta}^{(m)})$ is a Hermitian flat vector bundle for every $m\in\scrM$. In other words $\scrF_{m,\Delta}$ is constructed by a unitary representation of the fundamental group (c.f. for example \cite[Proposition 1.4.21, p.~13]{Kob87} or \cite[\S 6, pp.~260-261]{agbook})
\[
\rho_m:\pi_1(T,t_0)\to\Unitary(r_m)
\]
where 
\[
r_m:=\rank\scrF_{m,\Delta}=\dimcoh^0(F,K_F\ptensor[m]\otimes\scrO_F(m\Delta_F)).
\]
Since $\pi_1(T,t_0)$ is an Abelian group, every representation of $\pi_1(T)$ can be decomposed into (irreducible) sub-representations of rank $1$, hence a decomposition of $\scrF_{m,\Delta}$ into (numerically trivial) line bundles:
\begin{equation}
\label{eq_decomp-im-dir-plat}
\scrF_{m,\Delta}=L_1\oplus L_2\oplus\cdots\oplus L_{r_m}\,,\quad\text{with }L_i\in\Pic^0(T),\;\forall i
\end{equation}

\paragraph{\quad Step 1:}
First prove that $\Image(\rho_m)$ is finite for every $m\in\scrM$. In fact, suppose by contradiction that there exists $m\in\scrM$ such that $\Image(\rho_m)$ is infinite, hence there exists $j\in\{1,2\cdots,r_m\}$, say $j=1$, such that $L_j$ is not a torsion point in $\Pic^0(T)$. Consider the natural inclusion $L_1\hookrightarrow \scrF_{m,\Delta}$\,, which induces a non-zero section 
\[
\Coh^0(T,\scrF_{m,\Delta}\otimes L_1\inv)=\Coh^0(X,K_X\ptensor[m]\otimes\scrO_X(m\Delta)\otimes f^\ast L_1\inv).
\]
This implies that $\kappa(X,mK_X+m\Delta+f^\ast L_1)\geqslant 0$. As $f^\ast L_1\in\Pic^0(X)$, by {\hyperref[cor_thm_Simpson_pluri-klt_a]{Corollary \ref*{cor_thm_Simpson_pluri-klt}(a)}} and {\hyperref[ss_demo-main-thm_reduction-kappa<=0]{\S \ref*{ss_demo-main-thm_reduction-kappa<=0}}} we have
\[
\kappa(X,mK_X+m\Delta+f^\ast L_1)\leqslant\kappa(X,K_X+\Delta)\leqslant 0,
\]
hence a fortiori 
\begin{equation}
\label{eq_kappa-equality}
\kappa(X,mK_X+m\Delta+f^\ast L_1)=\kappa(X,K_X+\Delta)=0.
\end{equation}
By {\hyperref[cor_thm_Simpson_pluri-klt_b]{Corollary \ref*{cor_thm_Simpson_pluri-klt}(b)}}, the equality \eqref{eq_kappa-equality} implies that $f^\ast L_1$ is a torsion point in $\Pic^0(X)$, i.e. there is an $e>0$ such that $f^\ast L_1\ptensor[e]\simeq\scrO_X$, meaning that $L_1\ptensor[e]\simeq \scrO_T$ since the morphism
\[
f^\ast:\Pic^0(T)\to\Pic^0(X)
\]
is injective ($f$ being an analytic fibre space). This contradicts our supposition that $L_1$ is not a torsion element in $\Pic^0(T)$. Hence $\Image(\rho_m)$ is finite for each $m\in\scrM$. 

\paragraph{\quad Step 2:}
By the precedent step we see that $\Image(\rho_m)$ is a finite group. Denote 
\[
l_m:=\#\Image(\rho_m).
\]
For every $i\in\{\;1,2,\cdots,r_m\;\}$, consider  the action of $\pi_1(T)$ on $L_i$ via $\rho_m$. For any $t\in T$, the action of $\pi_1(T)$ on $(L_i)_t$ is given by the multiplication by a constant (a $l_m$-th root of unity), hence the action of $\pi_1(T)$ on $(L_i\ptensor[l_m])_t=(L_i)_t\ptensor[l_m]$ is trivial. But the action of $\pi_1(T)$ on $\scrF_{m,\Delta}$ is induced by the parallel transport, hence a vector in $(L_i\ptensor[l_m])_t$ extends via the parallel transport to a section in $\Coh^0(T,L_i\ptensor[l_m])$, c.f. \cite[(6.7) Proposition, p.~261]{agbook}. In particular, for $t\in T$ a general point, choose for every $i\in\{\;1,2,\cdots,r_m\;\}$ a non-zero vector
\[
s_i\in(L_i)_t\subseteq(\scrF_{m,\Delta})_t\simeq\Coh^0(F,K_F\ptensor[m]\otimes\scrO_F(m\Delta_F)), 
\]
then $s_i\ptensor[l_m]\in(L_i\ptensor[l_m])_t$\,, thus $s_i\ptensor[l_m]$ gives rise to a section 
\[
\sigma_i\in\Image\left(\Coh^0(T,\scrF_{m,\Delta}\ptensor[l_m])\to\Coh^0(X,K_X\ptensor[ml_m]\otimes\scrO_X(ml_m\Delta))\right)\subseteq\Coh^0(X,K_X\ptensor[ml_m]\otimes\scrO_X(ml_m\Delta)),
\]
such that
\[
\sigma_i|_{X_t}=\text{image of }s_i\ptensor[l_m]\text{ in }\Coh^0(X_t,K_{X_t}\ptensor[ml_m]\otimes\scrO_{X_t}(ml_m\Delta_t)).
\]
where $\Delta_t:=\Delta|_{X_t}$. In addition, the decomposition \eqref{eq_decomp-im-dir-plat} implies that for $i\neq j$, $s_i\ptensor[l_m]$ and $s_j\ptensor[l_m]$ are linearly independent, then the natural morphism
\[
(\scrF_{m,\Delta}\ptensor[l_m])_t=\Coh^0(X_t,K_{X_t}\ptensor[m]\otimes\scrO_{X_t}(m\Delta_t))\ptensor[l_m]\to\Coh^0(X_t,K_{X_t}\ptensor[ml_m]\otimes\scrO_{X_t}(ml_m\Delta_t)) 
\]
being injective, $\sigma_i$ and $\sigma_j$ are also linearly independent. In consequence, 
\begin{equation}
\label{eq_inegalite_X-F}
\dimcoh^0(X,K_X\ptensor[ml_m]\otimes\scrO_X(ml_m\Delta))\geqslant r_m.
\end{equation}
By {\hyperref[ss_demo-main-thm_reduction-kappa<=0]{\S \ref*{ss_demo-main-thm_reduction-kappa<=0}}} we have $\kappa(X,K_X+\Delta)\leqslant 0$, hence
\[
r_m\leqslant\dimcoh^0(X,K_X\ptensor[ml_m]\otimes\scrO_X(ml_m\Delta))=1.
\]
A fortiori $r_m=1,\;\forall m\in\scrM$. This implies that $\kappa(F,K_F+\Delta_F)=0$. On the other hand, \eqref{eq_inegalite_X-F} implies that $\kappa(X,K_X+\Delta)\geqslant 0$, hence 
\[
\kappa(X,K_X+\Delta)\geqslant\kappa(F,K_F+\Delta_F).
\]

\section{Geometric Orbifold Version of the Main Results}
\label{sec_orbifold}
In this last section, we will prove {\hyperref[main-thm_orbifold]{Theorem \ref*{main-thm_orbifold}}}, in other word, generalize {\hyperref[main-thm_II]{Part (I\!I)}} of our {\hyperref[main-thm]{Main Theorem}}, established in {\hyperref[sec_demo-main-thm]{\S \ref*{sec_demo-main-thm}}}, to the geometric orbifold setting. Along the way, we also show that $C_{n,m}^{\orbifold}$ holds when $(Y,B_{f\!,\Delta})$ is of log general type. Before entering into the proof of theses results, let us first clarify some definitions. Remind that for $f:X\to Y$ analytic fibre space between compact complex manifolds and for $\Delta$ effective $\QQ$-divisor on $X$, the branching divisor $B_{f\!,\Delta}$ is defined as the most effective $\QQ$-divisor on $Y$ such that $f^\ast B_{f\!,\Delta}\leqslant R_{f\!,\Delta}$ modulo exceptional divisors (see below, c.f. also {\hyperref[sec_intro]{Introduction}}); on the other hand, in \cite[Definition 1.29]{Cam04} Frédéric Campana defines a divisor on $Y$ with respect to $f$ and $\Delta$in the setting of geometric orbifolds, named "orbifold base". We will see in the sequel that these two definitions coincide when $(X,\Delta)$ is lc. Let us first recall the definition of Campana: 

\begin{defn}
\label{defn_orbifold-base}
Let $f:X\to Y$ and $\Delta$ as above such that $(X,\Delta)$ is lc. For any prime divisor $G$ on $Y$, write 
\[
f^\ast G=\sum_{j\in J(f\!,G)} \Ramification_{G_j}(f)G_j+(f\text{-exceptional divisor})\,,
\]
where $J(f\!,G)$ is the index set of all prime divisors mapped onto $G$. Then the orbifold base with respect to $f$ and $\Delta$ is defined to be the $\QQ$-divisor
\[
B_{f\!,\Delta}:=\sum_G \left(1-\frac{1}{m(f\!,\Delta;G)}\right)G
\]
where the multiplicity $m(f\!,\Delta;G)$ of $G$ with respect to $f$ and $\Delta$ is defined to be
\[
m(f\!,\Delta;G):=\inf\left\{\,\Ramification_{G_j}(f)m(\Delta;G_j)\,\big|\,j\in J(f\!,G)\,\right\}
\]
with $m(\Delta;G_j)\in\QQ_{\geqslant 1}\cup\{+\infty\}$ satisfying
\[
\order_{G_j}(\Delta)=1-\frac{1}{m(\Delta;G_j)}\,.
\]
\end{defn}
 
Now we have:

\begin{lemme}
\label{lemme_discriminant:orbifold-base}
Let $f:X\to Y$ and $\Delta$ as above such that $(X,\Delta)$ is lc.Let $B_{f\!,\Delta}$ be the orbifold base respect to $f$ and $\Delta$ in the sense of Campana, as defined in {\hyperref[defn_orbifold-base]{Definition \ref*{defn_orbifold-base}}} above. Then there is an $f$-exceptional effective $\QQ$-divisor $E$ such that the $\QQ$-divisor $R_{f\!,\Delta}+E-f^\ast B_{f\!,\Delta}$ is effective; and $B_{f\!,\Delta}$ is the most effective $\QQ$-divisor on $Y$ satisfying this property. 
\end{lemme}

\begin{proof}
The second assertion is evident by construction of $B_{f\!,\Delta}$. In fact, if $B$ is a divisor on $Y$ such that $f^\ast B\leqslant R_{f\!,\Delta}$\,, then for every prime divisor $G$ on $Y$ we have
\begin{align*}
\order_{G_j}(f^\ast B)=\Ramification_{G_j}(f)\order_G(B)\leqslant \order_{G_j}(R_{f\!,\Delta}) &=\Ramification_{G_j}(f)-1+\order_{G_j}(\Delta) \\
&=\Ramification_{G_j}(f)-\frac{1}{m(\Delta;G_j)}\,,\;\forall j\in J(f\!,G)\,,
\end{align*}
where 
\[
f^\ast G=\sum_{j\in J(f\!,G)}\Ramification_{G_j}G_j+(f-\text{exceptional divisor})\,;
\]
this implies that 
\[
\order_G(B)\leqslant 1-\frac{1}{\Ramification_{G_j}(f)m(\Delta;G_j)}\,,\;\forall j\in J(f\!,G)\,,
\]
and hence
\begin{align*}
\order_G(B)\leqslant \inf_{j\in J(f\!,G)}\left(1-\frac{1}{\Ramification_{G_j}(f)m(\Delta;G_j)}\right) &=1-\frac{1}{\inf\left\{\,\Ramification_{G_j}(f)m(\Delta;G_j)\,\big|\,j\in J(f\!,G)\,\right\}} \\
&=\order_G(B_{f\!,\Delta}).
\end{align*}

Now turn to the proof of the first assertion. To this end, it suffices to show that for any prime divisor $D$ on $X$ such that $f(D)$ is a divisor on $Y$ we have
\begin{equation}
\label{eq_inequality-discriminant:orbifold-base}
\order_D(R_{f\!,\Delta})=\order_D(\Sigma_f)+\order_D(\Delta)\geqslant \order_D(f^\ast B_{f\!,\Delta}).
\end{equation}
Let $\Sigma_Y$ be a (reduced) divisor containing $Y\backslash Y_0$ with $Y_0\subset Y$ the smooth locus of $f$ and write 
\[
f^\ast\Sigma_Y=\sum_{i\in I}b_i W_i\,,
\]
then 
\[
\Sigma_f:=\sum_{i\in I^{\divisor}}(b_i-1)W_i.
\]
where $I^{\divisor}$ denotes the set of indices in $I$ such that $f(W_i)$ is a divisor on $Y$. Now we consider separately the two cases:
\subparagraph{Case 1 :} $D\not\subset\Supp(\Sigma_f)$. Then $\order_D(\Sigma_f)=0$ and a general point of $f(D)$ is contained in $Y_0$\,, thus 
\[
f^\ast f(D)=D+(f\text{-exceptional divisor}).
\]
In consequence $\Ramification_D(f)=1$ and $J(f\!,f(D))=\{D\}$, which implies that $m(f\!,\Delta;f(D))=m(\Delta;D)$. Hence
\[
\order_D(f^\ast B_{f\!,\Delta})=\order_{f(D)}(B_{f\!,\Delta})=1-\frac{1}{m(\Delta;D)}=\order_D(\Delta)=\order_D(\Sigma_f)+\order_D(\Delta).
\]

\subparagraph{Case 2 :} $D\subset\Supp(\Sigma_f)$. Then $D=W_i$ for some $i\in I^{\divisor}$. In consequence, $f(W_i)\subset \Supp(\Sigma_Y)$ and
\[
f^\ast f(W_i)=\sum_{j\in J(f\!,f(W_i))}b_j W_j+(f\text{-exceptional divisor}),
\]
with $J(f\!,f(W_i))=\left\{\,j\in I^{\divisor}\,\big|\,f(W_j)=f(W_i)\,\right\}$ and $\Ramification_{W_j}(f)=b_j$. By definition we have
\[
m(f\!,\Delta;f(W_i))=\inf\left\{\,b_jm(\Delta;W_j)\,\big|\,j\in I^{\divisor}\text{ and }f(W_j)=f(W_i)\,\right\}\leqslant b_im(\Delta;W_i).
\]
Hence
\begin{align*}
\order_{W_i}(f^\ast B_{f\!,\Delta}) &=b_i\cdot\order_{f(W_i)}(B_{f\!,\Delta})=b_i\left(1-\frac{1}{m(f\!,\Delta;f(W_i))}\right) \\
&\leqslant 1-\frac{1}{b_im(\Delta;W_i)}
=(b_i-1)+(1-\frac{1}{m(\Delta;W_i)}) \\
&=\order_{W_i}(\Sigma_f)+\order_{W_i}(\Delta).
\end{align*}

In both cases, the inequality \eqref{eq_inequality-discriminant:orbifold-base} is established for prime divisor $D$ vertical w.r.t. $f$, hence the proof is proved.
\end{proof}

\begin{rmq}
\label{rmq_lemme_discriminant:orbifold-base}
As a corollary of the above lemma, one sees clearly:
\begin{itemize}
\item $f^\ast B_{f\!,\Delta}$ being a vertical divisor w.r.t. $f$ (i.e. not dominating $Y$), it is in fact the most effective divisor on $Y$ such that $f^\ast B_{f\!,\Delta}\leqslant R_{f\!,\Delta^{\vertical}}=\Sigma_f+\Delta^{\vertical}$ where $\Delta^{\vertical}$ denotes the vertical part of $\Delta$.
\item If $(X,\Delta)$ is klt and $\scrF_{m,\Delta}:=f_\ast\left(K_{(X,\Delta)/Y}\ptensor[m]\right)\neq 0$ for some $m$ sufficiently large and divisible, one can easily deduce from {\hyperref[prop_Bergman-discriminant]{Proposition \ref*{prop_Bergman-discriminant}}} (applied to $L=\scrO_X(m\Delta^{\horizontal})$ with $\Delta^{\horizontal}$ the horizontal part of $\Delta$) that there is an $f$-exceptional effective $\QQ$-divisor $E$ such that the $\QQ$-line bundle $K_{f\!,\Delta}^{\orbifold}+E$ is pseudoeffective, where the orbifold relative canonical bundle is defined (as a $\QQ$-line bundle) by the formula: 
\[
K_{f\!,\Delta}^{\orbifold}:=K_{(X,\Delta)/(Y,B_{f\!,\Delta})}=K_{X/Y}+\Delta-f^\ast B_{f\!,\Delta}.
\]
\end{itemize}
\end{rmq}

Before proving the {\hyperref[main-thm_orbifold]{Theorem \ref*{main-thm_orbifold}}}, let us first prove the result that the klt version of $C_{n,m}^{\orbifold}$ holds for fibre spaces over bases of general type in the sense of geometric orbifolds:
\begin{thm}
\label{thm_Iitaka-orbifold-type-gen}
Let $f:X\to Y$ be a surjective morphism between compact Kähler manifolds whose general fibre $F$ is connected. Let $\Delta$ be an effective $\QQ$-divisor on $X$ such that $(X,\Delta)$ is klt. Suppose that $(Y,B_{f\!,\Delta})$ is of log general type. Then 
\[
\kappa(X,K_X+\Delta)\geqslant \kappa(F, K_F+\Delta_F)+\dim Y,
\]
where $\Delta_F:=\Delta|_F$.
\end{thm}

Notice that a stronger (log canonical) version of the above theorem is proved in \cite{Cam04} (for $X$ projective) based on the a weak positivity theorem for direct images of twisted pluricanonical bundles. We will give here a new argument depending on the Ohsawa-Takegoshi extension theorem:
\begin{proof}[Proof of {\hyperref[thm_Iitaka-orbifold-type-gen]{Theorem \ref*{thm_Iitaka-orbifold-type-gen}}}]
First, as in the proof of {\hyperref[thm_Viehweg_Iitaka-det-gros]{Theorem \ref*{thm_Viehweg_Iitaka-det-gros}}}, by passing to a higher bimeromorphic model of $f$, we can assume that $f$ is neat and prepared (in virtue of {\hyperref[lemme_Viehweg-aplatissement]{Lemma \ref*{lemme_Viehweg-aplatissement}}} and {\hyperref[lemme_preservation-klt]{Lemma \ref*{lemme_preservation-klt}}}), that is, every $f$-exceptional divisor is also exceptional with respect to some bimeromorphic morphism $X\to X'$ and the the singular locus of $f$ is a (reduced) SNC divisor; in particular, for every effective $f$-exceptional divisor $E_0$ on $X$, we have $\kappa(X,K_X+\Delta)=\kappa(X,K_X+\Delta+E_0)$.

If $\kappa(F,K_F+\Delta_F)=-\infty$ then there is nothing to prove, hence suppose that $\kappa(F,K_F+\Delta_F)\geqslant 0$, this implies that there is $m>0$ sufficiently large and divisible such that $\scrF_{m,\Delta}:=f_\ast\left(K_{(X,\Delta)/Y}\ptensor[m]\right)\neq 0$. By {\hyperref[rmq_lemme_discriminant:orbifold-base]{Remark \ref*{rmq_lemme_discriminant:orbifold-base}}}, there is an effective $f$-exceptional $\QQ$-divisor $E$ such that the $\QQ$-line bundle $K_{f\!,\Delta}^{\orbifold}+E$ is pseudoeffective. Since $(Y,B_{f\!,\Delta})$ is of log general type, $Y$ is projective, one can fix a very ample line bundle $A_Y$ on $Y$ such that the $\QQ$-line bundle $A_Y-K_Y-B_{f\!,\Delta}$ is ample and that the Seshadri constant $\epsilon(A_Y-K_Y-B_{f\!,\Delta}\,,y)>\dim Y$ for general $y$. Now by our hypothesis $K_Y+B_{f\!,\Delta}$ is a big $\QQ$-line bundle, then (up to replacing $m$ by a multiple) we can assume that $m(K_Y+B_{f\!,\Delta})-2A_Y$ is effective. Then we have 
\[
\kappa(X,K_X+\Delta)=\kappa(X,K_X+\Delta+E)\geqslant\kappa(X,mK_{f\!,\Delta}^{\orbifold}+mE+2f^\ast A_Y).
\]
In virtue of {\hyperref[lemme_kod-eff+pull-ample]{Lemma \ref*{lemme_kod-eff+pull-ample}}} it suffices to show that
\[
\Coh^0(X, (K_{f\!,\Delta}^{\orbifold})\ptensor[m]\otimes\scrO_X(mE)\otimes f^\ast A_Y)\neq 0,
\]
which is a direct consequence of the Ohsawa-Takegoshi type extension {\hyperref[thm_Deng_OT]{Theorem \ref*{thm_Deng_OT}}}, as we precise below:

Since $\Delta$ is klt, by {\hyperref[thm_Cao_Bergman]{Theorem \ref*{thm_Cao_Bergman}}} the relative $m$-Bergman kernel metric $h_{X/Y\!,m\Delta^{\horizontal}}$ on $K_{X/Y}\ptensor[m]\otimes\scrO_X(m\Delta^{\horizontal})$ is semipositive (noting that $\Delta^{\horizontal}\big|_F=\Delta_F$). Set
\begin{align*}
L_{m-1} &:=K_{X/Y}\ptensor[(m-1)]\otimes\scrO_X(m\Delta^{\horizontal}), \\
L'_{m-1} &:=L_{m-1}\otimes\scrO_X(mE+m\Delta^{\vertical}-(m-1)f^\ast B_{f\!,\Delta}), 
\end{align*}
respectively equipped with the singular Hermitian metrics:
\begin{align*}
h_{L_{m-1}} &:=\left(h_{X/Y\!,m\Delta^{\horizontal}}^{(m)}\right)\ptensor[\frac{m-1}{m}]\otimes h_{\Delta^{\horizontal}}\,, \\
h_{L'_{m-1}} &:=h_{L_{m-1}}\otimes h_E\ptensor[m]\otimes h_{\Delta^{\vertical}}\ptensor[m]\otimes f^\ast h_{B_{f\!,\Delta}}\ptensor[-(m-1)]\,.
\end{align*}
where $h_{\Delta^{\horizontal}}$\,, $h_{\Delta^{\vertical}}$\,, $h_E$ and $h_{B_{f\!,\Delta}}$ denote the canonical singular metrics defined by the divisors. Then by {\hyperref[prop_Bergman-discriminant]{Proposition \ref*{prop_Bergman-discriminant}}} and {\hyperref[lemme_discriminant:orbifold-base]{Lemma \ref*{lemme_discriminant:orbifold-base}}} the curvature current of $h_{L'_{m-1}}$ satisfies 
\begin{align*}
\Theta_{h_{L'_{m-1}}}(L'_{m-1}) &= \frac{m-1}{m}\Theta_{h_{X/Y\!,m\Delta^{\horizontal}}}(K_{X/Y}\ptensor[m]\otimes\scrO_X(m\Delta^{\horizontal})+[\Delta]+(m-1)[\Delta^{\vertical}] \\
&\quad +m[E]-(m-1)[f^\ast B_{f\!,\Delta}] \\
&\geqslant (m-1)\left([\Sigma_f]+[E]+[\Delta^{\vertical}]-[f^\ast B_{f\!,\Delta}]\right)+[\Delta]+[E] \\
&\geqslant [\Delta]+[E]\geqslant 0.
\end{align*}
Moreover, since $L'_{m-1}|_F=L_{m-1}|_F$ and $h_{L'_{m-1}}|_F=h_{L_{m-1}}|_F$\,, by {\hyperref[rmq_klt-gen-iso]{Remark \ref*{rmq_klt-gen-iso}}} the natural inclusion
\begin{align*}
\Coh^0(F,K_F\otimes L'_{m-1}|_F\otimes\scrJ(h_{L'_{m-1}}|_F)) &=\Coh^0(F,K_F\otimes L_{m-1}|_F\otimes\scrJ(h_{L_{m-1}}|_F)) \\
&\hookrightarrow \Coh^0(F,K_F\otimes L_{m-1}|_F)=\Coh^0(F,K_F\ptensor[m]\otimes \scrO_F(m\Delta_F))
\end{align*}
is an isomorphism (and both of them are consequently non-vanishing). Hence by {\hyperref[thm_Deng_OT]{Theorem \ref*{thm_Deng_OT}}} we get a surjection
\[
\Coh^0(K_X\otimes L'_{m-1}\otimes f^\ast(A_Y\otimes K_{(Y,B_{f\!,\Delta})}\inv )\twoheadrightarrow\Coh^0(F, K_F\ptensor[m]\otimes\scrO_F(m\Delta_F)).
\]
Since 
\[
K_X\otimes L'_{m-1}\otimes f^\ast(A_Y\otimes K_{(Y,B_{f\!,\Delta})}\inv)=(K_{f\!,\Delta}^{\orbifold})\ptensor[m]\otimes\scrO_X(mE)\otimes f^\ast A_Y\,,
\]
this proves the non-vanishing of $\Coh^0((K_{f\!,\Delta}^{\orbifold})\ptensor[m]\otimes\scrO_X(mE)\otimes f^\ast A_Y)$.
\end{proof}

Finally, let us turn to the proof of {\hyperref[main-thm_orbifold]{Theorem \ref*{main-thm_orbifold}}}:
\begin{proof}[Proof of {\hyperref[main-thm_orbifold]{Theorem \ref*{main-thm_orbifold}}}]
Let us proceed by an induction on $\dim T$. If $B_{f,\Delta}=0$, then {\hyperref[main-thm_orbifold]{Theorem \ref*{main-thm_orbifold}}} is reduced to {\hyperref[main-thm_II]{Part (I\!I)}} of the {\hyperref[main-thm]{Main Theorem}}. Hence we assume that $B_{f\!,\Delta}\neq 0$. Then by \cite[Proposition 2.2]{Cao15b}, there is a subtorus $S$ of $T$ of dimension $<\dim T$ and an ample $\QQ$-divisor $H$ on $A:T/S$ such that $\pi^\ast H=B_{f\!,\Delta}$ with $\pi: T\to A=T/S$ the quotient map. 

Now let $f'=\pi\circ f:X\to A$, which is a fibre space with general fibre $F'$. Then $f|_{F'}:F'\to S$ is a fibre space with general fibre $F$. We have $B_{f|_{F'},\Delta_{F'}}\geqslant (B_{f\!,\Delta})|_S$, as one can easily check: for every component $G$ of $(B_{f\!,\Delta})|_S$, it arises from a prime divisor of $X$, hence $B_{f|_{F'},\Delta_{F'}}$ has the same vanishing order over $G$. This is enough for our use; we nevertheless remark that we have in fact the equality $B_{f|_{F'},\Delta_{F'}}=(B_{f\!,\Delta})|_S$ since every component of $B_{f|_{F'},\Delta_{F'}}$ must arise from a divisor on $X$: in fact, every component of $B_{f|_{F'},\Delta_{F'}}$ is either the image of a component of $\Delta_{F'}=\Delta|_{F'}$ or the image of a component of $\Sigma_{f|_{F'}}=(\Sigma_f)|_{F'}$ (we have the equality if we choose $S$ to be a general translate). Now the induction hypothesis gives:
\[
\kappa(F',K_{F'}+\Delta_{F'})\geqslant\kappa(F,K_F+\Delta_F)+\kappa(S,(B_{f\!,\Delta})|_S).
\]
Furthermore, since $\kappa(S,(B_{f\!,\Delta)|_S})\geqslant 0$, we have 
\begin{equation}
\label{eq_Iitaka-orb_ind-hyp}
\kappa(F',K_{F'}+\Delta_{F'})\geqslant\kappa(F,K_F+\Delta_F).
\end{equation}

We claim that 
\begin{equation}
\label{eq_Iitaka-orb_quotient}
\kappa(X,K_X+\Delta)\geqslant\kappa(F',K_{F'}+\Delta_{F'})+\dim A.
\end{equation}
If $\kappa(F',K_{F'}+\Delta_{F'})=-\infty$, then \eqref{eq_Iitaka-orb_quotient} evidently holds. Hence we can assume that $\kappa(F',K_{F'}+\Delta_{F'})\geqslant 0$. In this case, for $m$ sufficiently large and divisible, 
\[
\Coh^0(F,K_{F'}\ptensor[m]\otimes\scrO_{F'}(m\Delta_{F'}))\neq 0.
\]
Since $(X,\Delta)$ is klt, $(F',\Delta_{F'})$ is klt, then by {\hyperref[thm_Cao_Bergman]{Theorem \ref*{thm_Cao_Bergman}}} we can construct the relative $2m$-Bergman kernel metric $h_{X/A,2m\Delta^{\horizontal}}^{(2m)}$ on $K_{X/A}\ptensor[2m]\otimes \scrO_X(2m\Delta^{\horizontal})\simeq K_X\ptensor[2m]\otimes \scrO_X(2m\Delta^{\horizontal})$. Now put \[
L:=K_X\ptensor[(2m-1)]\otimes \scrO_X(2m\Delta+2mE-m(f')^\ast H)
\] 
equipped with the singular Hermitian metric
\[
h_L:=\left(h_{X/A,2m\Delta^{\horizontal}}^{(2m)}\right)\ptensor[\frac{2m-1}{2m}]\otimes h_{\Delta^{\horizontal}}\otimes h_{\Delta^{\vertical}}\ptensor[2m]\otimes h_E\ptensor[2m]\otimes(f')^\ast h_H\ptensor[-m]\,,
\]
where $E$ is an $f$-exceptional effective divisor as in {\hyperref[lemme_discriminant:orbifold-base]{Lemma \ref*{lemme_discriminant:orbifold-base}}} and $h_{\Delta^{\horizontal}}$\,, $h_{\Delta^{\vertical}}$\,, $h_E$ and $h_H$ are the canonical singular metrics defined by the divisors. Then by {\hyperref[prop_Bergman-discriminant]{Proposition \ref*{prop_Bergman-discriminant}}} and {\hyperref[lemme_discriminant:orbifold-base]{Lemma \ref*{lemme_discriminant:orbifold-base}}} the curvature current of $h_L$ satisfies
\begin{align*}
\Theta_{h_L}(L) &=\frac{2m-1}{2m}\Theta_{h^{(2m)}_{X/A,\Delta^{\horizontal}}}\!\!\left(K_{X/A}\ptensor[2m]\otimes\scrO_X(2m\Delta^{\horizontal})\right)+[\Delta^{\horizontal}]+2m[\Delta^{\vertical}]+2m[E]-m[(f')^\ast H] \\
&\geqslant (2m-1)[\Sigma_f]+[\Delta^{\horizontal}]+2m[\Delta^{\vertical}]+2m[E]-m[f^\ast B_{f\!,\Delta}] \\
&= [\Delta]+[E]+(m-1)([\Sigma_f]+[\Delta^{\vertical}]+[E])+m([\Sigma_f]+[\Delta^{\vertical}]+[E]-[f^\ast B_{f\!,\Delta}]) \\
&\geqslant [\Delta]+[E]+(m-1)([\Sigma_f]+[\Delta^{\vertical}]+[E])\geqslant 0.
\end{align*}
Since $h_L|_F=h_{L_{2m-1}}|_F$, where 
$L_{2m-1}:=K_X\ptensor[(2m-1)]\otimes\scrO_X(2m\Delta^{\horizontal})$ equipped with the singular metric 
\[
h_{L_{2m-1}}:=\left(h_{X/A,\Delta^{\horizontal}}^{(2m)}\right)\ptensor[\frac{2m-1}{2m}]\otimes h_{\Delta^{\horizontal}},
\]
then by {\hyperref[lemme_gen-iso_im-pluri-can]{Lemma \ref*{lemme_gen-iso_im-pluri-can}}} we see that the natural inclusion 
\[
f'_\ast\left(K_{X/A}\otimes L\otimes\scrJ(h_L)\right)\hookrightarrow f'_\ast(K_{X/A}\otimes L)
\]
is generically an isomorphism, hence by {\hyperref[thm_pos_im-dir]{Theorem \ref*{thm_pos_im-dir}}} the canonical $L^2$ metric on 
\[
f'_\ast(K_{X/A}\otimes L)=f'_\ast(K_{X/A}\ptensor[2m]\otimes\scrO_X(2m\Delta+2mE))\otimes H\ptensor[-m]
\]
is semi-positively curved. In particular its determinant is pseudoeffective, which implies that $\det\!f'_\ast(K_{X/A}\ptensor[2m]\otimes\scrO_X(2m\Delta+2mE))$ is big on $A$. Since $f'_\ast(K_{X/A}\ptensor[2m]\otimes\scrO_X(2m\Delta+2mE))$ and $f'_\ast(K_{X/A}\ptensor[2m]\otimes\scrO_X(2m\Delta))$ are equal in codimension $1$, hence 
\[
\det\!f'_\ast(K_{X/A}\ptensor[2m]\otimes\scrO_X(2m\Delta+2mE))=\det\!f'_\ast(K_{X/A}\ptensor[2m]\otimes\scrO_X(2m\Delta)),
\]
implying that $\det\!f'_\ast(K_{X/A}\ptensor[2m]\otimes\scrO_X(2m\Delta))$ is big on $A$. Since $\kappa(A)=0$, \eqref{eq_Iitaka-orb_quotient} results from {\hyperref[thm_Viehweg_Iitaka-det-gros]{Theorem \ref*{thm_Viehweg_Iitaka-det-gros}}}.

At last, by combining \eqref{eq_Iitaka-orb_ind-hyp} and \eqref{eq_Iitaka-orb_quotient} with the easy inequality \cite[Theorem 5.11, pp.~59-60]{Uen75} (applied to $\pi:T\to A$) we obtain:
\begin{align*}
\kappa(X,\Delta+X) &\geqslant \kappa(F',K_{F'}+\Delta_{F'})+\dim A\geqslant \kappa(F,K_F+\Delta_F)+\kappa(S,(B_{f\!,\Delta})|_S)+\dim A \\
&\geqslant \kappa(F,K_F+\Delta_F)+\kappa(T,B_{f\!,\Delta}).
\end{align*}
\end{proof}

\begin{appendices}
\section{Proof of the \texorpdfstring{\hyperref[lemme_neg]{Negativity Lemma}}{text}}
\label{appendix_sec_neg-lemma}
In this appendix we are engaged to prove the {\hyperref[lemme_neg]{Negativity Lemma}} in {\hyperref[ss_preliminary_toolkit]{\S \ref*{ss_preliminary_toolkit}}}. Let us recall the statement: $h: Z\to Y$ being a proper bimeromorphic morphism between normal complex varieties and $B$ being a Cartier divisor on $Z$ such that $-B$ is $h$-nef, we will prove that $B$ is effective if and only if $h_\ast B$ is effective. First notice that if $B$ is effective, then $h_\ast B$ is effective; hence it remains to show that $h_\ast B$ is effective $\Rightarrow$ $B$ is effective. To this end we proceed in three steps:
\subparagraph{(A) Reduction to the case where $h$ is a sequence of blow-ups with smooth centres \\}
\label{lemme-neg_demo-A}
For any proper bimeromorphic morphism $f:Z'\to Z$, $B$ is effective $\Leftrightarrow$ $f^\ast B$ is effective; moreover, if we note $h'=h\circ f$\,, then $h'_\ast f^\ast B=h_\ast B$ and $-f^\ast B$ is $h'$-nef. This observation gives us the flexibility to replace $Z$ with a higher bimeromorphic model. In particular, by Chow's Lemma (\cite[Corollary 2]{Hir75}) we can suppose that $h$ is projective. In addition, by Hironaka's construction in \cite{Hir75} we see that $h$ is in fact the blow-up of an analytic subspace (a coherent ideal) of $X$ (c.f. \cite[Definition 4.1]{Hir75}); hence by Hironaka's resolution of singularities, we can take a principalization $h'$ of this ideal, which is constructed by a sequence of blow-ups with smooth centres, by the universal property of blow-ups, $h'$ dominates $h$. C.f. also \cite[Lemma 4.1]{BJ17}. Now up to replacing $h'$ with $h$, we can assume that $h$ is a locally  finite (over $Y$) sequence of blow-ups with smooth centres; moreover the problem being local over $Y$, one can further assume that $h$ is a finite sequence. In particular, (e.g. by an induction on the number of blow-ups contained in $h$)
there exists an effective Cartier divisor $h$-exceptional divisor $A$ such that $-A$ is $h$-ample.

\subparagraph{(B) Reduction to the case where $-B$ is $h$-ample by an approximation argument \\}
\label{lemme-neg_demo-B}
In this step we use an approximation argument to reduce to the case where $-B$ is $h$-ample. To this end, assume that the lemma is true for $h$-anti-ample divisors. By {\hyperref[lemme-neg_demo-A]{Step (A)}}, one gets an $h$-exceptional divisor $A$ such that $-A$ is $h$-ample. Since $h_\ast A=0$, our assumption implies that $A$ is effective. For every $m>0$, the Cartier divisor$-mB-A$ is $h$-ample; in addition, $h_\ast(mB+A)=mh_\ast B\geqslant 0$, hence by our assumption, $mB+A$ is effective. Letting $m$ tend to $+\infty$ we obtain that $B$ is effective\footnote{In fact, the coefficients of the $\QQ$-divisor $$B-\frac{1}{m}A$$ are all $\geqslant 0$ for every $m>0$, let $m\to+\infty$ we see that the coefficients of $B$ are $\geqslant 0$, i.e. $B$ is effective.}. 

\subparagraph{(C) The case where $-B$ is $h$-ample \\}
\label{lemme-neg_demo-C}
By the reduction procedures {\hyperref[lemme-neg_demo-A]{(A)}} and {\hyperref[lemme-neg_demo-B]{(B)}}, we can suppose that $h$ is projective and that $B$ is a Cartier divisor on $Z$ such that $-B$ is $h$-ample. Since $-B$ is $h$-ample, then for any $m>>0$, the Cartier divisor $-mB$ is relatively globally generated, i.e. we have an surjection
\[
h^\ast h_\ast\scrO_Z(-mB)\twoheadrightarrow\scrO_Z(-mB)\,.
\]
In particular, $\scrO_Z(-mB)=h\inv\!\fraka_m\cdot\scrO_Y$ where $\fraka_m=h_\ast\scrO_Z(-mB)$ fractional ideal on $Y$ (i.e. a torsion free subsheaf of rank $1$ of $\scrM_Y$ the sheaf of germs of meromorphic functions on $Y$) since $h$ is bimeromorphic. It remains to see that $\fraka_m$ is an authentic ideal. To this end it suffices to consider the inclusion (by hypothesis $h_\ast B$ is effective)
\[
\fraka_m=h_\ast\scrO_Z(-mB)\subseteq\scrO_Y(-mh_\ast B)\subseteq\scrO_Y\,,
\]
where the inclusion $h_\ast\scrO_Z(-mB)\subseteq\scrO_Y(-mh_\ast B)$ above results from {\hyperref[lemme_incl-div-birat]{Lemma \ref*{lemme_incl-div-birat}}}.
\section{Proof of  \texorpdfstring{\hyperref[prop_div-exc-non-rel-psef]{Proposition \ref* {prop_div-exc-non-rel-psef}}}{text}}
\label{appendix_sec_ref-hull}
In this appendix, we give the detailed proof of {\hyperref[prop_div-exc-non-rel-psef]{Proposition \ref*{prop_div-exc-non-rel-psef}}} which serves to complete the proof of  {\hyperref[thm_env-ref]{Theorem \ref*{thm_env-ref}}}\,. Let $X\to S$ be a surjective morphism between complex varieties with $X$ smooth and $S$ normal, we will show that there is an effective $\pi$-exceptional divisor $E$ such that for any $\pi$-exceptional prime divisor $\Gamma$, $E|_{\Gamma}$ is not $\pi|_\Gamma$-pseudoeffective.

The starting point of the proof is the following observation: if $\pi$ is flat, then $\pi_\ast L$ is always reflexive. Consider thus a flattening of $\pi$ (c.f. \cite{Hir75}, or for the algebraic case, \cite[\S 4.1, Theorem 1, p.~26]{Ray72}): let $\nu: S'\to S$ be a projective bimeromorphic morphism (a sequence of blow-ups with smooth centres) which flattens $\pi$ and let $X'$ be the normalization of the main component of $X\underset{\scriptscriptstyle S}{\times}S'$ equipped with morphisms $X'\xrightarrow{\mu}X$ and $X'\xrightarrow{\phi}S'$ ($\mu$ is projective and $\phi$ equidimensional).
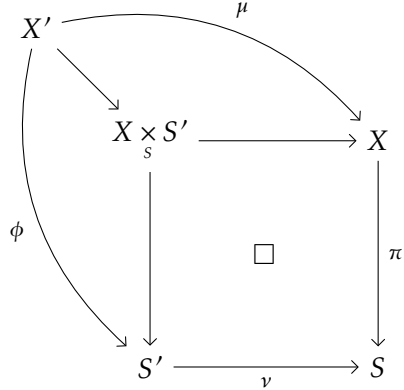
\begin{center}
\begin{tikzpicture}[scale=3.0]
\node (A) at (0,0) {$S$};
\node (B) at (0,1) {$X$};
\node (A') at (-1,0) {$S'$};
\node (B') at (-1,1) {$X\underset{\scriptscriptstyle S}{\times}S'$};
\node (C) at (-1.5,1.5) {$X'$};
\node (S) at (-0.5,0.5) {$\square$};
\path[->,>=angle 90]
(B') edge (A')
(B') edge (B)
(C) edge (B');
\path[->, font=\scriptsize,>=angle 90]
(B) edge node[right]{$\pi$} (A)
(A') edge node[below]{$\nu$} (A)
(C) edge[bend left] node[above right]{$\mu$} (B)
(C) edge[bend right] node[below left]{$\phi$} (A');
\end{tikzpicture}
\end{center}
By the construction of $\nu$, there is a  $\pi$-exceptional effective (Cartier) divisor $\Delta$ such that $-\Delta$ is $\nu$-ample. Consider the divisor $E:=\mu_\ast(\phi^\ast\Delta)$. Then $E$ is effective since $\Delta$ is effective; $E$ is Cartier since $X$ is smooth. Moreover, $-\Delta$ is $\nu$-ample, hence $-\phi^\ast\Delta$ is $\mu$-nef: in fact, let $C$ be a curve contracted by $\mu$, then $\phi_\ast C$ (which is, by definition, a curve on $S'$ if $C$ is not contracted by $\phi$ or is equal to $0$ otherwise) is contracted by $\nu$ since $\pi\circ\mu=\nu\circ\phi$, hence by the projection formula we get
\[
(-\phi^\ast\Delta\cdot C)=(-\Delta\cdot\phi_\ast C)\geqslant 0,
\]
$\mu$ being projective, this implies that $-\phi^\ast\Delta$ is $\mu$-nef; then so is $\mu^\ast E-\phi^\ast\Delta$. Now since 
\[
\mu_\ast(\mu^\ast E-\phi^\ast\Delta)=E-E=0,
\]
then we have $\mu^\ast E-\phi^\ast\Delta\leqslant 0$ by the {\hyperref[lemme_neg]{Negativity Lemma}}. 

Assume by contradiction that there exists a $\pi$-exceptional prime divisor $\Gamma$ such that $E|_{\Gamma}$ is $\pi|_{\Gamma}$-pseudoeffective and denote 
\[
\Gamma':=\text{ the strict transformation of }\Gamma\text{ by }\mu\inv.
\]

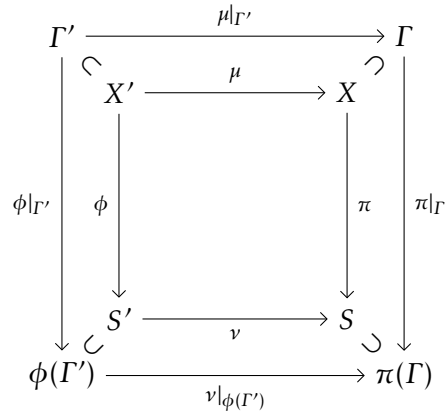
\begin{center}
\begin{tikzpicture}[scale=3.0]
\node (A) at (0,0) {$S$};
\node (B) at (0,1) {$X$};
\node (A') at (-1,0) {$S'$};
\node (B') at (-1,1) {$X'$};
\node (C) at (0.25,1.25) {$\Gamma$};
\node (C') at (-1.25,1.25) {$\Gamma'$};
\node (D) at (0.25,-0.25) {$\pi(\Gamma)$};
\node (D') at (-1.25,-0.25) {$\phi(\Gamma')$};
\path[->, font=\scriptsize,>=angle 90]
(B) edge node[right]{$\pi$} (A)
(A') edge node[below]{$\nu$} (A)
(B') edge node[above]{$\mu$} (B)
(B') edge node[left]{$\phi$} (A')
(C) edge node[right]{$\pi|_{\Gamma}$} (D)
(C') edge node[left]{$\phi|_{\Gamma'}$} (D')
(C') edge node[above]{$\mu|_{\Gamma'}$} (C)
(D') edge node[below]{$\nu|_{\phi(\Gamma')}$} (D);
\path[-stealth, auto]
(C) edge[draw=none] node [sloped, auto=false, allow upside down]{$\subset$} (B)
(C') edge[draw=none] node [sloped, auto=false, allow upside down]{$\subset$} (B')
(D) edge[draw=none] node [sloped, auto=false, allow upside down]{$\subset$} (A)
(D') edge[draw=none] node [sloped, auto=false, allow upside down]{$\subset$} (A');
\end{tikzpicture}
\end{center}

Then $\mu^\ast E|_{\Gamma'}$ is $(\pi\circ\mu)|_{\Gamma'}$-pseudoeffective, hence $\phi^\ast\Delta|_{\Gamma'}$ is $(\nu\circ\phi)|_{\Gamma'}$-pseudoeffective since $\mu^\ast E\leqslant \phi^\ast\Delta$. On the other hand,    
by our construction $-\Delta$ is $\nu$-ample, then $-\Delta|_{\phi(\Gamma')}$ is $\nu|_{\phi(\Gamma')}$-ample, and thus 
\[
-\phi^\ast\Delta|_{(\Gamma')}=(\phi|_{(\Gamma')})^\ast(-\Delta|_{\phi(\Gamma')})
\]
is $(\nu\circ\phi)|_{\Gamma'}$-nef. Therefore $-\phi^\ast\Delta|_{(\Gamma')}$ is $(\nu\circ\phi)|_{\Gamma'}$-numerically trivial, which implies that $-\Delta|_{\phi(\Gamma')}$ is $\nu|_{\phi(\Gamma')}$-numerically trivial. But $-\Delta|_{\phi(\Gamma')}$ is $\nu|_{\phi(\Gamma')}$-ample, this cannot happen unless $\nu|_{\phi(\Gamma')}:\phi(\Gamma')\to\pi(\Gamma)$ is finite. We will show in the sequel that $\nu|_{\phi(\Gamma')}$ is never finite:

Since $\phi$ is the composition of a finite morphism (normalization) followed by a flat morphism, $\phi$ is equidimensional; in particular, $\phi(\Gamma')$ is Weil divisor on $S$. Moreover, $\nu(\phi(\Gamma'))=\pi\circ\mu(\Gamma')=\pi(\Gamma)$ is of codimension $\geqslant 2$, hence $\phi(\Gamma')$ is $\nu$-exceptional;
in particular, the general fibre of the morphism $\nu|_{\phi(\Gamma')}:\phi(\Gamma')\to\pi(\Gamma)$ is of dimension $\geqslant 1$\,. Thus we prove the proposition.
\end{appendices}

\bibliography{Iitaka}

\begin{thebibliography}{{Wan}16}

\bibitem[{Art}86]{Art86}
Michael {Artin}.
\newblock {Néron Models}.
\newblock In Gary {Cornell} and Joseph~Hillel {Silverman}, editors, {\em
  {Arithmetic Geometry}}, pages 213--230, New York, NY, 1986. Springer-Verlag.

\bibitem[AS60]{AS60}
Aldo Andreotti and Wilhelm Stoll.
\newblock {Extension of Holomorphic Maps}.
\newblock {\em Annals of Mathematics}, 72(2):312--349, 1960.

\bibitem[BdFF12]{BdFF12}
Sébastien Boucksom, Tommaso de~Fernex, and Charles Favre.
\newblock {The Volume of an Isolated Singularity}.
\newblock {\em Duke Mathematical Journal}, 161(8):1455--1520, 2012.

\bibitem[BDPP13]{BDPP13}
Sébastien Boucksom, Jean-Pierre Demailly, Miahi {P\u aun}, and Thomas
  Peternell.
\newblock {The Pseudo-effective Cone of a Compact Kähler Manifold and
  Varieties of Negative Kodaira Dimension}.
\newblock {\em Journal of Algebraic Geometry}, 22(2):201--248, 2013.

\bibitem[Ber09]{Ber09}
Bo~Berndtsson.
\newblock {Curvature of Vector Bundles Associated to Holomorphic Fibrations}.
\newblock {\em Annals of Mathematics}, 169(2):531--560, 2009.

\bibitem[BJ17]{BJ17}
Sébastien Boucksom and Mattias Jonsson.
\newblock {Tropical and non-Archimedean Limits of Degenerating Families of
  Volume Forms}.
\newblock {\em Journal de l'\'Ecole polytechnique - Mathématiques}, 4:87--139,
  2017.

\bibitem[BL04]{BL04}
Christina {Birkenhake} and Herbert {Lange}.
\newblock {\em {Complex Abelian Varieties}}, volume 302 of {\em Grundlehren der
  mathematischen Wissenschften}.
\newblock Springer-Verlag, Berlin Heidelberg, 2004.

\bibitem[{Bou}04]{Bou04}
Sébastien {Boucksom}.
\newblock {Divisorial Zariski Decomposition on Compact Complex Manifolds}.
\newblock {\em Annales scientifiques de l'\'Ecole normale supérieure},
  37(4):45--67, 2004.

\bibitem[{Bou}16]{Bou16}
Sébastien {Boucksom}.
\newblock {Singularities of Plurisubharmonic Functions and Multiplier Ideals}.
\newblock Notes of Course
  \urlstyle{rm}\url{http://sebastien.boucksom.perso.math.cnrs.fr/notes/L2.pdf},
  2016.

\bibitem[BP08]{BP08}
Bo~Berndtsson and Mihai P{\u{a}}un.
\newblock {Bergman Kernels and the Pseudoeffectivity of Relative Canonical
  Bundles}.
\newblock {\em Duke Mathematical Journal}, 145(2):341--378, 2008.

\bibitem[BP10]{BP10}
Bo~Berndtsson and Mihai P{\u{a}}un.
\newblock {Bergman Kernels and the Subadjunction}.
\newblock Preprint\urlstyle{rm} \url{https://arxiv.org/abs/1002.4145}, 2010.

\bibitem[BS76]{BS76}
Constantin {Bănică} and Octavian {Stănăşilă}.
\newblock {\em {Algebraic Methods in the Global Theory of Complex Spaces}}.
\newblock John Wiley \& Sons, London, New York, Sydney, Toronto, 1976.
\newblock revised English version of {\it Metode algebrice în teoria globală
  a spaţiilor complexe}, Etitura academiei, Bucureşti, 1974.

\bibitem[Bud09]{Bud09}
Nero Budur.
\newblock {Unitary Local Systems, Multiplier Ideals, and Polynomial Periodicity
  of Hodge Numbers}.
\newblock {\em Advances in Mathematics}, 221(1):217--250, 2009.

\bibitem[BW15]{BW15}
Nero Budur and Botong Wang.
\newblock {Cohomology Jump Loci of Quasi-projective Varieties}.
\newblock {\em Annales scientifiques de l'\'Ecole normale supérieure},
  48(1):227--236, 2015.

\bibitem[BW17]{BW17}
Nero Budur and Botong Wang.
\newblock {Cohomology Jump Loci of Quasi-compact Kähler Manifolds}.
\newblock Preprint \url{https://arxiv.org/abs/1702.02186}, 2017.

\bibitem[Cam04]{Cam04}
Frédéric Campana.
\newblock {Orbifolds, Special Varieties and Classification Theory}.
\newblock {\em Annales de l'Institut de Fourier}, 54(3):499--665, 2004.

\bibitem[Cam09]{Cam09}
Frédéric Campana.
\newblock {Orbifoldes géométriques spéciales et classification biméromorphe
  des variétés kählériennes compactes}.
\newblock Prépublication IECN
  \url{https://hal.archives-ouvertes.fr/hal-00356763}, 2009.

\bibitem[Cao15]{Cao15b}
Junyan Cao.
\newblock {On the Approximation of Kähler Manifolds by Algebraic Varieties}.
\newblock {\em Mathmeatische Annalen}, 363(1-2):393--422, 2015.

\bibitem[{Cao}17]{Cao17}
Junyan {Cao}.
\newblock {Ohsawa-Takegoshi Extension Theorem for Compact Kähler Manifolds and
  Applications}.
\newblock In Daniele {Angella}, Costantino {Medori}, and Adriano {Tomassini},
  editors, {\em {Complex and Symplectic Geometry}}, volume~21 of {\em Springer
  INdAM series}, pages 19--38, New York, 2017. Springer International
  Publishing.

\bibitem[Car61]{Car61}
Henri Cartan.
\newblock {\em {Théorie élémentaire des fonctions analytiques d'une ou
  plusieurs variables complexes}}.
\newblock Herman, éditeur des sciences et des arts, Paris, 1961.

\bibitem[CH17]{CH17b}
Junyan Cao and Andreas H{\"{o}}ring.
\newblock {A Decomposition Theorem for Projective Manifolds with Nef
  Anticanonical Bundle}.
\newblock Preprint \urlstyle{rm}\url{https://arxiv.org/abs/1706.08814}, 2017.
\newblock To appear in {\it Journal of Algebraic Geometry}.

\bibitem[CHP16]{CHP16}
Frédéric Campana, Andreas H{\"{o}}ring, and Thomas Peternell.
\newblock {Abundance for Kähler Threefolds}.
\newblock {\em Annales de l'École normale supérieure}, 49(4):971--1025, 2016.

\bibitem[CKP12]{CKP12}
Frédéric Campana, Vincent Koziarz, and Mihai P{\u{a}}un.
\newblock Numerical character of the effectivity of adjoint line bundles.
\newblock {\em Annales de l'Institut Fourier}, 62(1):107--119, 2012.

\bibitem[CP11]{CP11}
Frédéric Campana and Thomas Peternell.
\newblock {Geometric Stability of the Cotangent Bundle and the Universal Cover
  of a Projective Manifold}.
\newblock {\em Bulletin de la société mathématique de France},
  139(1):41--74, 2011.

\bibitem[CP17]{CP17}
Junyan Cao and Mihai P{\u{a}}un.
\newblock {Kodaira Dimension of Algebraic Fibre Space over Abelian Varieties}.
\newblock {\em Inventiones mathematicae}, 207(1):345--387, 2017.

\bibitem[{Deb}99]{Deb99}
Olivier {Debarre}.
\newblock {\em {Tores et variétés abéliennes complexes}}, volume~6 of {\em
  Cours spécialisés}.
\newblock Société mathématique de France, EDP Sciences, Paris, 1999.

\bibitem[{Deb}01]{Deb01}
Olivier {Debarre}.
\newblock {\em {Higher-Dimensional Algebraic Geometry}}.
\newblock Universitext. Springer-Verlag, New York, BY, 2001.

\bibitem[{Dem}10]{Dem10}
Jean-Pierre {Demailly}.
\newblock {\em {Analytic Methods in Algebraic Geometry}}, volume~1 of {\em
  Surveys of Modern Mathematics}.
\newblock Higher Education Press; International Press, Beijing; Sommerville,
  2010.

\bibitem[Dem12]{agbook}
Jean-Pierre Demailly.
\newblock {\it Complex Analytic and Differential Geometry}.
\newblock OpenContent Book
  \url{https://www-fourier.ujf-grenoble.fr/~demailly/manuscripts/agbook.pdf},
  2012.

\bibitem[Dem15]{Dem15}
Jean-Pierre Demailly.
\newblock {Extension of Holomorphic Functions Defined on non Reduced Analytic
  Subvarieties}.
\newblock Preprint \url{https://arxiv.org/abs/1510.05230}, 2015.

\bibitem[{Den}17]{Den17}
Ya~{Deng}.
\newblock {Applications of the Ohsawa-Takegoshi Extension Theorem to Direct
  Image Problems}.
\newblock Preprint \url{https://arxiv.org/abs/1703.07279}, 2017.

\bibitem[DWZZ18]{DWZZ18}
Fusheng Deng, Zhiwei Wang, Liyou Zhang, and Xiangyu Zhou.
\newblock {New Characterizations of Plurisubharmonic Functions and Positivity
  of Direct Image Sheaves}.
\newblock {Preprint \url{https://arxiv.org/abs/1809.10371}}, 2018.

\bibitem[Esn81]{Esn81}
Hélène Esnault.
\newblock {Classification des variétés de dimension $3$ et plus}.
\newblock In {\em Séminaire Bourbaki n°23 (1981)}, 1981.
\newblock exposé 568.

\bibitem[Fuj78]{Fuj78}
Takao Fujita.
\newblock {On Kähler Fibre Spaces over Curves}.
\newblock {\em Journal of the Mathematical Society of Japan}, 30(4):779--794,
  1978.

\bibitem[Fuj17]{Fuj17}
Osamu Fujino.
\newblock {Notes on the Weak Positivity Theorem}.
\newblock In Kayo Masuda, Takashi Kishimoto, Hideo Kojima, Masayoshi Miyanishi,
  and Mikhail Zaidenberg, editors, {\em {Algebraic Varieties and Automorphism
  Groups, July 7-11, 2014}}, volume~75 of {\em Advances Studies in Pure
  Mathematics}, pages 73--118, Tokyo, Japan, 2017. Mathematical Society of
  Japan.

\bibitem[{Ful}84]{Ful84}
William {Fulton}.
\newblock {\em {Intersection Theory}}, volume~2 of {\em Ergebnisse der
  Mathematik und ihrer Grenzgebiete 3. Folge}.
\newblock Springer-Verlag, Berlin Heidelberg, 1984.

\bibitem[GG18]{GG18}
Mark Green and Phillip Griffiths.
\newblock {Positivity of Vector Bundles and Hodge Theory}.
\newblock Preprint \url{https://arxiv.org/abs/1803.07405}, 2018.

\bibitem[GL87]{GL87}
Mark Green and Robert Lazarsfeld.
\newblock {Deformation Theory, Generic Vanishing Theorems, and some Conjectures
  of Enriques, Cantanese and Beauville}.
\newblock {\em Inventiones Mathematicae}, 90(2):389--407, 1987.

\bibitem[GL91]{GL91}
Mark Green and Robert Lazarsfeld.
\newblock {Higher Obstructions to Deforming Cohomology Groups of Line Bundles}.
\newblock {\em Journal of the American Mathematical Society}, 4(1):87--103,
  1991.

\bibitem[GZ15]{GZ15}
Qi'an Guan and Xiangyu Zhou.
\newblock {A Solution of an $L^2$ Extension Problem with an Optimal Estimate
  and Applications}.
\newblock {\em Annals of Mathematics}, 181(3):1139--1208, 2015.

\bibitem[{Hir}75]{Hir75}
Heisuke {Hironaka}.
\newblock {Flattening Theorem in Complex Analytic Geometry}.
\newblock {\em American Journal of Mathematics}, 97(2):503--547, 1975.

\bibitem[H{\"{o}}r10]{Hor10}
Andreas H{\"{o}}ring.
\newblock {Positivity of Direct Image Sheaves - A Geometric Point of View}.
\newblock {\em L'enseignement mathématique}, 56(1/2):87--142, 2010.

\bibitem[HP16]{HP16}
Andreas H{\"{o}}ring and Thomas Peternell.
\newblock {Minimal Models for Kähler Threefolds}.
\newblock {\em Inventiones Mathematicae}, 203(1):217--264, 2016.

\bibitem[HPS18]{HPS18}
Christopher {Hacon}, Mihnea {Popa}, and Christian {Schnell}.
\newblock {Algebraic Fibre Spaces over Abelian Varieties: Around a Recent
  Theorem by Cao and P\u aun}.
\newblock In Nero {Budur}, Tommaso {de Fernex}, Roi {Docampo}, and Kevin
  {Tucker}, editors, {\em Local and Global Methods in Algebraic Geometry},
  volume 712 of {\em Contemporary Mathematics}, pages 143--195, Providence, RI,
  2018. American Mathematical Society.

\bibitem[{Kaw}81]{Kaw81}
Yujiro {Kawamata}.
\newblock Characterization of {Abelian} {Varieties}.
\newblock {\em Compositio Mathematica}, 43(2):253--276, 1981.

\bibitem[{Kaw}82]{Kaw82}
Yujiro {Kawamata}.
\newblock {Kodaira Dimension of Algebraic Fibre Spaces over Curves}.
\newblock {\em Inventiones mathematicae}, 66(1):57--71, 1982.

\bibitem[Kaw85]{Kaw85}
Yujiro Kawamata.
\newblock {Minimal Models and the Kodaira Dimension of Algebraic Fibre Spaces}.
\newblock {\em Journal für die reine und angewandte mathematik},
  1985(363):1--46, 1985.

\bibitem[Kaw13]{Kaw13}
Yujiro Kawamata.
\newblock {On the Abundance Theorem in the Case of Numerical Kodaira Dimension
  Zero}.
\newblock {\em American Journal of Mathematics}, 135(1):115--124, 2013.

\bibitem[KKMS73]{KKMS73}
George~Rushing {Kumpf}, Finn~Faye {Knudsen}, David~Bryant {Mumford}, and
  Bernard {Saint-Donat}.
\newblock {\em {Toroidal Embeddings}}, volume 339 of {\em Lecture Notes in
  Mathematics}.
\newblock Springer-Verlag, Berlin Heidelberg New York, 1973.

\bibitem[KM98]{KM98}
János {Kollár} and Shigefumi {Mori}.
\newblock {\em {Birational Geometry of Algebraic Varieties}}, volume 134 of
  {\em Cambridge Tracts in Mathematics}.
\newblock Cambridge University Press, Cambridge, 1998.

\bibitem[KMM87]{KMM87}
Yujiro {Kawamata}, Katsumi {Matsuda}, and Kenji {Matsuki}.
\newblock {Introduction to the Minimal Model Problem}.
\newblock In Tadao {Oda}, editor, {\em Algebraic Geometry, Sendai, 1985},
  volume~10 of {\em Advanced Studies in Pure Mathematics}, pages 283--360,
  Tokyo, 1987. Mathematical Society of Japan.

\bibitem[{Kob}87]{Kob87}
Shoshichi {Kobayashi}.
\newblock {\em {Differential Geometry of Complex Vector Bundles}}.
\newblock Princeton Legacy Library. Princeton University Press, Princeton, NJ,
  1987.

\bibitem[{Kol}97]{Kol97}
J\'anos {Koll\'ar}.
\newblock {Singularities of Pairs}.
\newblock In {J\'anos Koll\'ar and Robert Lazarsfeld and David {Robert}
  Morrison}, editor, {\em Algebraic Geometry Santa Cruz 1995}, volume~62 of
  {\em Proceedings of Symposia in Pure Mathematics}, pages 221--287,
  Providence, RI, 1997. American Mathematical Society.

\bibitem[KV80]{KV80}
Yujiro {Kawamata} and Eckart {Viehweg}.
\newblock {On a Characterization of an Abelian Variety in the Classification
  Theory of Algebraic Varieties}.
\newblock {\em Compositio Mathematica}, 41(3):355--359, 1980.

\bibitem[{Laz}04]{Laz04}
Robert {Lazarsfeld}.
\newblock {\em {Positivity in Algebraic Geometry I \& II}}, volume 48 \& 49 of
  {\em Ergebnisse der Mathematik und ihrer Grenzgebiete. 3. Folge}.
\newblock Springer-Verlag, Berlin Heidelberg, 2004.

\bibitem[{Les}16]{Les16}
John {Lesieutre}.
\newblock {A Pathology of Asymptotic Multiplicity in the Relative Setting}.
\newblock {\em Mathematical Research Letters}, 23(5):1433--1451, 2016.

\bibitem[Mat70]{Mat70}
Hideyuki Matsumura.
\newblock {\em {Commutative Algebra}}.
\newblock Mathematics Lecture Notes Series. W.A.Benjamin, Inc., New York, NY,
  1970.

\bibitem[Mat89]{Mat89}
Hideyuki Matsumura.
\newblock {\em {Commutative Ring Theory}}, volume~8 of {\em Cambridge Studies
  in Advanced Mathematics}.
\newblock Cambridge University Press, Cambridge, UK, 1989.

\bibitem[{Mil}86]{Mil86}
James~Stuart {Milne}.
\newblock {Abelian Varieties}.
\newblock In Gary {Cornell} and Joseph~Hillel {Silverman}, editors, {\em
  {Arithmetic Geometry}}, pages 103--150, New York, NY, 1986. Springer-Verlag.

\bibitem[{Nak}04]{Nak04}
Noboru {Nakayama}.
\newblock {\em {Zariski-decomposition and Abundance}}.
\newblock MSJ Memoir. Mathematical Society of Japan, Tokyo, 2004.

\bibitem[P{\u{a}}u16]{Pau16}
Mihai P{\u{a}}un.
\newblock {Singular Hermitian Metrics and Positivity of Direct Images of
  Pluricanonical Bundles}.
\newblock Survey \url{https://arxiv.org/abs/1606.00174}, 2016.

\bibitem[PT18]{PT18}
Mihai P{\u{a}}un and Shigeharu Takayama.
\newblock {Positivity of Twisted Relative Pluricanonical Bundles and Their
  Direct Images}.
\newblock {\em Journal of Algebraic Geometry}, 27(2):211--272, 2018.

\bibitem[Rau15]{Rau15}
Hossein Raufi.
\newblock {Singular Hermitian Metrics on Holomorphic Vector Bundles}.
\newblock {\em Arkiv für Mathematik}, 53(2):359--382, 2015.

\bibitem[{Ray}72]{Ray72}
Michel {Raynuad}.
\newblock {Flat Modules in Algebraic Geometry}.
\newblock {\em Compositio Mathematica}, 24(1):11--31, 1972.

\bibitem[RR70]{RR70}
Jean-Pierre Ramis and Gabriel Ruget.
\newblock {Complexe dualisant et théorèmes de dualité en géométrie
  analytique complexe}.
\newblock {\em Publication mathématique de IH\'ES}, 38:77--91, 1970.

\bibitem[Sim93]{Sim93}
Carlos Simpson.
\newblock {Subspaces of Moduli Spaces of Rank one Local Systems}.
\newblock {\em Annales scientifiques de l'\'Ecole normale supérieure},
  26(3):361--401, 1993.

\bibitem[{Uen}75]{Uen75}
Kenji {Ueno}.
\newblock {\em {Classification Theory of Algebraic Varieties and Compact
  Complex Spaces}}, volume 439 of {\em Lecture Notes in Mathematics}.
\newblock Springer-Verlag, Berlin Heidelberg, 1975.

\bibitem[{Vie}77]{Vie77}
Eckart {Viehweg}.
\newblock {Rational Singularities of Higher Dimensional Schemes}.
\newblock {\em Proceedings of The American Mathematical Society}, 63(1):6--8,
  1977.

\bibitem[{Vie}83]{Vie83}
Eckart {Viehweg}.
\newblock {Weak Positivity and the Additivity of the Kodaira Dimension for
  Certain Fibre Spaces}.
\newblock In Shigeru {Iitaka}, editor, {\em Algebraic Varieties and Analytic
  Varieties}, volume~1 of {\em Advanced Studies in Pure Mathematics}, pages
  329--353, Tokyo, 1983. Mathematical Society of Japan.

\bibitem[{Vie}95]{Vie95}
Eckart {Viehweg}.
\newblock {\em {Quasi-projective Moduli for Polarized Manifolds}}, volume~30 of
  {\em Ergebnisse der Mathematik und ihrer Grenzgebiete. 3. Folge}.
\newblock Springer-Verlag, Berlin Heidelberg, 1995.

\bibitem[{Voi}02]{Voi02}
Claire {Voisin}.
\newblock {\em {Théorie de Hodge et géométrie algébrique complexe}},
  volume~10 of {\em Cours spécialisés}.
\newblock Société mathématique de France, EDP Sciences, Paris, 2002.

\bibitem[{Wan}16]{Wan16}
Botong {Wang}.
\newblock {Torsion Points on the Cohomology Jump Loci of Compact Kähler
  Manifolds}.
\newblock {\em Mathematical Research Letters}, 23(2):545--563, 2016.

\bibitem[Zuo96]{Zuo96}
Kang Zuo.
\newblock {Kodaira Dimension and Chern Hyperbolicity of the Shafarevich Maps
  for Representations of $\pi_1$ of Compact Kähler Manifolds}.
\newblock {\em Journal für die reine und angewandte Mathematik},
  472(2):139--156, 1996.

\end{thebibliography}
\end{document}